\documentclass{article}
\usepackage[english]{babel}
\usepackage{amsmath,amssymb,xcolor, hyperref, stix}

\numberwithin{equation}{section}

\usepackage{mhequ,mhenvs}

\usepackage[top=4.0cm, bottom=4.0cm, left=3.8cm, right=3.8cm]{geometry}

\def\tr{\mathrm{tr}}
\def\u{\underline{u}}
\def\pr{\mathrm{pr}}
\def\sb{\tau_{\mathrm{fin}}}
\def\M{R}

\newcommand{\RR}{\mathbf{R}}
\newcommand{\NN}{\mathbf{N}}
\newcommand{\ZZ}{\mathbf{Z}}

\newcommand{\TT}{\mathbf{T}}

\newcommand{\EE}{\mathbb{E}}

\newcommand{\PP}{\mathbb{P}}

\newcommand{\mL}{\mathcal{L}}
\newcommand{\mN}{\mathcal{N}}
\newcommand{\mO}{\mathcal{O}}
\newcommand{\mP}{\mathcal{P}}

\newcommand{\mC}{\mathcal{C}}

\newcommand{\mM}{\mathcal{M}}
\newcommand{\mT}{\mathcal{T}}

\newcommand{\mF}{\mathcal{F}}

\newcommand{\mX}{\mathcal{X}}

\newcommand{\mf}[1]{\mathfrak{#1}}

\newcommand{\mt}[1]{\texttt{#1}}

\newcommand{\ve}{\varepsilon}

\newcommand{\vt}{\vartheta}

\newcommand{\bigslant}[2]{{\raisebox{.1em}{$#1$}\left/\raisebox{-.1em}{$#2$}\right.}}
\newcommand*{\ud}{\mathrm{\,d}}
\begin{document}

\title{Quantitative instability for stochastic scalar reaction-diffusion equations}
\author{Alexandra Blessing (Neam\c tu), Tommaso Rosati}

\maketitle

\begin{abstract}

This work studies the instability of stochastic scalar reaction diffusion
equations,
driven by a multiplicative noise that is white in time and smooth in space,
near to zero, which is assumed to be a fixed point for the equation.~We prove
that if the Lyapunov exponent at
zero is positive, then the flow of non-zero solutions admits uniform
bounds on small negative moments.~The proof
builds on ideas from stochastic homogenisation.~We require suitable
corrector estimates for the solution to a Poisson
problem involving an infinite-dimensional projective process, together with
a linearisation step that hinges on quantitative parametrix-like arguments. Overall, we are able to 
construct an appropriate Lyapunov functional for the nonlinear dynamics and address some
problems left open in the literature.

\end{abstract}

\setcounter{tocdepth}{2}
\tableofcontents

\section{Introduction}

This work studies long-time properties of solutions to
stochastic scalar reaction-diffusion equations of the form
\begin{equs}[e:main]
\ud u = \Delta u \ud t + f(u) \ud t + \sigma(u) \ud W_{t} \;, \quad
u(0, x) = u_{0} (x) \geqslant 0\;, \quad \forall x \in \TT^{d}\;, \quad t \geqslant 0 \;.
\end{equs}
Here $ \TT^{d} $ is the $ d- $dimensional torus and the solution $ u \colon [0, \infty)
\times\TT^{d} \to \RR $ a scalar. Moreover,  $ f$ and $\sigma$ are $ C^{1}$  non-linearities such that $ f(0)=
\sigma(0) = 0 $, and $ W $ a noise that is white in time and
sufficiently smooth in space, see Assumption~\ref{assu:nonlinearities},
so that the constant function $ u \equiv 0 $ is a solution to the
equation. Our objective is the study of \eqref{e:main} in the case in which
$ u \equiv 0 $ is a linearly unstable fixed point, meaning that the Lyapunov
exponent $ \lambda = \lim_{t \to \infty} t^{-1} \log{\| v_{t} \|_{L^{1}}}  $
associated to the linearised equation
\begin{equs}
\ud v = \Delta v \ud t + f^{\prime} (0) v \ud t + \sigma^{\prime} (0) v \ud
W_{t} \;, \qquad v(0, x) = v_{0} (x) >  0 \;,
\end{equs}
satisfies $ \lambda > 0 $. In this case, 
solutions to \eqref{e:main} should not come close to
zero.
Indeed, our main result,
Theorem~\ref{thm:moment-estimate}, states that under the assumption $ \lambda >
0 $ we can find $ \eta, \zeta,
C_{1}, C_{2} > 0 $ such that uniformly over all initial conditions $
u_{0} \geqslant 0 $ and all times $ t \geqslant 0 $
\begin{equs}[e:nmmt]
 \EE \Big[ \Big( \min_{x \in \TT^{d}} u(t,x)
\Big)^{- \eta} \Big] < C_{1} e^{- \zeta t}\Big( \min_{x \in \TT^{d}} u(0,x)
\Big)^{- \eta}   + C_{2} \;.
\end{equs}
In other words, we prove that $ V(u) = \left( \min_{x \in
\TT^{d}} u (x) \right)^{- \eta}$ is a Lyapunov functional for
\eqref{e:main}, for sufficiently small values of the parameter $ \eta >0 $. We
call this a quantitative estimate, as it guarantees for instance that $
\PP ( \min u_{t} \leqslant \ve) $ decays at least polynomially and uniformly in
time as $ \ve \downarrow 0 $.

Our result is intuitively clear, since a positive
Lyapunov exponent should guarantee the existence of a repelling region about zero,
in the sense that a solution to \eqref{e:main} that lies in that region will
eventually escape it (although this has so far been left without proof in
the case of stochastic PDEs).
However, the time that is required for escaping this region can in principle be
arbitrarily long, potentially leading to invariant measures of \eqref{e:main},
if these exist, that are heavy-tailed close to zero. This reflects the fact
that the empirical, finite-time Lyapunov exponent may be negative for long
times even if the limit Lyapunov exponent is positive. The latter is a classical
problem in the study of dynamical systems, for example recently
addressed in the context of the stochastic Allen--Cahn equation \cite{BEN23}.

From this perspective, our contribution is to prove that the opposite is true,
and our result can be rephrased as providing
upper bounds on the exit time from such repelling region. This perspective also
explains our proof idea: our approach to
obtain \eqref{e:nmmt} builds on tools from
stochastic homogenisation that are similar to
those used to study fluctuations of the finite-time Lyapunov exponent about the
true Lyapunov exponent \cite{DunlapGuKomo, gu2023kpz}.

To understand better the relevance of our work, it is useful to fix as an
example an equation of Allen--Cahn type:
\begin{equs}[e:ac]
\ud u = \Delta u \ud t +   \frac{\alpha}{2} \,(u^{3}-u)   \ud t +
\frac{\sigma}{2}  (1- u^{2}) \ud W_{t} \;,
\end{equs}
where $ \alpha \in \RR $ and $ \sigma \geqslant 0 $ are parameters. In this
case if we choose initial data with values in $ [-1, 1] $, then solutions exist
for all times and we can study their ergodic behaviour (note that our
main result holds in a more general setting, and allows for solutions to potentially
blow up in finite time).

In absence of noise ($ \sigma=0 $) for $
\alpha < 0 $ the fixed points $ u \equiv \pm 1 $ are stable and the equation is a
standard model in the study of metastability. In presence of noise ($ \sigma >
0 $), the only invariant
probability measures of the equation are interpolations of Dirac measures at $ u
\equiv \pm1 $. The dynamical picture changes at a specific value $
\alpha_{\star}(\sigma) \geqslant 0 $, were the fixed points $ \pm1 $ switch
from being linearly stable to being unstable. This happens precisely at $ \alpha_{\star}(\sigma) = -
\overline{\lambda} (\sigma) $, where $ \overline{\lambda}(\sigma)  $ is the
(non-positive, because the noise is in It\^o form) Lyapunov exponent associated to
$ \ud v = \Delta v + \sigma v \ud W $. For example $ \alpha_{\star}
(\sigma) = \sigma^{2}/2 $  if $ W $ is a space-independent Brownian
motion, since the linear equation reduces to geometric Brownian motion for
space-independent initial data.~For results on metastability for the stochastic Allen-Cahn equation in the small noise regime we refer to~\cite{BerglundGentz}.

Our result proves that this change at the level of the linearised
equation corresponds to the creation of a new non-trivial invariant measure,
for which we can obtain appropriate tail estimates.
%Namely, as a corollary of our main result, we obtain that for $
%\alpha > \alpha_{\star}(\sigma ) $ there exists an invariant measure supported
%in the space of functions with values in the open interval $ (-1, 1)  $, for
%example on the space $ C(\TT^{d}; (-1, 1)) $, and we provide quantitative estimates
%(in terms of small negative moments of the minimum) for such invariant measure
%close to the boundary. 
Uniqueness of such non-trivial invariant measure is a well-understood consequence of order
preservation, as long as the noise is sufficiently non-degenerate
\cite{Butkovsky}. 

The strength of our result lies in the
fact that we are able to tackle the entire unstable regime $ \lambda > 0 $
(equivalently, in the
example above $ \alpha > \alpha_{\star}(\sigma)$) and control the mentioned
negative moments. In particular, this should be compared to the only known result in
this direction, which is a recent work by
Khoshnevisan, Kim, Mueller and Shiu \cite{Khoshnevisan} concerning a class of equations similar to \eqref{e:ac} in
the case of one-dimensional space-time white noise. Here the authors could
prove existence of non-trivial invariant measures in the asymptotic regime $
\alpha \gg \alpha_{\star} (\sigma)  $ without our tail estimates.
Covering the whole unstable regime was left as an open problem  \cite[Section
2.4]{Khoshnevisan}.

Our approach to prove \eqref{e:nmmt} passes through the analysis of so-called
Furstenberg--Khasminskii formulas for the Lyapunov exponent, and estimates on
the distance between the sample Lyapunov exponent and the limit Lyapunov
exponent by means of corrector estimates that are typical in stochastic
homogenisation. At its heart, this approach builds on the study of
spectral gaps and ergodic properties for the ``projective process'' $ \pi_{t} = v_{t}/ \| v_{t}
\|_{L^{1}} $ associated to the linearised equation.

In the present order-preserving case (Equation~\eqref{e:main} is scalar and
satisfies a maximum principle), such spectral gaps on the projective
space are well understood: see \cite{Sinai1991Buergers, gu2023kpz,
Rosati22Synchro} and the references therein. 
This work relies on such understanding to implicitly construct a Lyapunov functional for
the nonlinear system through the analysis of a corrector, which is obtained by solving a Poisson
problem involving the generator of $ \pi_{t} $: see Section~\ref{sec:lyap-func}.

Here our main contribution lies in proving a uniform bound on the
corrector and in particular on its Fr\'echet derivative.
To bound the Fr\'echet derivative it seems fruitful to follow an
approach that uses Hilbert's metric on the projective space, cf.\ Lemma~\ref{lem:frechet}. In particular, this bound seems different from
previous estimates on the corrector obtained by Gu and Komorowski
\cite{gu2023kpz, DunlapGuKomo}, see the detailed discussion in Remark~\ref{rem:bdd}.

The final step in our proof is a linearisation argument that allows to relate
\eqref{e:main} to its linearisation at $ u \equiv  0 $. Here we use a
solution-dependent exponential transformation to reduce the problem to that of a PDE with
random coefficients. Then we conclude with a suitable cut-off argument: as in
the previous step, this step too relies heavily on the order-preservation of
the system. In particular, to deal with situations in which the initial
condition is small but very irregular, we must employ parametrix-like bounds
that allow for potential singularities at the starting time, and we rely on
the quantitative heat kernel estimates established by Perkowski and Van Zuijlen
\cite{PerkowskiVanZuijlen23HeatKernel}. See the discussion at the beginning of
Section~\ref{sec:fully-nonlinear}.

\subsubsection*{Literature}
The analysis of instabilities in reaction-diffusion
equations such as \eqref{e:ac} is relevant in many classical models for
population dynamics and is related to questions of metastability,
coexistence and persistence \cite{EtheridgeMCF, schreiber2011persistence, hening2018coexistence}. 
In particular, we highlight that it is a challenging open problem to extend the instability or
coexistence results of Hening and Nguyen \cite{hening2018coexistence} for
\emph{systems} of SDEs, to systems of stochastic reaction-diffusion equations.
This would require for example an analysis of the projective
process in the spirit of Hairer and Rosati \cite{hairer2023spectral}, and it
seems challenging to obtain suitable corrector estimates in such setting.

From the perspective of random dynamical systems, we highlight a recent work by
Blessing, Blumenthal and Engel concerning the quantitative study of finite-time Lyapunov
exponents~\cite{BEN23} and noise-induced synchronisation for the stochastic Allen-Cahn equation similar to
\eqref{e:ac}. Even if in the log run synchronisation by noise occurs~\cite{FlandoliGessScheutzow2017SynchroByNoiseOrderPreserving}, it was
proven that on finite time scales, the Lyapunov exponents are positive with
positive probability depending on the value of the parameter $\alpha\in\RR$. These results on finite-time Lyapunov exponents were
extended beyond the order-preserving case by Blessing and
Bl\"omker in~\cite{BN23}. 
%More broadly, this work relates to the well-established literature on
%ergodicity for order-preserving systems: see for
%example \cite{FGSSynchro} and the references therein. 

As mentioned, one of the main future challenges for our work is to
overcome the use of order preservation.
Here, the interest lies not only in systems of reaction diffusion
equations~\cite{GT22, hening2018coexistence}, but in particular in
models from fluid dynamics. For example, it is an open problem to prove
non-uniqueness of invariant measures for the stochastic Navier--Stokes equations
with degenerate noise. 
Instead, the finite-dimensional case is better understood.~In the setting of the
three-dimensional Lorentz '69 model (a toy model for stochastic Navier--Stokes), Coti
Zelati and Hairer \cite{MR4244269} have proven the mentioned non-uniqueness.~In
particular, the authors employ a similar construction of a Lyapunov functional
through a corrector (see also the references therein). Similar methods were
employed also in the study of mixing, for example 
by Blumenthal, Coti Zelati and Gvalani \cite{MR4597327}, by Gess and
Yaroslavtsev \cite{gess2021stabilization}, and by Bedrossian, Blumenthal and
Punshon-Smith \cite{bedrossian2021almost} with the objective of obtaining quantitative estimates
on Lagrangian chaos (a finite-dimensional system driven by an
infinite-dimensional vector field). One of the purposes of this work is to
provide a first extension of such tools to infinite dimensions.

Finally, we highlight the link with the study of the KPZ and Burgers' equation,
which would correspond to choosing $ f =0 $ and $ \sigma(u) =u $ in dimension $
d=1 $ with $ W $ space-time white noise.
For example, Dunlap, Gu and Komorowski have used a similar approach to ours,
to obtain Gaussian fluctuations of the sample Lyapunov exponent around the true Lyapunov
exponent, linking it to the $ 1:2:3 $ KPZ fixed point scaling on large volumes
\cite{gu2023kpz, DunlapGuKomo}. The study of other kinds of fluctuations, such
as LDPs, remains open. Even our result seems challenging to extend to irregular
noise and this analysis is left for some future work.~Indeed, in the case of space-time white noise, most of the estimates for the
corrector in Section~\ref{sec:lyap-func-constr} would break down and it is unclear how to
replace them (see also Remark~\ref{rem:bdd}).
Furthermore, our analysis of the nonlinear equation builds on lower bounds for
fundamental solutions of SPDEs (see Section~\ref{sec:lbds} and
\cite{PerkowskiVanZuijlen23HeatKernel}), which cease to be integrable when the
noise is too irregular.

\subsection*{Structure of the paper}
This work is divided as follows. In Section~\ref{sec:main} we introduce the
precise setting of our paper and statement of the main result,
Theorem~\ref{thm:moment-estimate}. The proof of this theorem builds on an
analogous statement in the case of linear equations: this is the
content of Section~\ref{sec:lyap-func}, and in particular
Theorem~\ref{thm:linearised-1}, which requires a stochastic homogenisation
argument and suitable corrector estimates. Finally, the proof of
Theorem~\ref{thm:moment-estimate} requires a linearisation argument that is
provided in Section~\ref{sec:linearisation}. The last sections,
\ref{sec:invariant} and \ref{sec:lbds}, contain technical aspects of our
analysis: spectral gap estimates for a ``projective process'' and lower bounds
to fundamental solutions of stochastic PDEs, respectively.

\subsection*{Acknowledgments} TR is very grateful to Martin Hairer for
suggesting this problem and for some early discussions. The authors thank
Willem van Zuijlen for useful discussions regarding~\cite{PerkowskiVanZuijlen23HeatKernel}. AB acknowledges support by
the DFG grant 543163250.

\subsection*{Notation} Let $ \NN = \{ 0, 1,2, \dots \}$, $ \NN_{+} = \NN
\setminus \{ 0 \} $ and $ \ZZ = \NN \cup - \NN $. We denote with $
\TT^{d} $ the $ d $-dimensional torus  $ \TT^d =\RR^d /\ZZ^d$, for $ d \in
\NN_{+} $.
For any \(k\in \NN, d \in \NN_{+}\)  we write \(C^{k}(\TT^{d} )\) for the space of $k$ times
differentiable functions \(\varphi \colon \TT^{d} \to \RR\) (the derivatives
being continuous).
We write $ \| \varphi \|_{\infty} = \sup_{x \in \TT^{d}} | \varphi(x) |$ for the
uniform norm of a function in $ \varphi \in C(\TT^{d})$. Next we write $
\mathbf{P}$ for the projective space associated to the cone of strictly
positive functions:
\begin{equs}
\mathbf{P} = \left\{ \varphi \in C ( \TT^{d}; (0, \infty)) \ \text{ such that }
  \ \int_{\TT^{d}} \varphi(x) \ud x = 1 \right\} \;.
\end{equs}
For any set $ \mX $ and two functions $ \varphi, \psi \colon \mX \to \RR $ we
write $ \varphi \lesssim \psi $ if there exists a constant $ C > 0 $ such that
$ \varphi (x) \leqslant C \psi (x) $ for all $ x \in \mX $. We additionally use
the symbol $ \lesssim_{\vt} $ to highlight the fact that the constant $ C $
depends on some parameter $ \vt $, if this is the case. Throughout the article, we will work with
the right-continuous filtration $
(\mF_{t})_{t \in \RR } $, generated by (a two-sided version of)
the noise driving the solution $ (u_{t})_{t \geqslant 0} $ to an SPDE. We write
\begin{equation*}
\begin{aligned}
\EE_{\tau} [X] = \EE [ X \vert \mF_{\tau}] \;,
\end{aligned}
\end{equation*}
for the conditional expectation induced by any stopping time $ \tau $, where
$ X \in L^{1}(\PP) $ is an arbitrary random variable.

\section{Setting and main result} \label{sec:main}

We start by introducing the precise assumptions under which we study
Equation~\eqref{e:main}. Since we will consider only positive solutions to the
equation, it suffices for the nonlinearities to be defined on $ [0, \infty) $.
In particular, let us write $ C^{1}([0, \infty); \RR) $ for the space of
functions $ \varphi \colon [0, \infty) \to \RR $ that are differentiable on $ [0, \infty)
$ (considering at zero the right derivative), with continuous derivative.
\begin{assumption}\label{assu:nonlinearities} ~ 
\begin{enumerate} 
    \item \textbf{(Nonlinearity)} Consider $ f, \sigma \in C^{1}([0, \infty) ;
\RR)$ such that
$f(0)= \sigma(0) =0$ and  $\sigma^{\prime} (0) =1$.
    \item \textbf{(Noise)} Let $ W$ be a two-sided Gaussian random field with correlation
    \begin{equs} 
    \EE [\ud W(t,x) \ud W(s, y)] = \delta(t-s) \kappa (x, y)\;, \qquad \forall t, s \in
\RR \;, x,y \in \TT^{d} \;, 
    \end{equs}
for a spatial correlation kernel $ \kappa \in C^1 (\TT^d \times \TT^d) $.
\end{enumerate}
\end{assumption}
The requirement $ \sigma^{\prime} (0) =1 $ is just a matter of convenience: we
can always absorb a multiplicative factor into the correlation kernel $ \kappa
$. Furthermore, the assumption $ \kappa \in \mC^{1} $ is not necessary,
and could in principle be replaced by $ \kappa \in \mC^{\ve} $ for any $ \ve > 0 $, see 
also Remark~\ref{rem:regkappa}. 

Under Assumption~\ref{assu:nonlinearities}, Equation \eqref{e:main} admits
local mild solutions for any positive initial condition $ u_{0} \geqslant 0,
u_{0} \in L^{\infty} $.
By local mild solution we mean that there exists a potential blow-up time $
\sb(u_{0}) \in (0, \infty] $ and a process $ u \in C([0,
\sb(u_{0})) ; L^{\infty}(\TT^{d} ; [0, \infty))) $ such that 
\begin{equs}
u_{t} = P_{t} u_{0} + \int_{0}^{t} P_{t-s} [f (u_{s})] \ud s +
\int_{0}^{t} P_{t -s}[ \sigma (u_{s}) ]\ud W_{s} \;, \qquad \forall t <
\sb(u_{0}) \;.
\end{equs}
Here $(P_t)_{t\geqslant 0}$ is the heat semigroup and the integral against $ \ud W $
should be understood in the sense of Walsh \cite{Walsh1986}.
\begin{remark}\label{rem:prop-solution}
In particular, we observe the following:
\begin{itemize}
\item For any $ u_{0} \geqslant 0 $ the solution $ u_{t} \geqslant 0 $ stays
positive for all $ t < \tau_{\mathrm{fin}} $. Therefore, the nonlinearities $
f, \sigma $ need not be defined for $ u < 0 $.
\item The blow-up time $ \sb $ is necessary, since we do not impose any growth
conditions on $ f, \sigma $. In particular we fix $ \sb $ to be maximal,
meaning that
\begin{equs}
\| u_{t} \|_{\infty} \to \infty \;, \text{ as } t \uparrow \sb \;.
\end{equs}
\item In view of this, we define for later convenience
\begin{equs}
u_{t} = \infty \;, \qquad \forall t \geqslant \sb \;,
\end{equs}
and we will use the usual convention $ 1/ \infty = 0$.
\end{itemize}
\end{remark}
Since we are interested in the behaviour of \eqref{e:main} near the fixed point
$u=0$, a particular role in our analysis is played by the linearised equation
\begin{equs}[e:linearised-n]
	\ud v = \Delta v \ud t + \gamma \, v \ud t + v \ud W_{t} \;, \qquad
v(0, \cdot) = u_{0}(\cdot) \geqslant 0\;, \qquad \gamma = f^{\prime} (0) \;.
\end{equs}
Recall here that by Assumption~\ref{assu:nonlinearities}, $ \sigma^{\prime}
(0) =1 $. Later it will be convenient to work with the flow $ \Phi$ on $
C(\TT^{d}; [0, \infty)) $ associated
to \eqref{e:linearised-n}:
\begin{equation} \label{e:Phi}
\begin{aligned}
\ud \Phi_{s, t} [w] & = \Delta\Phi_{s, t} [w]  \ud t +
\gamma \Phi_{s, t} [w] \ud t +  \Phi_{s, t} [w] \ud W_{t} \;,\\
\Phi_{s, s}[w] & = w \geqslant 0  \;,
\end{aligned}
\end{equation}
for $0 \leqslant  s \leqslant t < \infty $, so that then $ v_{t} = \Phi_{0,t}
[v_{0}] $.
This linearised equation admits an almost-sure Lyapunov exponent, which will
determine the dynamical properties that we will study.
\begin{lemma}\label{lem:lyap}
Under Assumption~\ref{assu:nonlinearities}, there exists a $\lambda \in \RR$
and a set $ \widetilde{\Omega} \subseteq \Omega $ of full $ \PP $-measure such that for
every initial condition $v_0 \in C(\TT^{d};
[0, \infty)) $, with $ v_{0} \neq 0 $, the solution $ v $ to \eqref{e:linearised-n} with initial
condition $v_0$ satisfies
\begin{equs}
\lambda = \lim_{t \to \infty} \frac{1}{t}
\log(\|v_t (\omega)\|_{L^1}) \;, \qquad \forall \omega \in \widetilde{\Omega} \;.
\end{equs}
\end{lemma} 
\begin{proof}

The existence of the Lyapunov exponent for any fixed initial condition is a
consequence of the multiplicative ergodic theorem. The only noteworthy aspect
is that the Lyapunov exponent and the null-set outside of which convergence
holds do not depend on the initial condition. This is a consequence of a
formula for the Lyapunov exponent, see Lemma~\ref{lem:FK-formula}, and of the
existence of a flow of solutions to \eqref{e:linearised-n}, see for
example Section~\ref{sec:lbds}.
\end{proof}
In this setting the main result of this article is the following theorem. It
provides a quantitative estimate (in terms of negative moments of the solution)
of the instability of \eqref{e:main} at $ u \equiv 0 $ when the
Lyapunov exponent is positive. Recall that we allow our solution to blow up in
finite time, and we define $ u_{t} = \infty $ for $ t \geqslant
\tau_{\mathrm{fin}} $, in which case $ (\min_{x} u(t, x))^{- \eta} = 0$ if $
t \geqslant \sb $.

\begin{theorem}\label{thm:moment-estimate}
   Under Assumption~\ref{assu:nonlinearities}, if the Lyapunov exponent
$\lambda$ from Lemma~\ref{lem:lyap} satisfies $\lambda >0$, then there exist
finite constants $\eta, \zeta, C_{1}, C_{2} > 0$ (depending on $ \lambda $)
such that for any $ u_{0} \in C(\TT^{d} ; [0, \infty)) $
and all $ t \geqslant 0 $:
\begin{equs}
 \EE \Big[ \Big( \min_{x \in \TT^{d}} u(t,x)
\Big)^{- \eta} \Big] < C_{1} e^{- \zeta t}\Big( \min_{x \in \TT^{d}} u(0,x)
\Big)^{- \eta}   + C_{2} \;.
\end{equs}
\end{theorem}
The starting point for the proof of Theorem~\ref{thm:moment-estimate} is to
prove the theorem in the linear case $ f(u)= \gamma u, \sigma(u) = u $ and then
introduce a linearisation argument to treat the nonlinearities. Accordingly, we
start with the analysis of the linear case, where the construction of the
Lyapunov functional requires the introduction of a suitable corrector.
The proof of Theorem~\ref{thm:moment-estimate} can then be found in Section~\ref{sec:partilly-linear} in the
simpler ``partially linear'' case $ \sigma (u) = u $. This case involves less technicalities and should
serve as a warm-up for the first reader. Instead, for the fully nonlinear case
we refer to Section~\ref{sec:fully-nonlinear}.

\section{The Lyapunov functional for the linearised problem}\label{sec:lyap-func}

This section is devoted to understanding Theorem~\ref{thm:moment-estimate} in the
linear case $ \sigma (u) = u , f(u) = \gamma u $. The aim of the
following section will be to prove the upcoming theorem. Note that unlike
Theorem~\ref{thm:moment-estimate}, we only consider the $ L^{1} $ norm and not
the minimum of the function. Improving the estimate to obtain that negative
polynomials of the minimum are also a Lyapunov functional is the content of 
Proposition~\ref{prop:lin}, and the consequence of parabolic regularity
estimates. The proof of the next result requires the use of a suitable
stochastic homogenisation argument.

\begin{theorem}\label{thm:linearised-1}
Under the assumption of Theorem~\ref{thm:moment-estimate}, there exist $ C
, \eta_{0}, \zeta > 0 $ such that for any stopping
time $ \tau < \infty  $ and $ t > 0 $ and $ \mF_{\tau} $-adapted $ w \in
C(\TT^{d}) $ 
\begin{equation*}
\begin{aligned}
\EE_{\tau} \left[ \left(  \int_{\TT^{d}} \Phi_{\tau, \tau +
t}[w] (x) \ud x
\right)^{- \eta } \right]^{\frac{1}{\eta}} \leqslant C e^{- \zeta t }
\int_{\TT^{d}} w (x) \ud x  \;,
\end{aligned}
\end{equation*}
for all $ \eta \in (0, \eta_{0}) $, and where $ \Phi $ is the flow in
\eqref{e:Phi}.
\end{theorem}
The proof of this result is provided at the end of this section.
It relies on the construction of a Lyapunov functional of
the form
$$ V(w) \sim \left( \int_{\TT^{d}}  w (x) \ud x\right)^{- \eta} \;, $$ which we
obtain in Section~\ref{sec:lyap-func-constr} below. Here the symbol $ \sim $ indicates that $ V $ contains a
(bounded) corrector that is constructed implicitly through the solution of a Poisson
problem involving the generator of the projective dynamic of $ \Phi $.

\subsection{The radial Lyapunov functional} \label{sec:lyap-func-constr}

Since the proof of
Theorem~\ref{thm:linearised-1} for arbitrary $ \tau $ follows from the case $
\tau = 0 $ by the strong Markov property, we shall fix the latter setting for
the remainder of this section and write $ \Phi_{t}[w] = \Phi_{0, t} [w] $ (and similarly for all
other processes). For the present discussion it will be convenient to split the solution $ \Phi$  to
\eqref{e:Phi} into its radial and angular component.
\begin{equs}
r_{t}[w] = \int_{\TT^{d}} \Phi_{t}[w]( x)  \ud x\;, \qquad \pi_t[w](x) =
\Phi_{t}[w](x) / r_{t}[w] \;.
\end{equs}
When clear from context, we will drop the dependence on the initial
condition $ w $.
Then, let us define the first order approximation of our Lyapunov functional:
\[ V^{(0)} (r) = r^{-\eta} . \]
We would like to prove that a suitably corrected version of $V^{(0)}$ is a
Lyapunov functional for the solution $ \Phi $ to \eqref{e:Phi}. By It\^o's
formula, since $t \mapsto r_{t}$ is a semimartingale, we find:
\[ \ud \log (r_t) = \frac{\ud r_t}{r_t} - \frac{\ud \langle r
   \rangle_t}{2 r_t^2} . \]
Furthermore, $r_t$ has quadratic variation
\[ \langle r \rangle_t =  \int_0^t \int_{\TT^d \times \TT^d} \Phi_t
   (x) \Phi_t (y) \kappa (x, y) \ud x \ud y \;,\]
where $ \kappa (x, y) $ is the spatial correlation of the noise, as defined in
Assumption~\ref{assu:nonlinearities}. Hence, overall we obtain that for some continuous martingale
$t \mapsto M_{t}$:
\begin{equs} \label{e:ito-log} 
\ud \log(r_t) = \gamma \ud  t - \frac{1}{2} \left( \int_{\TT^d \times
   \TT^d} \pi_t (x) \pi_t (y) \kappa (x, y) \ud x \ud y \right)
   \ud t + \ud M_t 
   \;.
\end{equs}
Now, by Corollary~\ref{cor:syncrho}, the
projective process $( \pi_t)_{t \geqslant 0}$ converges almost surely exponentially fast
(in the projective space with Hilbert's metric, as $ t \to \infty $) to an
invariant solution $ \pi^{\infty}_{t}$. From the ergodic theorem we then obtain the following result.

\begin{lemma}\label{lem:FK-formula}
Under the assumptions of Theorem~\ref{thm:moment-estimate} we find that
\begin{equation}\label{e:fk} \lambda  = \gamma -  \frac{1}{2} \mathbb{E} \left[ \left( \int_{\TT^d
   \times \TT^d} \pi_{\infty} (x) \pi_{\infty} (y) \kappa (x, y) \ud
   x \ud y \right) \right]  \;. 
   \end{equation}
\end{lemma}
\begin{proof}
This follows from \eqref{e:ito-log}, if we prove that
\begin{equs}[e:et]
\lim_{t \to \infty} \frac{1}{t} \int_{0}^{t}  \int_{\TT^d \times
   \TT^d} \pi_s (x) \pi_s (y) \kappa (x, y) \ud x \ud y \ud s = \EE \left[ \left( \int_{\TT^d
   \times \TT^d} \pi_{\infty} (x) \pi_{\infty} (y) \kappa (x, y) \ud
   x \ud y \right) \right].
\end{equs}
Indeed, we have that almost surely the martingale $t\mapsto M_t$ from \eqref{e:ito-log} satisfies
$\lim_{t \to \infty} t^{-1} M_{t} = 0$, since we can control its quadratic
variation. In fact, note that we have the identity
\begin{equs}
M_t = \int_{0}^{t} \int_{\TT^d} \pi_s(x)  W( \ud s, \ud  x) \;, 
\end{equs}
which has quadratic variation bounded as follows (since $ \smallint_{\TT^d} \pi
(x) \ud x =1 $):
\begin{equs}
\langle M \rangle_{t} = \int_{0}^{t} \int_{(\TT^{d})^{2}} \pi_{s}(x)
\pi_{s} (y) \kappa (x, y) \ud x \ud y \ud s \leqslant \| \kappa \|_{\infty} t
\;,
\end{equs}
so that the desired result follows.
As for the proof of \eqref{e:et}, we use the ergodic theorem and the
synchronisation result in Corollary~\ref{cor:syncrho}. Consider $
(\pi_{t}^{\infty})_{t \geqslant 0} $ the invariant projective processes from
Corollary~\ref{cor:syncrho}, then we find
\begin{equs}
\frac{1}{t}\left\vert \int_{0}^{t} \int_{\TT^{d}} (\pi_{s}(x) \pi_{s}(y) -
\pi^{\infty}_{s}(x) \pi^{\infty}_{s}(y)) \kappa (x, y) \ud x \ud y \right\vert
& \leqslant \frac{ 2   \| \kappa \|_{\infty} }{t} \int_{0}^{t} \| \pi_{s} -
\pi_{s}^{\infty} \|_{L^{1}} \ud s \\
& \leqslant \frac{ 2   \| \kappa \|_{\infty} }{t} \int_{0}^{t} \exp \left(
d_{\mathbf{P}} (\pi_{s}^{\infty}, \pi_{s})\right) -1  \ud s  \;,
\end{equs}
where we have used the bound in Lemma~\ref{lem:bush} in the last line.
Therefore, since almost surely for all $ s $ sufficiently large, $
d_{\mathbf{P}} (\pi_{s}, \pi^{\infty}_{s}) \leqslant e^{- \alpha s/2} $ by the
second statement of Corollary~\ref{cor:syncrho} (where $ \alpha > 0 $ is a
deterministic constant), and since
\begin{equs}
\lim_{t \to \infty} \frac{1}{t} \int_{0}^{t} e^{- \alpha s /2}  \ud s = 0\;,
\end{equs}
it suffices to show that
\begin{equs}
\lim_{t \to \infty} \frac{1}{t}  \int_{0}^{t}  \int_{\TT^d \times
   \TT^d} \pi_s^{\infty} (x) \pi_s^{\infty} (y) \kappa (x, y) \ud x \ud y \ud s = \EE \left[ \left( \int_{\TT^d
   \times \TT^d} \pi_{\infty} (x) \pi_{\infty} (y) \kappa (x, y) \ud
   x \ud y \right) \right] \;.
\end{equs}
Since $ (\pi^{\infty}_{t})_{t \geqslant 0} $ is invariant, this now follows
from the ergodic theorem (note that $ \pi^{\infty}_{s} \overset{d}{=}
\pi_{\infty} $ for all $ s \geqslant 0 $). This completes the proof of the lemma.
\end{proof}
At this point we can proceed with constructing a Lyapunov functional through a
suitable correction of $ V^{(0)} $. To this aim, let us apply the It\^o formula
to $V^{(0)}$, so that we obtain
\begin{equs}
\ud V^{(0)} (r_t) & = - \eta V^{(0)}(r_t) \frac{\ud r_t}{r_t} + \frac{\eta
(\eta +1)}{2} V^{(0)} (r_t) \frac{\ud \langle r \rangle_t}{r_t^2} \\
& = - \eta V^{(0)}(r_t) \gamma \ud t + \frac{\eta (\eta +1)}{2} V^{(0)}  (r_t)
Q(\pi_t) \ud t  + \ud M^{(0)}_t\;,
\end{equs}
with
\begin{equs}
Q(\pi_t)= \int_{\TT^d
   \times \TT^d} \pi_t (x) \pi_t(y) \kappa (x, y) \ud x \ud y \;.
\end{equs}
We can rewrite the previous expression as follows:
\begin{equs}[e:ito_v0]
\ud V^{(0)}  (r_t) & =- \eta \lambda V^{(0)} (r_t)  \ud t - \frac{\eta}{2}
V^{(0)} (r_t) F (\pi_t) \ud t  + \frac{\eta^2}{2} V^{(0)} (r_t) Q(\pi_t) \ud t + \ud
M_t^{(0)} \;,
\end{equs}
where we have defined
\begin{equs}[e:F]
F(\pi) = \gamma - \lambda - \frac{1}{2} \int_{\TT^{d} \times \TT^d} \pi (x) \pi(y)\kappa
(x, y) \ud x \ud y\;.
\end{equs}
Since the ``error'' term $F$ is centered, in the sense that it has mean zero
with respect to the invariant measure of
$(\pi_{t})_{t \geqslant 0} $ we do not expect it to contribute over long times. 
To remove this term from our analysis, we can correct the functional $V^{(0)}$
as follows, by defining
\begin{equs}[e:G]
V(r) = V^{(0)} (r)(1+\eta G(\pi))\;, \qquad \mL G = F\;,
\end{equs}
where $\mL$ is the generator of the projective process $( \pi_t)_{t \geqslant
0}$, and we refer to Section~\ref{sec:invariant} below for the construction
and analysis of such $ G $. 

To see that we stand to gain something from this perturbation, we proceed by
applying the It\^o formula to $V$. We obtain
\begin{equs}[e:dV]
\ud V (r_t) =   (1 + \eta G(\pi_t)) \ud V^{(0)} (r_t)+ \eta
V^{(0)} (r_t) \ud G(\pi_t) + \frac{\eta}{2} \ud \langle V^{(0)} (r) , G(\pi) \rangle_t\;.
\end{equs}
Here we have presumed that $ G(\pi_{t}) $ is a semi-martingale, which is indeed
the case, as for a continuous martingale $ t \mapsto
M^{G}_{t} $ we find via Lemma~\ref{lem:semimart}
\begin{equs}[e:dG]
\ud G (\pi_t) = \mL G (\pi_t) \ud t + \ud M^G_t = F (\pi_t) \ud t + \ud M^G_t\;.
\end{equs}
In particular, we can deduce the identity
\begin{equs}
\ud \langle V^{(0)} (r) , G(\pi) \rangle_t = \ud \langle M^{(0)} , M^G \rangle_t\;.
\end{equs}
To compute this covariation we must first find a convenient representation for
the two martingales involved. While $ M^{(0)} $ is easy to compute, the
following expression for $ M^{G} $ can be derived for example from the SPDE
representation of the projective process $ (\pi_{t})_{t \geqslant 0} $ given in
\eqref{eqn:angular}, assuming that $ G $ is Fr\'echet differentiable (see Lemma~\ref{lem:semimart}):
\begin{equs}
\ud M^{(0)}_t &= - \eta V^{(0)} (r_t) \int_{\TT^d} \pi_t(x)  W( \ud t, \ud  x) \;, \\
\ud M^G_t & = \langle DG (\pi_t), \pi_t \ud W - \langle\pi_t \ud W ,1 \rangle \pi_t \rangle \;.\label{e:MG}
\end{equs}
Note that $G$ is a functional
$G \colon \mathbf{P} \rightarrow \mathbf{R}$,
and since $ \mathbf{P} $ is affine, at every point $\nu \in \mathbf{P}$, the
tangent space of $\mathbf{P}$ at $\nu$ is given by the same space
$$T_\nu = \left\{ \varphi \in C(\TT^{d}) \text{ such that } \int_{\TT^d} \varphi (x) \ud  x = 0 \right\}\;.$$
In particular, if $G$ is smooth, then its Fr\'echet derivative is defined in
any direction belonging to $T_\nu$, and therefore
\eqref{e:MG} is well-defined (once we integrate in time), since $ \pi_t \ud W - \langle\pi_t \ud W ,1
\rangle \pi_t  $ has zero space average.
As for the quadratic covariation, we now have that
\begin{equs}
    \ud \langle V^{(0)}  (r) , G(\pi) \rangle_t & = - \eta V^{(0)} (r_t) \ud [
\langle \ud W, \pi_t\rangle , \langle \ud W \pi_t, DG(\pi_{t}) - \langle \pi_t,
DG(\pi_{t}) \rangle \rangle ] \\
    & = - \eta V^{(0)}(r_t) \int_{\TT^d \times \TT^d } \pi_t (x) \kappa (x,y)
\pi_t(y) \left\{ DG (\pi_{t}) (y) - \langle \pi_t, DG (\pi_{t}) \rangle \right\} \ud x \ud y \\
    & = - \eta V^{(0)} (r_t) \langle DG (\pi_t), \pi_t \cdot \kappa * \pi_t - \langle \pi_t \cdot \kappa * \pi_t, 1 \rangle \pi_t \rangle\;,
\end{equs}
where $\kappa * \pi = \int_{\TT^d} \kappa (x, y) \pi (y) \ud y$. In particular,
we observe that the direction
\begin{equs}
\mathcal{N}(\pi) = \pi \cdot \kappa * \pi - \langle \pi \cdot \kappa * \pi, 1 \rangle \pi 
\end{equs}
lies in $T_{\pi}$.
Now, by Lemma~\ref{lem:frechet} below, we find that
\begin{equs}
|\langle D G(\pi), \mN (\pi) \rangle | & \leqslant 2 \| G
\|_{\mathrm{Lip}(d_{\mathbf{P}})}  \| \mN(\pi) / \pi
\|_{\infty} \label{e:usegrad} \\
& \leqslant 2 \| G
\|_{\mathrm{Lip}(d_{\mathbf{P}})} \| \kappa * \pi - \langle
\pi \cdot \kappa * \pi,1 \rangle  \|_{\infty} \\
& \leqslant 2  \| G
\|_{\mathrm{Lip} (d_{\mathbf{P}})} \| \kappa \|_{\infty}
(\| \pi \|_{L^{1}} + \| \pi \|_{L^{1}}^{2}) \\
& \leqslant 4\| G
\|_{\mathrm{Lip} (d_{\mathbf{P}})} \| \kappa \|_{\infty} \;.
\end{equs}
Note that in virtue of the bound of Lemma~\ref{lem:frechet} the upper bound
does not depend on $ \pi $. This is not trivial and is the consequence of a
cancellation in the term $ \mN (\pi) / \pi $, as well as of the assumption $ \kappa \in
L^{\infty} $.
Overall we have thus obtained
\begin{equation}\label{e:qv-bd}
\begin{aligned}
\ud \langle V^{(0)} (r) , G(\pi) \rangle_t \leqslant \eta V^{(0) } (r_{t}) 4\| G
\|_{\mathrm{Lip} (d_{\mathbf{P}}) } \| \kappa \|_{\infty} \;.
\end{aligned}
\end{equation}
All together, we have therefore proven the following result.
\begin{lemma}\label{lem:ub-lyap}
Let $ V $ be defined as in \eqref{e:G}, with $ G $ the corrector constructed
in Lemma~\ref{lem:local-lipschitz}. Then the
following differential inequality holds true:
\begin{equs}
\ud V (r_t) \leqslant   (1 + \eta G(\pi_t)) \ud V^{(0)} (r_t)+ \eta
V^{(0)} (r_t) \ud G(\pi_t) + 2 \eta^{2} \| G
\|_{\mathrm{Lip} (d_{\mathbf{P}}) } \| \kappa \|_{\infty}  V^{(0) } (r_{t})\ud t\;.
\end{equs}
\end{lemma}
\begin{remark}\label{rem:bdd}
In order to obtain Lemma~\ref{lem:ub-lyap}, and
therefore prove Theorem~\ref{thm:linearised-1}, we have used
the uniform bound \eqref{e:usegrad} on the Fr\'echet derivative of $ G $, which
is a consequence of the upcoming Lemma~\ref{lem:frechet}. This seems
substantially different
from the bound obtained in \cite[Corollary 5.11]{gu2023kpz}, which
estimates $ \| DG(\pi) \|_{\infty} \lesssim 1 + \| \pi \|_{\infty}  $, and
in our setting would lead to the estimate $ | \langle D G(\pi),
\mN(\pi) \rangle | \lesssim 1 + \| \pi \|_{\infty} $, which seems to be too weak
to close our argument. It is not clear to us how to obtain
a similar estimate without building on Hilbert's projective metric.

%By comparison, below we use the bound $ \sup_{\pi \in \mathbf{P}}| G(\pi) | < \infty  $,
%which is stronger than the estimate $ | G(\pi) | \lesssim 1 + \| \pi
%\|_{\infty} $ in \cite[Proposition 5.5]{gu2023kpz}. However, in this case, our
%bound seems a relatively simple consequence of the fact that when the correlation kernel of the
%noise is bounded, then the rest functional $ F $ is uniformly bounded, cf.
%Proposition~\ref{prop:bdd}.

\end{remark}
With the upper bound of the previous lemma, we are ready to prove Theorem~\ref{thm:linearised-1}.

\begin{proof}[of Theorem~\ref{thm:linearised-1}]
proof follows by applying Lemma~\ref{lem:ub-lyap}, together with
\eqref{e:ito_v0} and \eqref{e:dG}: in the latter case, the existence of the
martingale $ M^{G} $ is guaranteed by Lemma~\ref{lem:ub-lyap}. Altogether, we
obtain
\begin{equation*}
\begin{aligned}
\ud V (r_{t})  \leqslant & \left\{ - \eta \lambda V(r_t)  + \frac{\eta^2}{2} V
(r_t) Q(\pi_t) + 4 \eta^{2} V^{(0) }(r_{t}) \| G
\|_{\mathrm{Lip} (d_{\mathbf{P}})} \| \kappa \|_{\infty} \right\} \ud t
\\
& + (1+ \eta G( \pi_{t})) \ud M_t^{(0)} + V^{(0) }(r_{t}) \ud M_{t}^{G} \\
\leqslant & \left\{ - \eta \lambda V(r_t)  + \frac{\eta^2}{2} V
(r_t) \| \kappa \|_{\infty} + 4 \eta^{2} V^{(0)}(r_{t}) \| G
\|_{\mathrm{Lip} (d_{\mathbf{P}})} \| \kappa \|_{\infty} \right\} \ud t
\\
& + (1+ \eta G( \pi_{t})) \ud M_t^{(0)} + V^{(0)}(r_{t}) \ud M_{t}^{G} \;.
\end{aligned}
\end{equation*}
Now we observe that the last drift term depends on $ V^{(0)} $ and not on $ V
$. At this point, we use that $ G $ is bounded, which is proven in
Proposition~\ref{prop:bdd} below, so that for $ \eta $ such that $ 2 \eta  \| G
\|_{\infty} \leqslant  1$ we have $ (1 + \eta G) \geqslant  1/2 $ (here $ \| G \|_{\infty} =
\sup_{\pi \in \mathbf{P}} |G(\pi)| $) and hence for such $ \eta $:
\begin{equation*}
\begin{aligned}
\ud V (r_{t})  \leqslant & \left\{ - \eta \lambda   + \frac{\eta^2}{2}  \|
\kappa \|_{\infty} + 8 \eta^{2} \| G \|_{\mathrm{Lip} (d_{\mathbf{P}})}
\| \kappa \|_{\infty} \right\} V(r_t)\ud t \\
& +  (1+ \eta G( \pi_{t})) \ud M_t^{(0)} + \eta V^{(0)}(r_{t}) \ud M_{t}^{G} \;.
\end{aligned}
\end{equation*}
Finally, for $ \eta $ such that
\begin{equation}\label{e:small-eta}
\begin{aligned}
\eta & \leqslant \eta_{0} \eqdef \min \left\{ \frac{\lambda}{2} \left\{ \frac{\| \kappa \|_{\infty}}{2} + 8  \| G
\|_{\mathrm{Lip} (d_{\mathbf{P}}) } \| \kappa \|_{\infty}
\right\}^{-1} \,, \,  \frac{1}{2 \| G \|_{\infty}} \right\} \;,
\end{aligned}
\end{equation}
we obtain that
\begin{equation*}
\begin{aligned}
\ud V (r_{t})  \leqslant & - \frac{\eta \lambda}{2}  V(r_t)  \ud t +
\ud \mM_{t} \;,
\end{aligned}
\end{equation*}
where $ \mM_{t} $ is the local martingale
\begin{equation*}
\begin{aligned}
\ud \mM_{t} =  (1 + \eta G(\pi_{t})) \ud M^{(0)}_{t} + \eta V^{(0)}(r_{t}) \ud
M^{G}_{t} \;.
\end{aligned}
\end{equation*}
Now we can deduce that $ e^{\frac{\eta \lambda}{2} t} V(r_{t}) $ is a
non-negative local super-martingale. By Fatou's lemma, given an arbitrary localising
sequence $ \{ \tau_{n} \}_{n \in \NN} $ we find
\begin{equs}[e:almost-done]
\EE \left[ e^{\frac{\eta \lambda}{2} t} V(r_{t}) \right] & =\EE \left[
\lim_{n \to \infty} e^{\frac{\eta \lambda}{2} t \wedge \tau_{n} } V(r_{t \wedge
\tau_{n}}) \right]  \leqslant \lim_{n \to \infty}\EE \left[
 e^{\frac{\eta \lambda}{2} t \wedge \tau_{n} } V(r_{t \wedge
\tau_{n}}) \right] \leqslant V(r_{0}) \;,
\end{equs}
so that the Lyapunov property is proven.
In particular, Theorem~\ref{thm:linearised-1} is proven, with
\begin{equation*}
\begin{aligned}
\zeta = \frac{\lambda}{2} \qquad  \text{ and } \qquad C = (1 - \eta_{0} \| G
\|_{\infty})^{-1} (1 + \eta_{0} \| G \|_{\infty}) \leqslant 3 \;,
\end{aligned}
\end{equation*}
where to obtain the last inequality we used that
\begin{equs}
(1 + \eta \| G \|_{\infty} ) V(r) \leqslant V^{(0)} (r) \leqslant (1- \eta \| G
\|_{\infty})^{-1}  V(r) \;,
\end{equs}
by our definition of $ \eta_{0} $, to replace $ V $ by $ V^{(0)} $ in \eqref{e:almost-done}.
\end{proof}
The rest of this section is devoted to proving the differential inequality in
Lemma~\ref{lem:ub-lyap}.

\subsection{Fr\'echet estimates on the corrector} \label{sec:frechet}

We start by proving an estimate on the Fr\'echet derivative of functionals on
$ \mathbf{P} $ that are Lipschitz continuous with respect to Hilbert's metric $
d_{\mathbf{P}} $, which is defined in \eqref{e:hilbert}. We write $
\mathrm{Lip}(d_{\mathbf{P}}) $ for the space of Lipschitz continuous
functionals on $ \mathbf{P} $, cf. \eqref{e:def-lip}.

\begin{lemma}\label{lem:frechet}
For any functional $G \in \mathrm{Lip}(d_{\mathbf{P}})$, and
any $\pi \in \mathbf{P}$ and $\mN \in T_\pi$ it holds that:
\begin{equs}
\limsup_{\delta \downarrow 0}\delta^{-1} |G(\pi + \delta \mN)- G(\pi)|
\leqslant 2\|G\|_{\mathrm{Lip}(d_{\mathbf{P}})} \| \mN / \pi\|_{\infty}\;.
\end{equs}
\end{lemma}
\begin{proof}
For any $ \delta > 0$, if $ \delta <  \left( \min_{x \in \TT^{d}} \pi(x)
\right) / \| \mN \|_{\infty} $ we have that $
\pi + \delta \mN \in \mathbf{P} $. Therefore, for such $ \delta $ we have by
definition
\begin{equs}
|G(\pi + \delta \mN)- G(\pi)| \leqslant  \|G\|_{\mathrm{Lip}(d_{\mathbf{P}})}
d_{\mathbf{P}} (\pi + \delta \mN, \pi) \;.
\end{equs}
Now, from the definition of Hilbert's distance 
\begin{equs}
d_{\mathbf{P}} (\pi + \delta \mN, \pi)= \max \log \left( \frac{\pi + \delta \mN}{\pi}\right) -\min \log \left( \frac{\pi + \delta \mN}{\pi}\right) \;.
\end{equs}
Hence as $ \delta \downarrow 0 $ we obtain
\begin{equs}
\lim_{\delta \downarrow 0}\delta^{-1} d_{\mathbf{P}} (\pi + \delta \mN, \pi) = \max
\frac{\mN}{\pi} - \min \frac{\mN}{\pi} \leqslant 2\| \mN / \pi \|_{\infty}\;,
\end{equs}
which is the desired result.
\end{proof}
Next we recall a result which guarantees that $ (G (\pi_{t}))_{t \geqslant 0} $
is a semi-martingale. The proof of this result is analogous to \cite[Corollary
5.20]{gu2023kpz},
but we provide an easy self-contained proof below for the sake of completeness.
\begin{lemma}\label{lem:semimart}
The functional $ G \colon \mathbf{P} \to \RR $ constructed in
Lemma~\ref{lem:local-lipschitz} is Fr\'echet differentiable and satisfies
\begin{equs}
\ud G (\pi_{t}) = F(\pi_{t}) \ud t + \langle DG (\pi_t), \pi_t \ud W -
\langle\pi_t \ud W ,1 \rangle \pi_t \rangle\;.
\end{equs}
\end{lemma}
\begin{proof}
To establish this result it suffices to prove that $ G $ is Fr\'echet
differentiable. To prove that $ G $ is differentiable we use the representation 
\begin{equs}
G(\pi) = \int_{0}^{\infty} \EE[F ( \Phi^{\pr}_{t} \pi) ] \ud t \;, \quad
\Phi^{\pr}_{t} \pi = \frac{\Phi_{t} \pi}{\| \Phi_{t} \pi \|_{L^{1}}} \;,
\end{equs}
with $ \Phi $ the flow defined in \eqref{e:Phi}. Now, one can differentiate $
\Phi^{\pr}_{t} $ with respect to the initial datum to obtain
\begin{equs}
\langle D \Phi^{\pr}_{t} (\pi), \eta \rangle & = \frac{\Phi_{t} \eta}{\|
\Phi_{t} \pi\|_{L^{1}}} - \frac{\Phi_{t} \pi}{\| \Phi_{t} \pi
\|_{L^{1}}^{2}} \langle \Phi_{t} \eta, 1 \rangle \\
& = \left( \Phi^{\pr}_{t} \eta - \Phi^{\pr}_{t} \pi \right) \frac{\| \Phi_{t} \eta
\|_{L^{1}}}{\| \Phi_{t} \pi \|_{L^{1}}} \;,
\end{equs}
where the last identity holds assuming without loss of generality that $ \eta
> 0 $ (up to separating $ \eta = \eta_{+} +1 - (\eta_{-}+1) $ where $
\eta_{\pm}$ are respectively the positive and negative
parts of $ \eta  $, and using linearity of the derivative). Therefore,
combining the synchronisation in Corollary~\ref{cor:syncrho}
and the spectral gap estimate in Lemma~\ref{lem:dM}, one obtains that the
following integral is convergent, and is the Fr\'echet derivative of $ G
$: 
\begin{equs}
\langle D G(\pi), \eta \rangle = - \int_{0}^{\infty} \int_{(\TT^{d})^{2}}   \kappa
(x, y) \EE \left[ \langle D \Phi_{t}^{\mathrm{pr}} (\pi), \eta\rangle (x) \Phi^{\mathrm{pr}}_{t}
(\pi)(y) \right] \ud x \ud y  \ud t\;.
\end{equs}
We can bound the term inside the integral by
\begin{equs}
\| \kappa \|_{\infty} \EE \left[ \| \Phi^{\mathrm{pr}}_{t} (\pi) \|_{L^{1}} \| \langle
 D \Phi_{t}^{\mathrm{pr}} (\pi), \eta\rangle\|_{L^{1}} \right] \leqslant \|
\kappa \|_{\infty} \EE \left[ \| \langle
 D \Phi_{t}^{\mathrm{pr}} (\pi), \eta\rangle\|_{L^{1}} \right] \;.
\end{equs}
As for the last term, we can bound through the previous computation
\begin{equs}
 \| \langle
 D \Phi_{t}^{\mathrm{pr}} (\pi), \eta\rangle\|_{L^{1}} = \| \Phi^{\pr}_{t} \eta - \Phi^{\pr}_{t} \pi \|_{L^{1}}  \frac{\| \Phi_{t} \eta
\|_{L^{1}}}{\| \Phi_{t} \pi \|_{L^{1}}} \;.
\end{equs}
Since $ \eta >  0 $ there exists a $ c(\eta, \pi) $ such that $ \pi \leqslant c
\eta $. Hence, we can always bound $ \Phi_{t} \eta \leqslant c
\Phi_{t} \pi $ for
all $ t \geqslant 0 $. In particular, we deduce that
\begin{equs}
\EE \left[ \| \langle
 D \Phi_{t}^{\mathrm{pr}} (\pi), \eta\rangle\|_{L^{1}}  \right] \lesssim  \EE
\left[ \| \Phi^{\pr}_{t} \eta - \Phi^{\pr}_{t} \pi \|_{L^{1}}
\right] \;,
\qquad \forall t \geqslant 0\;.
\end{equs}
Now we can use \eqref{lem:bush} to obtain that for some $ \zeta > 0 $
\begin{equs}
\EE \left[ \| \Phi^{\pr}_{t} \eta - \Phi^{\pr}_{t} \pi \|_{L^{1}}  \right]
& \leqslant \EE \left[  \exp \left( d_{\mathbf{P}}(\pi, \eta ) \prod_{i=1}^{\lfloor t \rfloor} \tau
\left( \Phi^{\mathrm{pr}}_{i-1, i} \right)  \right) -1  \right] \\
& \lesssim \EE \left[  e^{d_{\mathbf{P}}(\pi, \eta )} d_{\mathbf{P}}(\pi, \eta )\prod_{i=1}^{\lfloor t \rfloor} \tau
\left( \Phi^{\mathrm{pr}}_{i-1, i} \right)   \right] \\
& \lesssim e^{d_{\mathbf{P}}(\pi, \eta )}d_{\mathbf{P}}(\pi, \eta)\EE \left[ \tau
\left( \Phi^{\mathrm{pr}}_{0, 1} \right) \right]^{\lfloor t
\rfloor} \lesssim e^{d_{\mathbf{P}}(\pi, \eta )}d_{\mathbf{P}}(\pi, \eta)
e^{- \zeta t} \;,
\end{equs}
where we have used the inequality $ e^{t x} -1 \leqslant e^{x} tx $ for $ t \in
[0, 1] $. Furthermore, we have written $ \Phi^{\mathrm{Pr}}_{s, t} $ for the
flow of the projective dynamic and $ \tau( A) $ for the contraction constant
with respect to Hilbert's distance of a linear operator acting on projective
space, see Theorem~\ref{thm:contraction}. This proves the desired result.
\end{proof}

\section{Proof of the main result: the linearisation step}\label{sec:linearisation}

In this section we complete the proof of Theorem~\ref{thm:moment-estimate},
bringing together all the elements of our analysis. In particular, the focus of this section lies in the linearisation
step, meaning that we reduce the analysis of \eqref{e:main} to that of the
linearised equation \eqref{e:linearised-n}, which was already analysed in
Section~\ref{sec:lyap-func}.

We split the proof of Theorem~\ref{thm:moment-estimate} into two cases: first we
treat a ``partially linear''
case, in which we assume $ \sigma (u) = u $. This case is less
technical, and allows us to introduce the strategy of the proof through a
suitable linearisation. We then treat the general nonlinear case, which
requires a few additional technical tweaks, in terms of suitable 
stopping times and cut-offs, to obtain the final result. In particular, in the
fully nonlinear case, we must deal with potential blow-up appearing in the
linearisation, due to the possible irregularity of the initial condition. Where possible, we
avoid repetitions.

Our approach to link the analysis of
\eqref{e:linearised-n} with that of the nonlinear equation \eqref{e:main} is to construct a process $ t \mapsto
w_{t} $ with the property that
\begin{equs}
w_{t} \lesssim  u_{t}\;, \qquad \forall t \geqslant 0 \;,
\end{equs}
which we build by following
the flow $ \Phi  $ associated to \eqref{e:linearised-n}, defined in
\eqref{e:Phi}, up to certain stopping times.
More precisely, we will define suitable stopping times
\begin{equs}
0 = \tau_{0} < \tau_{1} < \dots < \tau_{n} < \dots \;, \qquad \tau_{n} \uparrow
\infty \;,
\end{equs}
such that $ w_{t} \simeq \Phi_{\tau_{i}, t}[w_{\tau_{i}}] $ for all $ t \in
[\tau_{i}, \tau_{i+1}) $. The definition of the stopping times and how we
change $ w $ at these stopping times will be described in the next sections.
Since the construction of the stopping times is simpler in the partially linear case $ \sigma
(u) = u$, we start in that setting.

\subsection{The partially linear case} \label{sec:partilly-linear}
If \(\sigma\) is linear, that is $ \sigma (u) = u $, then we can construct
a process $ (w_{t})_{t \geqslant 0} $ with the properties described above by
using a comparison principle.
Indeed, because $ f \in C^{1} $ and $
f(0) = 0 $,  there exists a $ \ve_{0}  $ such that
\begin{equation} \label{e:epszero}
\begin{aligned}
f(u) \geqslant \left( f^{\prime} (0) - \frac{\lambda}{2} \right)  u \;, \qquad \forall u \in
[0, \ve_{0}] \;.
\end{aligned}
\end{equation}
We will consider $ \ve_{0} $ fixed henceforth.
Since we are interested only in the behaviour of \eqref{e:main} near $ u \equiv
0$, we introduce a cut-off process $ t \mapsto w_{t} $, which is comparable to the
linearisation \eqref{e:Phi} and bounds the solution $ t \mapsto u_{t} $ from below:
\begin{equation*}
\begin{aligned}
0 \leqslant w_{t} \leqslant u_{t}\;, \quad \text{ and } \quad \| w_{t}
\|_{\infty} \leqslant \ve_{0}  \;, \quad \forall t \geqslant 0\;.
\end{aligned}
\end{equation*}
The process $ w_{t} $ will be discontinuous (in time), at certain stopping
times, after which it will follow a dampened versions of the flow in
\eqref{e:Phi}. Here, to deal with the non-linearity we allow ourselves a small
``wiggling room'' by defining
\begin{equs}[e:xi]
\Xi_{s, t} = e^{- \frac{\lambda}{2} (t -s)} \Phi_{s, t} \;,
\end{equs}
so that $ \Xi $ is again the flow of a linear SPDE of the type \eqref{e:Phi},
only with $ \gamma = f^{\prime} (0) - \lambda /2 $. This will not be an issue,
because $ \Xi $ has Lyapunov exponent $ \lambda/2 $, which is still positive.
To be precise, let us introduce the cut-off operator  
\begin{equ}[e:def-cutoff]
\mT \varphi (x) = \min \{ \varphi (x), \ve  \} \;, \quad \forall
\varphi \in C(\TT^{d}; [0, \infty)) \quad  \text{ and for } \quad
\ve < \ve_{0}   \;.
\end{equ}
Note that $ \mT $ depends on the parameter $ \ve \in (0, \ve_{0}) $, but we will omit such
dependence to keep the notation clean. We will highlight when this dependence
is important. 
\begin{definition}[Piecewise linearised process]\label{def:w-cut-off}
For any $ \mf{t} > 0 $ and $ \ve \in (0, \ve_{0}) $
(with $ \ve_{0} $ as in \eqref{e:epszero}) let us define, for $ \Xi $ the flow associated to
\eqref{e:Phi} and any $ w \in C (\TT^{d}; [0, \infty)) $
\begin{equation*}
\begin{aligned}
\tau (s, w) = \inf \{ t > s  \; \colon \; \| \Xi_{s, t}[w] \|_{\infty}
\geqslant \ve_{0} \} \wedge(s+ \mf{t}) \;.
\end{aligned}
\end{equation*}
Next we define iteratively for all $ i \in \NN$:
\begin{equation*}
\begin{aligned}
\tau_{0} = 0 \;, \qquad \tau_{i+1} = \tau (\tau_{i},
\mT w_{\tau_{i}}) \;, \qquad w_{t} = \Xi_{\tau_{i}, t}
[\mT w_{\tau_{i}}] \;, \qquad \forall t \in [\tau_{i}, \tau_{i+1}) \;.
\end{aligned}
\end{equation*}
In particular, by comparison, the process $ w_{t} $ that we have
constructed satisfies the following: 
\begin{equation*}
\begin{aligned}
\| w_{t} \|_{\infty} \leqslant \ve_{0} \;, \qquad w_{t} \leqslant  u_{t}\;, \qquad
\forall t \geqslant 0 \;.
\end{aligned}
\end{equation*}
\end{definition}
We are now ready to prove Theorem~\ref{thm:moment-estimate} in the case $
\sigma(u) = u $, building on 
results which are stated in the upcoming sections.

\begin{proof}[of Theorem~\ref{thm:moment-estimate} in the linear case $ \sigma
(u) = u$ ]
The proof of Theorem~\ref{thm:moment-estimate} follows if we can prove that
there exist $ \eta, \zeta, C_{1}, C_{2} > 0 $ such that:
\begin{equation}\label{e:aim}
\begin{aligned}
\EE \left[ \left(\min_{x \in \TT^{d}} w (t, x)
\right)^{- \eta} \right] \leqslant C_{1} e^{- \zeta t} \left(\min_{x \in \TT^{d}} w (0, x)
\right)^{- \eta} + C_{2} \;, \qquad \forall t \geqslant 0 \;,
\end{aligned}
\end{equation}
where $ w_{t} $ is defined as in Definition~\ref{def:w-cut-off} with $
w_{0}= \mT u_{0}$.
To see that this implies Theorem~\ref{thm:moment-estimate}, we use that by construction $ w_{t} \leqslant
u_{t} $, so that as a consequence of the previous estimate
\begin{equs}
\EE \left[ \left(\min_{x \in \TT^{d}} u (t, x)
\right)^{- \eta} \right] \leqslant C_{1} e^{- \zeta t} \left(\min_{x \in \TT^{d}} w (0, x)
\right)^{- \eta} + C_{2}\;,
\end{equs}
and in addition that at time $ t = 0 $ we can bound $ \min_{x \in
\TT^{d}} u_{0} (x) \lesssim_{\ve} \min_{x \in \TT^{d}} w_{0}(x) $. Now, to obtain \eqref{e:aim} we use the
sequence of stopping times $ \{ \tau_{i} \}_{i \in \NN} $ as in
Definition~\ref{def:w-cut-off}. Our aim is to establish the Lyapunov property
for the process at the stopping times $ \tau_{i} $, and then build on this to
obtain the Lyapunov property for the process at deterministic times. This is
achieved by applying Lemma~\ref{lem:discretisations}. Here the key point is
that to obtain the Lyapunov property we must balance the contraction constant
$ \tilde{c} \in (0, 1) $ appearing in the discretised problem with certain
error terms. 
More precisely, as a consequence of Proposition~\ref{prop:intermediate} and
Proposition~\ref{prop:unif} we obtain that for some  $ \tilde{c}
\in (0,1), \widetilde{C}_{2}, \widetilde{C}_{3} \in (0, \infty) $ and $ \eta >
0 $ (up to replacing the $ \eta $ in Proposition~\ref{prop:intermediate} with
$ 2 \eta $):
\begin{equs}
\EE_{\tau_{i}} \left[  \left( \min_{x \in \TT^{d}} w (\tau_{i+1}, x)\right)^{-
2 \eta}  \right] & \leqslant \tilde{c} \left( \min_{x \in \TT^{d}} w (\tau_{i}, x)\right)^{-
2 \eta} + \widetilde{C}_{2} \;, \\
\EE_{\tau_{i}} \left[ \sup_{ \tau_{i}  \leqslant s < \tau_{i+1} }
\left( \min_{x \in \TT^{d}} w (s, x)\right)^{- 2 \eta} \right] &\leqslant
\widetilde{C}_{3}\left( \min_{x \in \TT^{d}} w (\tau_{i}, x)\right)^{- 2 \eta} \;.
\end{equs}
Then we are in the setting of
Lemma~\ref{lem:discretisations}, and \eqref{e:aim} follows if we can prove the
following estimates for some $ t_{\star} > 0 $:
\begin{equs}
\sum_{i \in \NN} \PP (t_{\star} \in
[\tau_{i}, \tau_{i +1}))^{\frac{1}{2}} & < \infty \;, \label{e:cont-aim1}\\
\sum_{i \in \NN} \widetilde{C}_{3} \tilde{c}^{i} \PP (t_{\star} \in [\tau_{i},
\tau_{i+1}))^{\frac{1}{2}} & < 1 \;, \label{e:cont-aim}
\end{equs}
since this would imply that for some $ c \in (0, 1) $ and $ C > 0 $
\begin{equs}
\EE \left[ \left(\min_{x \in \TT^{d}} w (t_{\star}, x)
\right)^{- \eta} \right] \leqslant c  \left(\min_{x \in \TT^{d}} w (0, x)
\right)^{- \eta} + C\;,
\end{equs}
which in turn implies \eqref{e:aim}.

Let us first concentrate on
proving the second estimate \eqref{e:cont-aim}.
To achieve this, we must choose appropriately the
parameters $ \mf{t}, \eta $ and $ \ve $. As for the constant $
\widetilde{C}_{3} $ we use that by
Proposition~\ref{prop:unif} 
we obtain an upper bound on $ \widetilde{C}_{3} $  that is uniform in $
\ve, $ and $ \eta \in (0, \eta_{0}) $ for some $ \eta_{0} > 0 $. Namely, we
have that (independently of $ \ve> 0 $)
\begin{equs}
C_{3} (\ve, \mf{t}) \leqslant e^{\alpha \mf{t} } \;,
\end{equs}
for some constant $ \alpha > 0 $.
In addition, if we choose $ \ve $ appropriately, we obtain an upper bound
on $ \tilde{c} $ that is uniform over $ \eta$. Namely, by
Proposition~\ref{prop:intermediate}, up to choosing a potentially smaller $
\eta_{0} > 0 $, we find that for any $
\mf{t} > 1$ there exists an $ \ve(\mf{t}) \in (0, \ve_{0}) $ such that
\begin{equs}
\tilde{c} (\mf{t}, \eta, \ve(\mf{t})) & \leqslant \widetilde{C}_{1}
(\eta_{0}) e^{- \zeta \eta \mf{t}} \;,
\end{equs}
for some $ \zeta > 0 $ independent of all other parameters.
Therefore, choose $ \mf{t} >1 $ sufficiently large and $ \ve(\mf{t}) > 0
$ sufficiently small for the above estimates to hold. Then, if $ t_{\star} =
(n+1) \mf{t} $, for some $ n \in \NN $ to be fixed later one, we find that by
construction $ \PP (t_{\star} \leqslant \tau_{n}) = 0 $ and hence
\begin{equs}
\sum_{i \in \NN } \widetilde{C}_{3}
\tilde{c}^{i} \PP (t_{\star} \in [\tau_{i}, \tau_{i+1}))^{\frac{1}{2}}
& \leqslant \sum_{i \geqslant n}\widetilde{C}_{3} \tilde{c}^{i} \\
& \leqslant \widetilde{C}_{3} \widetilde{c}(\mf{t}, \eta, \ve(\mf{t}))^{n}
\frac{1}{1- \widetilde{c}(\mf{t}, \eta, \ve(\mf{t}))} \\
& \leqslant \overline{C} e^{\alpha \mf{t}  - n \zeta \eta \mf{t}} \;,
\end{equs}
for some constant $ \overline{C} > 0 $. It follows that if $ n $ is chosen
sufficiently large, then $ \overline{C} e^{\alpha \mf{t}  - n \zeta \eta
\mf{t}} < 1 $ as desired.

At this point all the parameters of the problem have been fixed and we are left
with checking \eqref{e:cont-aim1}. This is a consequence of
Lemma~\ref{lem:lb-st} based on the uniform estimate on the linear flow $\Xi$ obtained in Lemma~\ref{lem:ub}, together with the same arguments as in the proof of
\cite[Lemma 5.2]{hairer2023spectral}. Indeed, Lemma~\ref{lem:lb-st} guarantees
that the stopping times do not kick in too quickly, so that eventually
\begin{equs}
\PP (t_{\star} \in [\tau_{i}, \tau_{i+1})) \leqslant c e^{- \widetilde{c} i}
\;,
\end{equs}
for some constants $ c, \widetilde{c} $ which depend on all the parameters of
the problem. This concludes the proof of the result.
\end{proof}

\subsection{The discrete Lyapunov property in the partially linear case}
\label{sec:sup-lin}

In this subsection we collect a number of supporting statements that
are required for the proof of Theorem~\ref{thm:moment-estimate} in the
partially linear case $ \sigma (u) = u $.
We start with a discrete Lyapunov
property of the process $ (w_{t})_{t \geqslant 0} $, evaluated at the stopping
times $ \{ \tau_{i} \}_{i \in \NN} $ from Definition~\ref{def:w-cut-off}.

\begin{proposition}\label{prop:intermediate}
In the setting of Theorem~\ref{thm:moment-estimate}, there exist $ \zeta,
\eta_{0} > 0 $
such that the following holds. For any $ \mf{t} >1 $ there exists an $
\ve (\mf{t}) \in (0, \ve_{0}) $ and a constant $
\widetilde{C}_{1}(\eta_{0})$ (independent of $ \ve $) and $\widetilde{C}_{2} ( \mf{t},
\ve, \eta_{0}) > 0 $ such that for all $i \in \NN$ and $ \eta \in (0,
\eta_{0}) $
\begin{equation*}
\begin{aligned}
\EE_{\tau_{i}} \left[ \left(\min_{x \in \TT^{d}} w ( \tau_{i+1}, x)
\right)^{- \eta} \right] \leqslant \widetilde{C}_{1}(\eta_{0}) e^{- \eta \zeta \mf{t}}
\left(\min_{x \in \TT^{d}} w ( \tau_{i}, x) \right)^{- \eta}+
\widetilde{C}_{2}( \mf{t}, \ve, \eta_{0}) \;.
\end{aligned}
\end{equation*}
\end{proposition}

\begin{proof}
Recall that $ w $ is c\'adl\'ag, with possible jumps at the stopping times $
\tau_{i} $. Therefore, we start by bounding the jump at time $ \tau_{i+1} $, where
\begin{equation*}
\begin{aligned}
w (\tau_{i +1}, x) = \mT w (\tau_{i +1} -, x) \;.
\end{aligned}
\end{equation*}
Therefore
\begin{align*}
    \left( \min_{x \in \TT^{d}} w ( \tau_{i+1}, x) \right)^{-\eta} & \leqslant
\left(\min_{x \colon w ( \tau_{i+1}, x) \leqslant  \ve} w ( \tau_{i+1}, x) \right)^{-\eta} +\ve^{-\eta} \\
    & \leqslant \left(\min_{x \in \TT^d} w ( \tau_{i+1}-, x) \right)^{-\eta} +\ve^{-\eta} \;.
\end{align*}
In particular, we have obtained the bound:
\begin{equation}\label{e:prf-prop-1}
\begin{aligned}
\EE_{\tau_{i}} \left[ \left(\min_{x \in \TT^{d}} w ( \tau_{i+1}, x)
\right)^{- \eta} \right] \leqslant  \EE_{\tau_{i}} \left[ \left(\min_{x \in \TT^{d}} w (
\tau_{i+1}-, x) \right)^{- \eta} \right] + \ve^{-\eta}.
\end{aligned}
\end{equation}
Now we will prove the result by obtaining a suitable upper bound to the right
hand-side of \eqref{e:prf-prop-1}. To do so, we distinguish between the events
\begin{align*}
    \tau_{i+1} = \tau_i + \mf{t}\;, \qquad \text{ and } \qquad \tau_{i+1} < \tau_i + \mf{t} \;.
\end{align*}

\textit{The event $ \tau_{i+1} = \tau_{i} + \mf{t} $.}
Here we prove that there exists a $ \zeta > 0 $ such that
\begin{equation} \label{e:aim-int1}
\begin{aligned}
    \EE_{\tau_{i}} \left[ \left(\min_{x \in \TT^{d}} w ( \tau_{i+1}-, x)
\right)^{- \eta} 1_{\{\tau_{i+1}= \tau_i + \mf{t}\} } \right] \leqslant
\widetilde{C}_{1}(\eta_{0}) e^{- \zeta \mf{t}} \left(\min_{x \in \TT^{d}} w ( \tau_{i}, x)
\right)^{- \eta} \;.
\end{aligned}
\end{equation}
Indeed, since
\begin{equation*}
\begin{aligned}
\EE_{\tau_{i}} \left[ \left( \min_{x \in \TT^{d}}  w (\tau_{i+1}-, x) 
\right)^{- \eta } 1_{\{ \tau_{i +1} = \tau_{i} + \mf{t} \}} \right] \leqslant 
\EE_{\tau_{i}} \left[ \left(  \min_{x \in \TT^{d}}  \Xi_{\tau_{i}, \tau_{i} +
\mf{t}}[ w (\tau_{i}, \cdot)] (x)
\right)^{- \eta } \right] \;,
\end{aligned}
\end{equation*}
we can apply Proposition~\ref{prop:lin} (which is the analogue
of Theorem~\ref{thm:moment-estimate} for linear equations).
Note that the
proposition is stated for $
\Phi $, but applies analogously to the flow $ \Xi $ defined in \eqref{e:xi},
since the latter still has a positive Lyapunov exponent.
Further observe that we can apply Proposition~\ref{prop:lin} provided we choose
$ \mf{t} > 1 $, to obtain 
\begin{equation*}
\begin{aligned}
\EE_{\tau_{i}} \left[ \left(  \min_{x \in \TT^{d}}  w (\tau_{i+1}-, x)
\right)^{- \eta } 1_{\{ \tau_{i +1} = \tau_{i} + \mf{t} \}} \right] & \leqslant
C(\eta_{0}) e^{- \zeta(\gamma) \mf{t}} \left( \int_{\TT^{d}}
 w(\tau_{i}, x)  \ud x \right)^{- \eta} \\
& \leqslant C(\eta_{0}) e^{- \zeta(\gamma) \mf{t}} \left( \min_{x \in
\TT^{d}}  w(\tau_{i}, x)  \right)^{- \eta}\;,
\end{aligned}
\end{equation*}
where the constant $ C(\eta_{0})  $ appearing on the right-hand side is independent of $
\ve $. Hence, \eqref{e:aim-int1} holds with $
\widetilde{C}_{1}(\eta_{0})  = C (\eta_{0})$.

\textit{The event $ \tau_{i+1} < \tau_{i} + \mf{t} $.} Here we use
Cauchy--Schwarz to bound
\begin{equation*}
\begin{aligned}
\EE_{\tau_{i}} & \left[  \left(  \min_{x \in \TT^{d}} w (\tau_{i+1}-, x) 
\right)^{- \eta } 1_{\{ \tau_{i +1} < \tau_{i} + \mf{t} \}}  \right] \\
& \leqslant \EE_{\tau_{i}} \left[  \left( \min_{x \in \TT^{d}}  w (\tau_{i+1}-, x) 
\right)^{- 2 \eta } 1_{\{ \tau_{i +1} < \tau_{i} + \mf{t} \}} 
\right]^{\frac{1}{2}} \PP_{\tau_{i}}(\tau_{i +1} < \tau_{i} +
\mf{t})^{\frac{1}{2}} .
\end{aligned}
\end{equation*}
Then we will prove the following two facts. First, we will show that we can
choose $ \ve( \mf{t}) \in (0, \ve_{0}) $ sufficiently small such that 
\begin{equation}\label{e:prf-prop-2}
\begin{aligned}
\PP_{\tau_{i}} (\tau_{i +1} < \tau_{i} + \mf{t}) \leqslant e^{- 2 \zeta
\mf{t}} \;.
\end{aligned}
\end{equation}
Second, we show that there exists a constant $ C> 0 $ such that
\begin{equation}\label{e:prf-prop-3}
\begin{aligned}
\EE_{\tau_{i}}\left[  \left(  \min_{x \in \TT^{d}} w (\tau_{i+1}-, x)
\right)^{- 2 \eta } 1_{\{ \tau_{i +1} < \tau_{i} + \mf{t} \}} 
\right]^{\frac{1}{2}} \leqslant C  \left( \min_{x \in \TT^{d}} w (\tau_{i},
x) \right)^{- \eta}  \;.
\end{aligned}
\end{equation}
As for \eqref{e:prf-prop-3}, it suffices to show that for any $ \eta > 0 $
\begin{equation}\label{e:prf-prop-4}
\begin{aligned}
\EE_{\tau_{i}} \left[ \left(  \min_{x \in \TT^{d}}  \Xi_{\tau_{i}, \tau_{i+1}}
[w] (x)  \right)^{- \eta } 1_{\{ \tau_{i+1} < \tau_{i} + \mf{t} \}}
\right] \leqslant C (\eta) \;, \qquad \forall w  \; \colon \; \min_{x
\in \TT^{d}} w(x) \geqslant 1 \;.
\end{aligned}
\end{equation}
Since $ \Xi $ is the flow to a linear equation, we can represent it through an
integral kernel: 
\begin{equation*}
\begin{aligned}
\Xi_{s, t} [w] = \int_{\TT^{d}} K_{s, t} (x, y) w (y) \ud y \;.
\end{aligned}
\end{equation*}
Then if $ \min_{x \in \TT^{d}} w (x) \geqslant 1 $ we can write
\begin{equation*}
\begin{aligned}
\min_{x \in \TT^{d}} \Xi_{\tau_{i}, \tau_{i} + t}
[w] (x) & = \min_{x \in \TT^{d}} \int_{\TT^d} K_{\tau_{i}, \tau_{i} + t} (x, y) w (y) \ud y \\
&\geqslant \min_{x \in \TT^{d}} \int_{\TT^d} K_{\tau_{i}, \tau_{i} + t} (x, y) \ud y = c_{K}
(\tau_{i}, \tau_{i} + t) > 0 \;,
\end{aligned}
\end{equation*}
where we have defined
\begin{equation*}
\begin{aligned}
c_{K} (s, t) = \min_{x \in \TT^{d}}\int_{\TT^{d}} K_{s , t} (x, y)  \ud y \;.
\end{aligned}
\end{equation*}
Then by Lemma~\ref{lem:c-bd} (once more, the lemma is stated for $ \Phi $ but
holds analogously for $ \Xi $) we know that for $ \eta_{0} $ sufficiently small
and all $ \eta \leqslant \eta_{0} $
\begin{equation*}
\begin{aligned}
\EE_{\tau_{i}} \left[ \sup_{0 \leqslant t \leqslant 1 }
c_{K}(\tau_{i}, \tau_{i} +t)^{- 2 \eta}\right] < C(\eta) \;.
\end{aligned}
\end{equation*}
Therefore it makes sense to further decompose the current case $
\tau_{i +1} < \tau_{i} + \mf{t} $ into the two different cases $
\tau_{i+1} < \tau_{i} + 1 $ and  $ \tau_{i+1} \in
[\tau_{i}+1 , \tau_{i} + \mf{t}) $.
In the first case we find for any $ w $ with $ \min_{x \in \TT^{d}} w(x)
\geqslant 1 $
\begin{equation*}
\begin{aligned}
\EE_{\tau_{i}} \left[ \left( \min_{x \in \TT^{d}}  \Xi_{\tau_{i}, \tau_{i+1}}
[w] (x) \ud x \right)^{- 2\eta } 1_{\{ \tau_{i+1} < \tau_{i} + 1\}}\right] \leqslant
\EE_{\tau_{i}} \left[  \sup_{0 \leqslant t \leqslant 1} c_{K}(\tau_{i},
\tau_{i}+t)^{- 2 \eta}\right] \leqslant C( \eta) \;.
\end{aligned}
\end{equation*}
To conclude the proof of \eqref{e:prf-prop-4} we have now reduced ourselves to
finding an upper bound  for any $ w $ with $ \min_{x \in \TT^{d}} w
(x) \geqslant 1 $:
\begin{equation*}
\begin{aligned}
\EE_{\tau_{i}} \left[ \left(  \min_{x \in \TT^{d}}  \Xi_{\tau_{i}, \tau_{i+1}}
[w] (x) \ud x \right)^{- 2\eta } 1_{ [ \tau_{i} + 1,
\tau_{i} + \mf{t})} (\tau_{i+1}) \right] \;.
\end{aligned}
\end{equation*}
In this case we can use the second statement of Proposition~\ref{prop:lin}, to
obtain for some constant $C(\eta_{0}) > 0 $ (which is uniform over $ \mf{t} $ and
$ \ve $):
\begin{equation*}
\begin{aligned}
\EE_{\tau_{i}} \left[ \left(  \min_{x \in \TT^{d}}  \Xi_{\tau_{i}, \tau_{i+1}}
[w] (x) \ud x \right)^{- 2\eta } 1_{ [ \tau_{i} + 1,
\tau_{i} + \mf{t})} (\tau_{i+1}) \right] \leqslant C(\eta_{0}) \left(
\int_{\TT^{d}} w(x) \ud x \right)^{- \eta} \leqslant C (\eta_{0}) \;,
\end{aligned}
\end{equation*}
since $ \min_{x \in \TT^{d}} w (x) \geqslant 1 $.
Therefore, the proof is complete if we show that \eqref{e:prf-prop-2} holds
true. This in turn is a simple consequence of the upper bound from
Lemma~\ref{lem:ub} (with $ \| w \|_{\infty} \leqslant \ve $), by
choosing $ \ve (\mf{t}) $ sufficiently small.
\end{proof}
The next ingredient in the proof of Theorem~\ref{thm:moment-estimate} is the
following uniform estimate.

\begin{proposition}\label{prop:unif}
There exists an $ \eta_{0} > 0 $ such that the following bounds are uniform
over all $ \eta \in (0, \eta_{0}) $, all $ \mf{t}> 0 $ and all $ i \in \NN $,
for some constants $ C, c > 0 $:
\begin{equation*}
\begin{aligned}
\EE_{\tau_{i}} \left[ \sup_{\tau_{i} \leqslant t < \tau_{i+1}} \left( \min_{x \in \TT^{d}} w (
t, x) \right)^{- \eta  } \right] \leqslant e^{c \mf{t} } \left( \min_{x \in \TT^{d}} w (
\tau_{i}, x) \right)^{- \eta} \;.
\end{aligned}
\end{equation*}
\end{proposition}

\begin{proof}
Without loss of generality, we may assume that $ \tau_{i} = 0 $. Then observe
that for any $ 0 \leqslant s \leqslant t $
\begin{equs}
\min_{x \in \TT^{d}} \Xi_{0, t} [w] (x) \geqslant  \left( \min_{x \in \TT^{d}}
\int_{\TT^{d}} K_{s, t} (x, y) \ud y \right) \left( \min_{x \in \TT^{d}}
\Xi_{0, s}[w] (x) \right)  = c_{K} (s, t) \left( \min_{x \in \TT^{d}}
\Xi_{0, s}[w] (x) \right)\;,
\end{equs}
for all $ w \geqslant 0  $, and
where $ K_{s, t} $ is the integral kernel associated to the flow $
\Xi_{s, t} $ in \eqref{e:xi} and
\begin{equs}
c_{K}(s, t)=  \min_{x \in \TT^{d}} \int_{\TT^{d}} K_{s, t} (x, y) \ud y \;.
\end{equs}
Iterating this bound $ \lfloor \mf{t} \rfloor + 1 $ times we obtain that if in
addition $ \min_{x \in \TT^{d}} w (x) \geqslant 1 $ (which we can assume
without loss of generality):
\begin{equs}
\sup_{0 \leqslant t \leqslant \mf{t}}  \left( \min_{x \in \TT^{d}} w (
t, x) \right)^{- \eta  } \leqslant \prod_{i =0}^{\lfloor t \rfloor} \sup_{0
\leqslant t \leqslant 1 } c_{K} (i,i+ t)^{- \eta} \;,
\end{equs}
where we used that since $ c_{K}(i,i) = 1 $ each of the terms in the product is
greater than $ 1 $.
Furthermore, any two $ c_{K}(i, i+s), c_{K}(j, j+s) $ are independent for $
i \neq j $ and $ s \in [0, 1] $. Therefore we obtain
\begin{equs}
\EE \left[ \sup_{0 \leqslant t \leqslant \tau_{1}}  \left( \min_{x \in \TT^{d}} w (
t, x) \right)^{- \eta  } \right] \leqslant \prod_{i =1}^{\lfloor \mf{t}
\rfloor}\EE \left[ \sup_{0 \leqslant t \leqslant 1 } c_{K} (i,i+ t)^{- \eta}
\right] \leqslant e^{c \lfloor t \rfloor} \;,
\end{equs}
by Lemma~\ref{lem:c-bd}, for some $ c > 0 $, provided that $ \eta
\leqslant \eta_{0} $ for some suitably small $ \eta_{0} > 0 $. This proves the
result (recall that we we have assumed $ \min_{x \in \TT^{d}} w(x) \geqslant 1
$).

\end{proof}
Recalling that the flow $ \Xi $ solve \eqref{e:Phi} with
$\gamma=f'(0)$ we obtain the
following result for the linear flow $\Xi$.

\begin{proposition}\label{prop:lin} Let $ \Xi $ be the flow defined in
\eqref{e:xi}.
Under the assumption of Theorem~\ref{thm:moment-estimate} there exist deterministic constants $
\eta_{0}(\gamma ), \zeta(\gamma) > 0 $ such that
for all $ \eta \in (0, \eta_{0}) $, any stopping time $ \tau $, and any $ \mF_{\tau}
$-adapted $ w \in C(\TT^{d}) $ we have for some $ C(\eta_{0}) > 0 $
\begin{equation*}
\begin{aligned}
\EE_{\tau} \left[ \left( \min_{x \in \TT^{d}} \Xi_{\tau, \tau +
t}[w] (x)
 \right)^{- \eta } \right]^{\frac{1}{\eta}} \leqslant C(\eta_{0})  e^{- \zeta t }
\int_{\TT^{d}} w (x) \ud x  \;, \qquad \forall t > 1 \;.
\end{aligned}
\end{equation*}
\end{proposition}
\begin{proof}
This result follows from Theorem~\ref{thm:linearised-1}.
Indeed, we have that for any $ \eta, t > 0 $
\begin{equs}
\EE_{\tau} & \left[ \left( \min_{x \in \TT^{d}} \Xi_{\tau, \tau+ t}[w]  (x)
\right)^{- \eta}  \right]\\
& \lesssim \EE_{\tau} \left[ \left(  \int \Xi_{\tau, \tau+ t} [w] ( x) \ud x \right)^{- 2 \eta}
\right]^{\frac{1}{2}} \EE_{\tau} \left[ \left( \min_{x \in \TT^{d}} \pi_{\tau,
\tau+ t} [w] ( x)
\right)^{- 2 \eta} \right]^{\frac{1}{2}} \;, \label{e:prf-1}
\end{equs}
where for $ s \leqslant t $ the projective flow $ \pi_{s,t} $ is defined as
\begin{equs}
\pi_{s, t}[w](x) =  \Xi_{s, t}[w](x) \left( \int_{\TT^{d}}
\Xi_{s, t}[w](x) \ud x \right)^{-1} \;.
\end{equs}
Now, if $ K_{s, t} $ is the kernel associated to $ \Xi_{s, t} $, so that $
\Xi_{s, t} [w] (x) = \smallint_{\TT^{d}}K_{s, t} (x, y) w (y) \ud y  $. Then
\begin{equs}
\min_{x \in \TT^{d}} \pi_{s, t} [w] (x) \geqslant  \left( \min_{x,y \in \TT^{d}}
K_{s, t} (x, y)\right) \left( \int_{\TT^{d}} w(y) \ud y \right) \left(
\int_{\TT^{d}} \Xi_{s, t}[w] (x) \ud x \right)^{-1} \;.
\end{equs}
In addition
\begin{equs}
 \left( \int_{\TT^{d}} \Xi_{s, t}[w] (x) \ud x \right)^{-1} \geqslant \left(
\max_{x, y \in \TT^{d} } K_{s, t} (x, y) \right)^{-1} \left(
\int_{\TT^{d}} w (y) \ud y \right)^{-1} \;,
\end{equs}
so that we obtain
\begin{equs}
\min_{x \in \TT^{d}} \pi_{s, t} [w] (x) \geqslant  \left( \min_{x,y \in \TT^{d}}
K_{s, t} (x, y)\right) \left(
\max_{x, y \in \TT^{d} } K_{s, t} (x, y) \right)^{-1} \;.
\end{equs}
Note that this lower bound is uniform over the initial condition. Therefore,
using that $ t> 1 $ and with $ t_{0} \in (0, 1) $ as in Lemma~\ref{lem:c-bd} we
find that with $ \sigma = \tau + t - t_{0} $
\begin{equs}
\EE_{\tau} \left[ \left( \min_{x \in \TT^{d}} \pi_{\tau,
\tau+ t} [w] ( x)
\right)^{- 2 \eta} \right] & = \EE_{\tau} \left[ \left( \min_{x \in \TT^{d}}
\pi_{\sigma,
\sigma+ t_{0}} [ \pi_{\tau, \sigma}[ w] ] ( x)
\right)^{- 2 \eta} \right] \\
& \leqslant \EE_{\tau} \left[ \left( \min_{x,y \in \TT^{d}}
K_{\sigma, \sigma + t_{0}} (x, y)\right)^{2 \eta} \left(
\max_{x, y \in \TT^{d} } K_{\sigma, \sigma + t_{0}} (x, y) \right)^{-2 \eta}
\right]\\
& \leqslant C (\eta_{0}) \;.
\end{equs}
Instead, for the first term in \eqref{e:prf-1} we have via
Theorem~\ref{thm:linearised-1} that (provided $ \eta_{0}(\gamma) $ and $
 \zeta (\gamma) $ are chosen appropriately small):
\begin{equs}
\EE_{\tau} \left[ \left(  \int_{\TT^{d}} \Xi_{\tau, \tau+ t} [w] ( x) \ud x \right)^{- 2
\eta} \right] \leqslant C e^{- 2 \eta \zeta t } \left( \int_{\TT^{d}} w (x)
\ud x \right)^{- 2\eta} \;.
\end{equs}
As usual, we remark that
Theorem~\ref{thm:linearised-1} applies to the flow $ \Phi $, but also to the
flow $ \Xi $, since the latter still admits a positive Lyapunov exponent.
This completes the proof of the desired result. 
\end{proof}

\subsection{The fully nonlinear case} \label{sec:fully-nonlinear}

The crucial difference between the linear $ \sigma $ case and the fully
nonlinear case, is that we can not use a maximum principle to compare the
solution to the linearised process to the original equation - at least not in
the simple way we did previously. Instead we must
consider a slightly more involved linearisation argument, which requires the
introduction of several auxiliary stopping times. Overall though, the strategy remains
unchanged.

We start by considering $ \Psi $ be the nonlinear flow associated to \eqref{e:main}, so that $
(t, x) \mapsto \Psi_{s, t}[u] (x) $ is the solution to \eqref{e:main} with
initial condition $ u \in C(\TT^{d}; [0, \infty)) $ at time $ s \geqslant 0 $. Then, let us introduce the
following stopping times, where $ \ve_{1}, \delta \in (0, 1), \varrho \in
(0, 1/2) $ and $ M > 1
$ are parameters that we will fix later on:
\begin{equs}[e:def-stnl]
\tau (s,u) & = \inf  \{ t > s  \; \colon \; \| \Psi_{s, t}[u] \|_{\infty}
\geqslant \ve_{1} \}\;,\\
\tau^{X} (s, u)& = \inf \left\{ t \geqslant s  \; \colon \; \| 
X_{s, t}[u] -  Y_{s, t} \|_{\infty} + (t-s)^{\varrho}\| \nabla ( X_{s, t}[u] -  Y_{s, t})
\|_{\infty} \geqslant \delta 
\right\} \;, \\
\tau^{Y} (s,u) & = \inf \left\{ t \geqslant s  \; \colon \; \| 
Y_{s, t} \|_{\mC^1} \geqslant M \right\} \;.
\end{equs}
Here the processes $ X_{s, \cdot}[u] $
and $ Y_{s, \cdot} $ are respectively the solutions to 
\begin{equs}[e:XY]
\ud  X_{s,t}[u] & = \Delta X_{s, t}[u] \ud t +  \frac{f(\Psi_{s, t}[u])}{\Psi_{s,
t}[u]} \ud t + \frac{\sigma
(\Psi_{s, t}[u])}{\Psi_{s, t}[u]} \ud W_{t} \;, \qquad & & X_{s,s} =0 \;,\\
\ud Y_{s, t} & = \Delta Y_{s, t}  \ud t +  f^{\prime} (0) \ud t  + 
\ud W_{t}\;, \qquad & & Y_{s,s} = 0 \;.
\end{equs}
Recall that by Assumption~\ref{assu:nonlinearities}, $ \sigma^{\prime}
(0) =1 $.
Furthermore, let us observe that we have introduced a parameter $ \varrho $, which describes a
potential blow-up of the difference $ X_{s, y}[u] - Y_{s, t} $ at time $ t =s $. This is a technical
necessity which is the consequence of the fact that we do not assume that $
u $ is smooth, but only ask that it is bounded in $ L^{\infty}
$. Therefore, we require some degree of regularisation from the heat
semigroup to obtain smoothness of $ X[u] $, see Lemma~\ref{uy}.
The parameter $ \varrho $ can be chosen arbitrarily small.

Finally, let $ \tau^{\mathrm{tot}}(s, u) $ be the smallest of all the
stopping times in \eqref{e:def-stnl}, up to a deterministic threshold $ s + \mf{t} $, where $
\mf{t} >0$ is an additional parameter that will be fixed later on:
\begin{equs}[e:df-ttot]
\tau^{\mathrm{tot}}(s, u) =  \tau(s, u) \wedge \tau^{X} (s, u) \wedge
\tau^{Y}(s, u) \wedge(s+ \mf{t}) \;.
\end{equs}
The motivation for this definition is that we can rewrite the flows
associated to the nonlinear equation \eqref{e:main} and to the linearised
equation \eqref{e:Phi} as follows:
\begin{equs}\label{phi:psi}
\Psi_{s, t} [u] = e^{X_{s, t}[u]} \overline{\Psi}_{s, t} [u] \;, \qquad
\Phi_{s, t} [w] = e^{Y_{s, t}} \overline{\Phi}_{s, t} [w] \;,
\end{equs}
where $ \overline{\Psi} $  solves the following equation with $
\kappa_{\mathrm{tr}} (x) = \kappa (x, x) $, with $ \kappa $ the correlation
function in Assumption~\ref{assu:nonlinearities}:
\begin{equs}[e:psibar]
(\partial_{t} - \Delta) \overline{\Psi}_{s, t} [u] & =2\nabla \overline{\Psi}_{s,t}[u] \cdot \nabla X_{s,t}[u]+
\overline{\Psi}_{s, t} [u] |\nabla X_{s,t}[u] |^{2}-\frac{1}{2} \kappa_{\tr}
\frac{\sigma(\Psi_{s,t}[u])^2}{\Psi^2_{s,t}[u]} \overline{\Psi}_{s,t}[u]\;, \\
\overline{\Psi}_{s, s} [u] & = u \;,
\end{equs}
and similarly $ \overline{\Phi} $ solves the equation:
\begin{equs}[e:phibar]
(\partial_{t} - \Delta) \overline{\Phi}_{s, t} & = 2 \nabla \overline{\Phi}_{s,t} \cdot \nabla Y_{s,t} + \overline{\Phi}_{s, t} |
\nabla Y_{s, t} |^{2} - \frac{1}{2} \kappa_{\tr}
\overline{\Phi}_{s,t} \;, \qquad  & & \overline{\Phi}_{s,
s}[u] = u \;.
\end{equs}
Here the terms involving $\kappa_{\text{tr}}$ 
arise from the quadratic variation terms in It\^o's formula, see the derivation
of \eqref{e:v2}.
In the present setting, we can lower bound the evolution of the non-linear
equation through that of the linear flow $ \Phi $, as follows.

\begin{lemma}\label{lem:comp-phi-psi}
There exists a (deterministic) constant $ C > 1 $, and for every $
\mf{t} > 1 $ there exists a constant $ \overline{\ve}_{1}(\mf{t}) \in
(0, 1) $, such that for all parameters $
\delta \in (0, 1), M >1  $ satisfying
\begin{equs}[e:dtm]
\delta e^{C \mf{t} M^{2}} \leqslant 1 \;, \qquad \ve_{1} \in (0,
\overline{\ve}_{1} (\mf{t})) \;,
\end{equs}
the following holds.
For all $ 0 \leqslant s  < \infty $ and  $ u \in C
(\TT^{d}; [0, \infty)) $:
\begin{equs}
\Psi_{s, t} [u] \geqslant \frac{1}{2} \Phi_{s, t} [u] \;, \qquad \forall s
\leqslant t < \tau^{\mathrm{tot}}(s, u) \;.
\end{equs}
\end{lemma}
\begin{proof}
We start by observing that since $ \| X_{s, t}[u] - Y_{s, t} \|_{\infty} \leqslant \delta  $ for $ t
< \tau^{\mathrm{tot}}(s, u) $, we immediately deduce that $ \exp (X_{s, t}[u])
\geqslant \exp (Y_{s, t} - \delta) $. Therefore, in view of~\eqref{phi:psi}, if we show that 
$ \overline{\Psi}_{s, t} [u] \geqslant \frac{3}{4}  \overline{\Phi}_{s, t} [u]
$ for $ s \leqslant t \leqslant \tau^{\mathrm{tot}}(s, u) $, we deduce the
desired result, provided that $ e^{- \delta} 3/4 > 1/2  $.
To prove this estimate, we rewrite $ \overline{\Psi} [u] $ through the
fundamental solution $ \Gamma^{\Psi} $ to \eqref{e:psibar} and similarly $
\overline{\Phi}[u] $ through the fundamental solution $ \Gamma^{\Phi} $ to
\eqref{e:phibar}. Namely, for any $ y \in \TT^{d} $ let $ \Gamma^{\Psi} $ solve
\begin{equs}[e:gamma-psi]
(\partial_{t} - \Delta) \Gamma^{\Psi}_{s, t} & =2\nabla \Gamma^{\Psi}_{s,t} \cdot \nabla X_{s,t}[u]+
\Gamma^{\Psi}_{s, t}  |\nabla X_{s,t}[u] |^{2}-\frac{1}{2} \kappa_{\tr}
\frac{\sigma(\Psi_{s,t}[u])^2}{\Psi^2_{s,t}[u]} \Gamma^{\Psi}_{s,t}\;, \\
\Gamma^{\Psi}_{s, s} (x, y)  & = \delta_{y} (x) \;,
\end{equs}
so that $ \overline{\Psi}_{s, t} [u] = \int_{\TT^d} \Gamma^{\Psi}_{s, t} (x, y)
u(y) \ud y $. Note that this notation can be slightly misleading, because $
\Gamma^{\Psi} $ depends itself on $ u $ and on $ \Psi [u] $. On the other hand, we
now use the fact that when $ u $ is  close to the origin, then the dynamics of $ \Gamma^{\Psi} $ are
close to those of $ \Gamma^{\Phi} $, where $\Gamma^{\Phi}$ solves for any $y\in\TT^d$ 
\begin{equs}[e:gamma-phi]
(\partial_{t} - \Delta) \Gamma^{\Phi}_{s, t} & =2\nabla \Gamma^{\Phi}_{s,t} \cdot \nabla Y_{s,t}+
\Gamma^{\Phi}_{s, t}  |\nabla Y_{s,t} |^{2}-\frac{1}{2} \kappa_{\tr}
\Gamma^{\Phi}_{s,t}\;, \\
\Gamma^{\Phi}_{s, s} (x, y)  & = \delta_{y} (x) \;.
\end{equs}
Indeed, we will prove that there exists a
constant $ c > 0 $ such that
\begin{equs}
\Gamma^{\Psi}_{s, t} (x, y) \geqslant c \Gamma^{\Phi}_{s, t} (x, y) \;, \qquad \forall
x, y \in \TT^{d} \;, \quad s \leqslant t < \tau^{\mathrm{tot}}
(s, u) \;.
\end{equs}
To obtain this bound, first of all, we consider the multiplicative terms in
\eqref{e:gamma-psi}, and then the drift terms. Since $\sigma\in C^1$ with
$\sigma'(0)=1$  we have that for $ t \leqslant
\tau^{\mathrm{tot}}(s, u) $ and any $ \ve^{\prime} \in (0, 1) $
%\tommaso{Recall that we don't assume
%$ C^{2} $.}
\begin{equs}
\|| \nabla X_{s, t}[u] |^{2} - | \nabla Y_{s, t} |^{2} \|_{\infty} \leqslant 
 \delta (t -s)^{-\varrho} (2 M + \delta (t-s)^{- \varrho})
\;, \qquad \left\| 1 - \frac{\sigma^{2}(\Psi_{s, t}[u])}{\Psi^{2}_{s,
t}[u]} \right\|_{\infty} \leqslant \ve^{\prime}  \;,
\end{equs}
where the last inequality follows from the continuous differentiability of $ \sigma $, provided that
$ \ve_{1} =\ve_{1}( \ve^{\prime}) $ is chosen sufficiently small, depending on $ \ve^{\prime}  $.
Since \begin{align*}
(\partial_t -\Delta) \Gamma^{\Psi}_{s,t}&= 2\nabla
\Gamma^{\Psi}_{s,t}\cdot\nabla X_{s,t}[u] +\Gamma^{\Psi}_{s,t}|\nabla
Y_{s,t}|^2  +\Gamma^{\Psi}_{s,t} \left[ |\nabla X_{s,t}[u]|^2 -|\nabla
Y_{s,t}|^2 \right]\\&
-\frac{1}{2} \kappa_{\text{tr}} \left(
\frac{\sigma(\Psi_{s,t}[u])^2}{\Psi^2_{s,t}[u]}  -1  \right)\Gamma^{\Psi}_{s,t} -
\frac{1}{2} \kappa_{\text{tr}}\Gamma^{\Psi}_{s,t},
\end{align*}
we have by comparison that 
\begin{equs}
 \Gamma^{\Psi}_{s, t} & \geqslant \exp
\left( - (t -s)^{1-2 \varrho} \delta^{2} M - (t -s)^{1 - \varrho} 2\delta  M -(t-s) ( \|
\kappa_{\mathrm{tr}} \|_{\infty} \ve^{\prime}  /2 )\right)
\widetilde{\Gamma}^{\Psi}_{s, t} \\
& \geqslant \exp ( - 3 \mf{t}^{1 - \varrho}  \delta M -  \mf{t} \| \kappa
\|_{\infty} \ve^{\prime}/2   )\widetilde{\Gamma}^{\Psi}_{s, t} \;,
\end{equs}
where $ \widetilde{\Gamma}^{\Psi} $ is the
solution to
\begin{equs}\label{gamma:tilde}
(\partial_{t} - \Delta) \widetilde{\Gamma}^{\Psi}_{s, t} & =2\nabla
\widetilde{\Gamma}^{\Psi}_{s,t} \cdot \nabla X_{s,t}[u]+
\widetilde{\Gamma}^{\Psi}_{s, t} |\nabla Y_{s,t}|^{2}-\frac{1}{2} \kappa_{\tr}
\widetilde{\Gamma}^{\Psi}_{s,t}\;, \qquad \widetilde{\Gamma}^{\Psi}_{s, s} (x,
y)  & = \delta_{y} (x) \;.
\end{equs}
Note that Lemma~\ref{lem:pvz}
provides upper and lower bounds on the fundamental solution of the equation $(\partial_t -\Delta) \Gamma_{s,t}=2\nabla \Gamma_{s,t}\cdot \nabla Y_{s,t}$, where $Y$ solves the linear SPDE~\eqref{e:X}. These bounds can easily extended to the fundamental solutions $\Gamma^\Phi$ respectively $\Gamma^\Psi$
arising above.
Now, if we could replace in~\eqref{gamma:tilde} the drift $ \nabla X_{s, t}[u] $ with $ \nabla
Y_{s, t} $ we would have obtained the fundamental solution $ \Gamma^{\Phi} $
associated to $ \overline{\Phi} $. Comparing the fundamental solution $
\widetilde{\Gamma}^{\Psi} $ and $ \Gamma^{\Phi} $ is slightly more involved than the
treatment of the multiplicative term. Namely, we must apply the parametrix
method to obtain the required lower bound.

\emph{Parametrix applied to $ \widetilde{\Gamma}^{\Psi} $.} Let us write $
\mL $ for the time- and space-inhomogeneous operator $$ \mL =
\partial_{t} - \Delta - 2 (\nabla Y_{s,t}) \cdot \nabla - | Y_{s, t}
|^{2} + \kappa_{\mathrm{tr}}/2 \;.$$
Then $ \Gamma^{\Phi} $ is a solution to $ \mL \Gamma^{\Phi} = 0 $, while $
\widetilde{\Gamma}^{\Psi} $ is a solution to $ \mL
\widetilde{\Gamma}^{\Psi} = b \cdot \nabla \widetilde{\Gamma}^{\Psi} $, where
$  b_{s, t}  =  2(\nabla X_{s, t}[u] - \nabla Y_{s, t} )$ satisfies that $ \|
b_{s, t} \|_{\infty} \leqslant  2 \delta (t-s)^{- \varrho} $. This is
exactly the setting of the parametrix method, which aims to treat $ 
\widetilde{\Gamma}^{\Psi} $ as a perturbation (in $ \delta $ or for short times) of the equation
for $ \Gamma^{\Phi} $. For example, in first approximation, $ \widetilde{\Gamma}^{\Psi} $ can be
represented as follows (from now on let us fix $ s = 0 $ without loss of
generality and write $ \Gamma^\Phi_{t} = \Gamma^\Phi_{0, t} $ and similarly for all other
kernels):
\begin{equs}
\widetilde{\Gamma}^{\Psi}_{t} \simeq \Gamma^{(1)}_{t} =  \Gamma^\Phi_{t} + (\Gamma^\Phi \star b \cdot \nabla
\Gamma^\Phi )_{t}\;, \qquad \Gamma^\Phi \star \varphi (t, x)= \int_{0}^{t} \Gamma^\Phi_{r, t}
(x, y) \varphi (r, y) \ud r  \ud y \;.
\end{equs}
Now, $ \mL \Gamma^{(1)} = b \cdot \nabla \Gamma^{(1)} - b \cdot \nabla (\Gamma^\Phi \star b \cdot \nabla
\Gamma^\Phi ) $, so that formally $ \widetilde{\Gamma}^{\Psi} = \Gamma^{(1)} +
\mO(\delta^{2}) $. We can proceed with this expansion at will, and eventually,
formally, represent
\begin{equs}[e:parametrix]
\widetilde{\Gamma}^{\Psi}_{s, t}(x,y)  = \Gamma^\Phi_{s, t}(x,y) +
\sum_{k = 1}^{\infty} (\Gamma^\Phi \star \Xi^{(k)}_{y})_{t}(x)  \;,
\end{equs}
where 
\begin{equs}
\Xi^{(1)}_{y} (s,z)  & = b(s, z) \cdot \nabla_{z} \Gamma^\Phi_{0,s}(z, y) \;, \\
\Xi^{(k)}_{y} (s, z) & = b (s, z) \cdot \nabla_{z} (\Gamma^\Phi \star \Xi^{(k-1)
}_{y}) (s,z) \;, \qquad \forall k \geqslant 2 \;.
\end{equs}
To prove that \eqref{e:parametrix} is not only a formal identity, we must show
that the series converges. Here we use the heat kernel estimates in
Lemma~\ref{lem:pvz} which rely themselves on the parametrix method:
our aim is to obtain a good bound for the ratio
$| \Xi^{(k)}_{y} (s, z) | / | \Gamma^\Phi_{0, s} (z, y) |$. Since $\|\nabla
Y_{s,t}\|_\infty\leqslant M$ for $s\leqslant t \leqslant \tau^{\text{tot}}(s,u)\leqslant s+\mf{t}$,
we find by the second estimate in Lemma~\ref{lem:pvz} applied to $\Gamma^\Phi$
for some $ C , c >0 $ that are allowed to depend on $ \| \kappa \|_{\infty} $:
\begin{equs}
| \Xi^{(k)}_{y}(s, z) | \leqslant C \| b_{s} \|_{\infty} e^{\mf{t} C M^{2}} \int_{0}^{s}
\int_{\TT^{d} } (s-r)^{- \frac{1}{2}} p (c (s -r), z- v) | \Xi^{(k-1)}_{y} (r,
v) |
\ud v \ud r \;.
\end{equs}
If we apply this estimate to $ k =1 $, with the convention $
\Xi^{(0)}_{y} (s, z) = \delta_{y}(z) \delta_{0} (s) $, we obtain, since $ \| b_{s} \|_{\infty}
\leqslant 2 \delta s^{- \varrho} $, that
\begin{equs}
| \Xi^{(1)}_{y} (s, z)  | & \leqslant  C^{2}  \delta e^{ 2\mf{t}CM^{2}}
s^{- \varrho- \frac{1}{2}}
p (c s, z- y) \\
& \leqslant 2 C^{2} c^{-1} \delta e^{3\mf{t}CM^{2}} \, s^{-\frac{1}{2}
- \varrho} \, \Gamma^\Phi_{0, s}(z,y) \;,
\end{equs}
where the last line follows from the first bound in Lemma~\ref{lem:pvz}.
Iterating this bound, we obtain that, up to modifying the value of the constant
$ C >0 $
\begin{equs}
| \Xi^{(k)}_{y} (s, z) | \leqslant \frac{(C \delta e^{C \mf{t} M^{2}
})^{k}}{ \sqrt{k !}} s^{k (1/2 - \varrho) -1 } \Gamma^\Phi_{0 , s} (z, y) \;.
\end{equs}
For the derivation of this upper bound, we refer to the proof of
Lemma~\ref{ux}, where an analogous calculation is performed.
Hence the identity \eqref{e:parametrix} holds true and it immediately follows
that if $  \delta e^{C\mf{t} M^{2}}
\leqslant 1$ is sufficiently small (which we can rephrase as choosing the
constant $ C >1 $
sufficiently large), then we obtain $ \widetilde{\Gamma}^{\Psi}_{0, t} \geqslant \frac{4}{5}
\Gamma^\Phi_{0, t} $ for all $ t \leqslant \tau^{\mathrm{tot}}(0, u)
$. Hence for
such choice of parameters we have proven that
\begin{equs}
\Gamma^{\Psi}_{0, t} \geqslant \frac{4}{5} \exp \left(  - 3 \mf{t}^{1 - \varrho}  \delta M -  \mf{t} \| \kappa
\|_{\infty} \ve^{\prime}/2  \right) \Gamma^\Phi_{0, t}
\geqslant \frac{3}{4} \Gamma^\Phi_{0, t} \;.
\end{equs}
were the last inequality follows provided that we can choose $ \delta $
sufficiently small depending on $ M $, and $ \ve^{\prime} \in (0, 1) $
sufficiently small as well (leading to our choice of $ \overline{\ve}_{1} $).
The result is proven.
\end{proof}
Now, similarly to the case $\sigma(u)=u$, we have proven that while none of the above stopping times kick
in, we can control the process $ \Psi_{s, t} [u] $ through its linearisation at
zero. As soon as one of the stopping times kick in, we start over. This
motivates the following definition, compare with Definition~\ref{def:w-cut-off}.

\begin{definition}\label{def:flow:nl}
For any $s\geqslant 0$ and $ \mf{t} > 1 $ fix any parameters $ M > 1, \delta,
\ve_{1} \in (0, 1) $ satisfying \eqref{e:dtm} and any $ \ve \in (0, \ve_{1}) $.
Then let us define
\begin{equation*}
\begin{aligned}
\tau^{\mathrm{tot}} (s, u) = \tau(s, u) \wedge \tau^{X} (s, u) \wedge
\tau^{Y} (s, u) \wedge(s+ \mf{t}) \;.
\end{aligned}
\end{equation*}
Furthermore, set iteratively for all $ i \in \NN$:
\begin{equation*}
\begin{aligned}
\tau_{0} = 0 \;, \qquad \tau_{i+1} = \tau^{\mathrm{tot}} \left(\tau_{i},
\mT \u_{\tau_{i} -} \right) \;, \qquad \u_{t} = \Psi_{\tau_{i}, t}
[\mT \u_{\tau_{i}}] \;, \qquad \forall t \in [\tau_{i}, \tau_{i+1}) \;.
\end{aligned}
\end{equation*}
As before, recall~\eqref{e:def-cutoff}, the cut-off operator $\mathcal{T}$
is given by
\begin{equ}
\mT \varphi (x) = \min \{ \varphi (x), {\ve}  \} \;, \quad \forall
\varphi \in C(\TT^{d}; [0, \infty)) \quad  \text{ and for } \quad
{\ve} < \ve_{1}   \;.
\end{equ}
In particular, by comparison, the process $\u_{t} $ that we have
constructed satisfies the following: 
\begin{equation*}
\begin{aligned}
\| \u_{t} \|_{\infty} \leqslant \ve_{1} \;, \qquad \u_{t} \leqslant u_{t}\;, \qquad
\forall t \geqslant 0 \;.
\end{aligned}
\end{equation*}
Finally define
\begin{equs}
 w_{0} = \u_{0}\;, \qquad w_{t} = \Phi_{\tau_{i}, t}[ \u_{\tau_{i}} ] \;, \qquad \forall t \in
[\tau_{i}, \tau_{i+1})  \;.
\end{equs}
Then, since we have assumed the parameters to satisfy \eqref{e:dtm}, from Lemma \eqref{lem:comp-phi-psi} we obtain that $ u_{t} \geqslant \u_{t}
\geqslant \frac{1}{2} w_{t} $ for all $ 0\leqslant t  \leqslant \tau^{\mathrm{tot}}(0,u)
\leqslant \mf{t} $. 
\end{definition}
Let us observe that the exact value of the parameters will be chosen through
Corollary~\ref{cor:stop} and Proposition~\ref{prop:intermediate:nl}. In
particular, Proposition~\ref{prop:intermediate:nl} will determine the value of
$ \mf{t} >1 $, and Corollary~\ref{cor:stop} fixes all other parameters in such
a way that a number of rest terms and error estimates are suitably small.
We are now ready to prove Theorem~\ref{thm:moment-estimate} in full generality, building on 
results which are stated and proven in the upcoming sections.

\begin{proof}[of Theorem~\ref{thm:moment-estimate}]
First of all, it suffices to consider the process $(\u_t)_{t\geqslant 0}$
constructed in Definition~\ref{def:flow:nl} and show that, similarly to~\eqref{e:aim}
there exist $ \eta, \zeta, C_1, C_2> 0 $ such that
\begin{equation}\label{e:aim:nl}
\begin{aligned}
\EE \left[ \left(\min_{x \in \TT^{d}} \u (t, x)
\right)^{- \eta} \right] \leqslant C_{1} e^{- \zeta t} \left(\min_{x \in \TT^{d}} \u (0, x)
\right)^{- \eta} + C_{2} \;, \qquad \forall t \geqslant 0 \;.
\end{aligned}
\end{equation}
Namely, by Proposition~\ref{prop:intermediate:nl} and
Lemma~\ref{prop:unif:nl} we infer that for any
$\mf{t},M>1 $ sufficiently large, $\delta\in(0,1)$ sufficiently small satisfying \eqref{e:dtm} and $ \ve_{1}=\ve_1(\delta)$
as in Lemma~\ref{lem:comp-phi-psi}, there exist constants ${c}\in(0,1), {C}_2>0$ such that 
\begin{equation*}
\begin{aligned}
\EE_{\tau_{i}} \left[ \left(\min_{x \in \TT^{d}} \u ( \tau_{i+1}, x)
\right)^{- \eta} \right] \leqslant {c}
\left(\min_{x \in \TT^{d}} \u ( \tau_{i}, x) \right)^{- \eta}+
{C}_{2}\;.
\end{aligned}
\end{equation*}
Furthermore there exists a constant ${C}_3>0$ such that for
all $ i \in \NN $ we have
\begin{equation*}
\begin{aligned}
\EE_{\tau_{i}} \left[ \sup_{\tau_{i} \leqslant t < \tau_{i+1}} \left( \min_{x \in \TT^{d}} \u (
t, x) \right)^{- \eta  } \right] \leqslant {C}_3 \left( \min_{x \in \TT^{d}} \u (
\tau_{i}, x) \right)^{- \eta} \;.
\end{aligned}
\end{equation*}
Now let us apply Lemma~\ref{lem:discretisations}. Then, as in the proof of the
linear case $ \sigma(u) = u $, our result follows if we can show that for some $ t_{\star} > 0 $:
\begin{equs}
\sum_{i \in \NN} \PP (t_{\star} \in
[\tau_{i}, \tau_{i +1}))^{\frac{1}{2}} & < \infty \;, \label{e:cont-aim1-nl}\\
\sum_{i \in \NN} {C}_{3} {c}^{i} \PP (t_{\star} \in [\tau_{i},
\tau_{i+1}))^{\frac{1}{2}} & < 1 \;, \label{cont-aim_nl}
\end{equs}
for the stopping times $ \{\tau_i \}_{i\in \mathbb{N}}$ introduced in
Definition~\ref{def:flow:nl}. 
To this aim we must choose appropriately the time horizon $\mf{t}$ as well as
the parameters $M,\delta$ and $\varepsilon_1, \ve$ appearing in
the construction of the stopping times~\eqref{e:stop:nl}.

First of all, let us choose $ \eta = \eta_{0} $, were the latter is as in
Proposition~\ref{prop:intermediate:nl}. Next, let us fix a time horizon $ \mf{t} > 1$ such that
\begin{equs}
 C_{1} (\eta) e^{- \zeta \eta \mf{t}} \leqslant 1/2 \;,
\end{equs}
where $ C_{1} (\eta) $ is again as in Proposition~\ref{prop:intermediate:nl}.
Finally, fix $ M(\mf{t}), \delta (\mf{t}, M), \ve_{1} (\mf{t}, M, \delta), \ve
(\mf{t}, M, \delta, \ve_{1}) $ in function of $ \mf{t} $ once more as provided
in Proposition~\ref{prop:intermediate:nl}. Note that the choice of parameters
is the same as in Corollary~\ref{cor:stop}, so in particular \eqref{e:dtm} is
verified, and we are indeed in the setting of Definition~\ref{def:flow:nl}.
Before proceeding, we observe that by Proposition~\ref{prop:intermediate:nl},
the constant $ c > 0 $ appearing in \eqref{cont-aim_nl} is bounded from
above by $ C_{1}(\eta) e^{- \zeta \eta \mf{t}} $. In particular, by our choice
of $ \mf{t} $, we have that $ c \leqslant 1/2 $.

Now we are ready to verify~\eqref{cont-aim_nl}. To this aim we note that by
Lemma~\ref{prop:unif:nl} we obtain a constant ${C}_3>0$
which is bounded by ${C}_3\leqslant e^{{\alpha}\mf{t}} $ for some constant
${\alpha}>0$, uniformly in all the parameters fixed above.
Next, let us fix $ t_{\star} =
(n+1) \mf{t} $, for some $ n \in \NN $ to be fixed below. Then we have by the construction of the stopping times  in Definition~\eqref{def:flow:nl} that
$ \PP (t_{\star} \leqslant \tau_{n}) = 0 $ and hence
\begin{equs}
\sum_{i \in \NN } {C}_{3}
{c}^{i} \PP (t_{\star} \in [\tau_{i}, \tau_{i+1}))^{\frac{1}{2}}
& \leqslant \sum_{i \geqslant n}{C}_{3}{c}^{i}  \leqslant {C}_{3} 2^{-n+1}
\leqslant {K} e^{{\alpha} \mf{t}  - n b} \;,
\end{equs}
for some constant $ {K},b > 0 $. Now we choose $ n $ 
sufficiently large, depending on $ \zeta, \eta, \mf{t} $, such that $ {K}
e^{{\alpha} \mf{t}  - n b} < 1 $ as required in~\eqref{cont-aim_nl}. 

To conclude the proof, we only have to check \eqref{e:cont-aim1-nl}. This is a consequence of
the estimate
\[ \PP_{\tau_i}(\tau_{i+1} <\tau_i +\mf{t}  )\leqslant  e^{-\mf{t}}\;, \]
which in turn is a consequence of Corollary~\ref{cor:stop} (again, recall that
we have fixed all parameters but $ \mf{t} $ as prescribed in that corollary).
With the previous bound the summability of the series follows along the same
arguments of the proof of \cite[Lemma 5.2]{hairer2023spectral}. The proof is
complete.
\end{proof}

\subsection{The discrete Lyapunov property in the fully nonlinear case}

In this subsection, we collect a number of supporting result that are required in
the proof of Theorem~\ref{thm:moment-estimate} in the fully non-linear case.
Some of the proofs of these statements follow along the same lines of the proofs
in Section~\ref{sec:sup-lin}. When this is the case, we avoid repeating similar
arguments.

 We start by proving an upper bound for the nonlinear flow
$\Psi[u]$ for $u\in C(\TT^d;[0,\infty))$, up to the stopping time
$\tau^{\text{tot}}(s,u)$ appearing in Definition~\ref{def:flow:nl}. This bound
is similar to the one derived in Lemma~\ref{lem:ub} for the linear flow $\Phi$.

\begin{lemma}\label{u:phi}
For every $ \mf{t}, M >1 $, there exists a deterministic constant $C(M, \mf{t})>0$ such that
uniformly over every $u\in C(\TT^d;[0,\infty))$ the following estimate holds: 
   \begin{equation*}
   \sup\limits_{s\leqslant t \leqslant (s+\mf{t}) \wedge \tau^{X}(s, u) \wedge \tau^{Y}
(s, u)}\|\Psi_{s,t}[u]\|_\infty \leqslant C(M,
\mf{t}) \|u\|_\infty \;.
   \end{equation*}
\end{lemma}

\begin{proof}
 Using the definition of $\tau^{X}(s,u)$ and the decomposition $\Psi = e^X \overline{\Psi}$ from \eqref{phi:psi}, we obtain the following estimate
\begin{equation}\label{e:psi}
\|\Psi_{s,t}[u]\|_\infty \leqslant e^{\|X_{s,t}[u]\|_\infty } \|\overline{\Psi}_{s,t}[u]\|_{\infty} \leqslant e^{ \delta +\|Y_{s,t}\|_\infty }\|\overline{\Psi}_{s,t}[u]\|_\infty \leqslant e^{ \delta +M }\|\overline{\Psi}_{s,t}[u]\|_\infty \;,
\end{equation}  
for all $s\leqslant t\leqslant (s+\mf{t}) \wedge \tau^{X}(s, u) \wedge \tau^{Y}
(s, u)$. To simplify the notation, let us write $ \overline{\tau}$ for the stopping
time
\begin{equs}
\overline{\tau} = (s+\mf{t}) \wedge \tau^{X}(s, u) \wedge \tau^{Y} (s, u) \;.
\end{equs} 
In order to derive an upper bound for $\overline{\Psi}$ we  compare the
fundamental solution $\Gamma^\Psi$ of~\eqref{e:gamma-psi} with the fundamental
solution $\Gamma^Y$ of 
\begin{equs}
(\partial_{t} - \Delta) \Gamma^Y_{s, t} & =2\nabla \Gamma^Y_{s,t} \cdot \nabla Y_{s,t}+
\Gamma^Y_{s, t}  |\nabla Y_{s,t} |^{2}, \\
\Gamma_{s, s} (x, y)  & = \delta_{y} (x) \;.
\end{equs}
Observing that 
\begin{align*}
(\partial_t -\Delta) \Gamma^{\Psi}_{s,t}= & 2\nabla \Gamma^{\Psi}_{s,t}\cdot\nabla X_{s,t}[u] +\Gamma^{\Psi}_{s,t}|\nabla Y_{s,t}|^2  +\Gamma^{\Psi}_{s,t} [|\nabla X_{s,t}[u]|^2 -|\nabla Y_{s,t}|^2]\\& -\frac{1}{2} \kappa_{\text{tr}} \frac{\sigma(\Psi_{s,t}[u])^2}{\Psi^2_{s,t}[u]}  \Gamma^\Psi_{s,t} \;,
\end{align*}
we have
by Lemma~\ref{lem:pvz} for $s\leqslant t \leqslant \overline{\tau}$, dropping the trace term
because it is negative and since $ | \nabla X_{s, t} |^{2}- | \nabla Y
|^{2} \leqslant \delta (t-s)^{- \varrho}(M + \delta (t-s)^{- \varrho})  $ for
$ s \leqslant t \leqslant \overline{\tau} $, that
\[ \Gamma^\Psi_{s,t}(x,y) \leqslant \exp (2M\delta(t-s)^{1 - 2 \varrho} )\tilde{\Gamma}_{s,t}(x,y), \]
where $\tilde{\Gamma}$ is the fundamental solution 
of
\[ (\partial_t -\Delta)\tilde{\Gamma}_{s,t} = 2\nabla \tilde{\Gamma}_{s,t} \cdot \nabla X_{s,t}[u] + \tilde{\Gamma}_{s,t}|\nabla Y_{s,t}|^2. \]
By the parametrix method, similarly to Lemma~\ref{lem:comp-phi-psi}, one can
compare the two fundamental solutions $\tilde{\Gamma}$ and $\Gamma^Y$ and show
that there exists a constant $C(\mf{t}, M) >0$ such that $\tilde{\Gamma}_{s,t}(x,y)\leqslant C
\Gamma^Y_{s,t}(x,y)$ for all $x,y\in\TT^d$ and $s\leqslant t \leq
\overline{\tau} $. More precisely, one considers the operator
$\overline{\mathcal{L}}=\partial_t -\Delta - 2(\nabla Y_{s,t})\cdot \nabla
-|\nabla Y_{s,t}|^2$ and obtains that $\Gamma$ is a solution to
$\overline{\mathcal{L}}\Gamma^Y=0$, while $\tilde{\Gamma}$ is a solution to
$\overline{\mathcal{L}}\tilde{\Gamma}= b \cdot \nabla \tilde{\Gamma}$ for $b=2
(\nabla X_{s,t}[u] - \nabla Y_{s,t})$ which is the setting of the parametrix
method described in Lemma~\ref{lem:comp-phi-psi}.
For $s\leqslant t\leqslant \overline{\tau} $ 
\[ \Gamma^Y_{s,t}(x,y) \leqslant C(\mf{t}, M) \Gamma_{s,t}(x,y) \;, \qquad
\forall x,y\in\TT^d \;, \]
where $\Gamma$ is the fundamental solution of~\eqref{e:drift} and the
inequality holds up to potentially increasing the value of the constant $
C(M, \mf{t}) $. 
Moreover, by~\eqref{e:lbkernel} we can bound $\Gamma$ through the kernel $K$
associated to the linear flow $\Phi$ as follows
 \[  \Gamma_{s,t}(x,y) \leqslant K_{s,t}(x,y) e^{\|Y_{s,t}\|_\infty
+\frac{\|k_{\text{tr}}\|_\infty}  {2}(t-s)  }, \qquad \forall x,y\in\TT^d.
\]
Therefore, we have for $s\leqslant t \leqslant \overline{\tau}$ that overall, again up to
increasing the value of $ C(\mf{t}, M) $:
\[ \Gamma^{\Psi}_{s,t}(x,y) \leqslant C(\mf{t}, M) K_{s,t}(x,y) \;. \]
 For $u\in C(\TT^d;[0,\infty))$ we recall that $\overline{\Psi}_{s,t}[u]=\int_{\TT^d} \Gamma^{\Psi}_{s,t}(x,y)u(y)~\ud y$ and $\Phi_{s,t}[u]=\int_{\TT^d} K_{s,t}(x,y)u(y)~\ud y$.
Since $\Gamma^\Psi_{s,t}(x,y)>0~\forall x,y\in\TT^d$, the previous estimate further leads to 
 \[ \|\overline{\Psi}_{s,t}[u]\|_{\infty} \leqslant C(\mf{t}, M) \|\Phi_{s,t}[u]\|_\infty. \]
Next, from \eqref{e:psi} we deduce that for $s\leqslant t \leqslant
\overline{\tau} $ we have for $u\in C(\TT^d;[0,\infty))$ that
\[\|\Psi_{s,t}[u]\|_\infty \leqslant C(\mf{t}, M) \|\Phi_{s,t}[u]\|_\infty,  \]
and now we can use Lemma~\ref{lem:ub} to bound $\|\Phi_{s,t}[u]\|_\infty$ for
$u\in C(\TT^d;[0,\infty))$ through
\[ \sup\limits_{s\leqslant t \leqslant \overline{\tau} } \|\Phi_{s,t}[u]\|_\infty \leqslant \exp \Big(\|Y_{s,\cdot}\|_\infty + \mf{t}  \|\nabla Y_{s,\cdot} \|^2_\infty  \Big)\|u\|_\infty.   \]
In conclusion, up to once more increasing the value of the constant $ C > 0
$, we have proven that uniformly over $u\in C(\TT^d;[0,\infty))$ we can bound
\[\sup\limits_{s\leqslant t \leqslant \overline{\tau}}\|\Psi_{s,t}[u]\|_\infty \leq
C(\mf{t}, M) \|u\|_\infty.  \]
Hence the result is proven.
\end{proof}
The next set of results guarantee that the stopping times introduced in
\eqref{e:def-stnl} and used in Definition~\ref{def:flow:nl} do
not kick in too quickly.  
We start by proving that with high probability the stopping time
$\tau^Y(s,u)$ does not kick in before $\mf{t}$ . This immediately follows
choosing $M>1$ large enough.
\begin{lemma}\label{ux}
   For any $ \mf{t}> 1 $ there exists an $M(\mf{t})>1$  such that
    \[ \PP(\tau^Y(s,u)<s+\mf{t})\leqslant \frac{1}{2}  e^{-\mf{t}}. \]
\end{lemma}
\begin{proof}
The statement follows by choosing $M>1$ sufficiently large by Markov's
inequality, in view of the fact that $\EE\Big[
\|Y\|^2_{L^\infty([0,\mf{t}];\mathcal{C}^1(\TT^d))} \Big]<\infty$  as
established in Lemma~\ref{lem:X}. 
\end{proof}
An analogous statement holds for the stopping time $\tau^X(s,u)$ provided that
the parameters $\varepsilon_1,\delta\in(0,1)$ are chosen appropriately.

\begin{lemma}\label{uy}
Fix any $ \varrho \in (0, 1/4) $ and
consider the stopping times defined in \eqref{e:def-stnl} (with $
\tau^{X} $ depending on the parameter $ \varrho $).
For any $\mf{t}, M>1$ and $\delta \in (0,1)$, there exists an $\ve_1(\mf{t},M,
\delta) \in(0,1)$ such that for all $ u \in C(\TT^{d}; (0, \infty)) $
  \[
\PP \left( \tau^X(s,u)<(s+\mf{t})
\wedge \tau(s, u) \wedge \tau^{Y}(s, u) \right) \leqslant  \frac{1}{2}
e^{-\mf{t}}. \] 
\end{lemma}
\begin{proof}
    We derive an estimate for the difference between $X$ and $Y$. As usual, we
assume without loss of generality that $ s = 0 $. By Duhamel's formula we find
that for $ \gamma = f^{\prime} (0)$
    \[ Y_t =\gamma t + \int_0^t P_{t-s}~\ud
W_s \;,\]
    and similarly for $ X [u] $:
    \[ X_t[u] = \int_0^t P_{t-s} \frac{f( \Psi_{0,s}[u] ) }{ \Psi_{0,s}[u]  }  ~\ud s + \int_0^t P_{t-s}  \frac{\sigma( \Psi_{0,s}[u] ) }{ \Psi_{0,s}[u]  } ~\ud W_s.  \]
   Therefore
\begin{equs}[e:X-Y]
 X_t[u]- Y_t = \int_0^t P_{t-s}  \left(\frac{f( \Psi_{0,s}[u])} {
\Psi_{0,s}[u] }  -\gamma\right)~\ud s + \int_0^t P_{t-s} \left( \frac{\sigma(
\Psi_{0,s}[u] ) }{ \Psi_{0,s}[u]  }   -1\right)   ~\ud W_s. 
\end{equs}
Now, since both $ f $ and $ \sigma $ are $ \mC^{1} $ and $ f^{\prime}
(0) = \gamma, \sigma^{\prime} (0) =1 $, for any $ \ve^{\prime}
\in (0, 1) $ there exists an $ \ve_{1} (\ve^{\prime}) \in (0, 1) $ such that if
$ u > 0 $ and $ \| u \|_{\infty} \leqslant \ve_{1} $, then
\[  \left\| \gamma - \frac{f(u)}{u} \right\|_{\infty} \leqslant 
\ve^{\prime}  \qquad  \text{ and } \qquad \left\| 1-
\frac{\sigma(u)}{u} \right\|_{\infty} \leqslant \ve^{\prime} \;. \]
Therefore, we can estimate the first term in \eqref{e:X-Y} deterministically via Schauder
estimates:
\begin{equs}
\left\| \int_{0}^{t} P_{t-s}  \left(\frac{f( \Psi_{0,s}[u])} {
\Psi_{0,s}[u] }  -\gamma\right)\ud s \right\|_{\mC^{1}} \lesssim
\int_{0}^{t} \sqrt{t-s} \ve^{\prime}  \ud s \lesssim \sqrt{\mf{t}} \,
\ve^{\prime} \leqslant \frac{\delta}{2} \;,
\end{equs}
where the last inequality follows provided that $ \ve^{\prime} \leqslant C
\delta \mf{t}^{- \frac{1}{2}} $.
As for the second term in \eqref{e:X-Y}, to simplify the notation let us write 
\begin{equs}
\mu_{s} = \frac{\sigma( \Psi_{0,s \wedge \tau (0, u)}[u] ) }{ \Psi_{0,s \wedge
\tau(0, u)}[u]  }   -1 \;,
\end{equs}
so that $ \| \mu_{s} \|_{\infty} \leqslant \ve^{\prime}  $ for all $ s
\geqslant 0$. Then it suffices to prove that for an appropriate choice of $
\ve^{\prime} $ 
\begin{equs}
\PP \left( \sup_{0 \leqslant t \leqslant \mf{t}} \left\| \int_{0}^{t} P_{t-s}
\mu_{s} \ud W_{s} \right\|_{\mC^{1} } \geqslant \delta/2 \right) \leqslant
\frac{1}{2} e^{-  \mf{ t}} \;.
\end{equs}
Now, to treat the stochastic convolution appearing in this bound, we use the
classical ``factorisation method'', see for example \cite{DPZConvo}. Namely,
for any $ \alpha \in (0, 1) $ we can rewrite, for some constant $ c_{\alpha } > 0 $
\begin{equ}
\int_{0}^{t} P_{t-s} \mu_{s} \ud W_{s}  = c_{\alpha} \int_{0}^{t} (t -s)^{\alpha-1} P_{t-s}  \left(
\int_{0}^{s}(s-r)^{-\alpha} P_{s-r} \mu_{r} \ud W_{r} \right)\ud s \;.
\end{equ}
Since bounding the $ L^{\infty} $ norm of the convolution is simpler, we
restrict to estimating its gradient, in order to obtain a bound in $
\mC^{1} $.
In particular, for any $ \beta \in (0, 2 \alpha) $ and $ p \geqslant 1 $, and for $
W^{\beta,p} $ the fractional Sobolev space with
regularity parameter $ \beta $ and integrability parameter $ p$, we have by Schauder
estimates, for any $ i \in \{ 1, \dots, d \} $:
\begin{equs}
\left\| \partial_{x_{i}} \int_{0}^{t} P_{t-s} \mu_{s} \ud W_{s} \right\|_{W^{\beta, p}} \lesssim
\int_{0}^{t} (t- s)^{\alpha - \frac{\beta}{2} -1} \left\|
\partial_{x_{i}} \int_{0}^{s}(s-r)^{-\alpha} P_{s-r} \mu_{r} \ud W_{r} 
\right\|_{L^{p}}  \ud s \;.
\end{equs}
In particular, upon taking expectations and using the BDG inequality for the
last stochastic convolution integral, we can estimate
\begin{equs}
\EE \left\| \partial_{x_{i}} \int_{0}^{t} P_{t-s} \mu_{s} \ud W_{s}
\right\|_{W^{\beta, p}} & \lesssim
\int_{0}^{t} (t- s)^{\alpha - \frac{\beta}{2} -1} \EE\left\|
\partial_{x_{i}} \int_{0}^{s}(s-r)^{-\alpha} P_{s-r} \mu_{r} \ud W_{r} 
\right\|_{L^{p}}  \ud s \\
& \lesssim\int_{0}^{t} (t- s)^{\alpha - \frac{\beta}{2} -1}
\left( \EE\left\|\partial_{x_{i}} \int_{0}^{s}(s-r)^{-\alpha} P_{s-r} \mu_{r} \ud W_{r}
\right\|_{L^{p}}^{p} \right)^{\frac{ 1}{p}}   \ud s \\
& \lesssim\int_{0}^{t} (t- s)^{\alpha - \frac{\beta}{2} -1}
\left( \EE\left\|\partial_{x_{i}} \int_{0}^{s}(s-r)^{-\alpha} P_{s-r} \mu_{r} \ud W_{r}
\right\|_{L^{2}}^{2} \right)^{\frac{1}{2}}   \ud s \;.
\end{equs}
Now, for the last quantity, we estimate for any $ x \in \TT^{d} $, with $
p_{t}(x) $ the periodic heat kernel:
\begin{equs}
\EE & \left\vert \partial_{x_{i}} \int_{0}^{s}(s-r)^{-\alpha} P_{s-r} \mu_{r} \ud W_{r}(x)
\right\vert^{2} \\
& = \int_{0}^{s} \int_{(\TT^{d})^{2}} (s-r)^{- 2 \alpha}
\partial_{x_{i}} p_{s-r}(x-y_{1}) \partial_{x_{i}} p_{s-r}(x - y_{2}) \EE [
\mu(r, y_{1}) \mu (r, y_{2})] \kappa
(y_{1}, y_{2}) \ud y_{1} \ud y_{2} \ud r \;.
\end{equs}
Now note that if
we would naively use the uniform bound $ \| \mu \|_{\infty} \leqslant
\ve^{\prime} $ and $ \| \kappa \|_{\infty}< \infty $, we would have to bound
the heat kernels in \( L^{1} \), so that $ \| p_{s-r} \|_{L^{1}} \leqslant
(s - r)^{- \frac{1}{2}} $ and since we have two heat kernels we would end up
with a bound that is not integrable.

To circumvent this issue we use once integration by parts. In particular let us
assume that the following bound holds true:
\begin{equs}[e:bd-der-mu]
\| \partial_{x_{i}} \mu (s, x) \|_{\infty}  \leqslant C(M, \mf{t}, \varrho) s^{- \frac{1}{2}} \ve^{\prime} \;.
\end{equs}
We leave the verification of this bound to the next step of the present proof.
Using this estimate and once integration by parts, we obtain that (for some
multiplicative constants depend on $ M, \mf{t} $ and $ \varrho $) 
\begin{equs}
\EE & \left\vert \partial_{x_{i}} \int_{0}^{s}(s-r)^{-\alpha} P_{s-r} \mu_{r} \ud W_{r}(x)
\right\vert^{2} \\
& \lesssim (\ve^{\prime})^{2} \int_{0}^{s}
\int_{(\TT^{d})^{2}} (s-r)^{- 2 \alpha}
|\partial_{x_{i}} p_{s-r}(x-y_{1})| |p_{s-r}(x - y_{2})| r^{- \frac{1}{2}}
 \ud y_{1} \ud y_{2} \ud r \\
& \lesssim (\ve^{\prime})^{2}  \int_{0}^{s}
(s-r)^{- 2 \alpha} (s -r)^{- \frac{1}{2}}  r^{- \frac{1}{2}}
 \ud r \\
& \lesssim (\ve^{\prime})^{2}  s^{-2 \alpha} \;,
\end{equs}
provided that $ \alpha \in (0, 1/4) $.  We have proven that for any $ p
\geqslant 1 $ and $ 0 < \beta/2 < \alpha < 1/4 $ we can estimate
\begin{equs}
\sup_{0 \leqslant s \leqslant \mf{t}} \EE \left[ \left\| t^{- \alpha} \partial_{x_{i}}
\int_{0}^{t} P_{t-s} \mu_{s} \ud W_{s} \right\|_{W^{\beta, p}} \right] \lesssim
\ve^{\prime}  
\;.
\end{equs}
Through a standard Kolmogorov-type argument our previous calculations can be
adapted to further show
\begin{equs}[e:final-blowup]
 \EE \left[ \sup_{0 \leqslant s \leqslant \mf{t}}\left\| t^{- \alpha} \partial_{x_{i}}
\int_{0}^{t} P_{t-s} \mu_{s} \ud W_{s} \right\|_{L^{\infty}} \right] \lesssim
\ve^{\prime} \;,
\end{equs}
since we can choose $ p \geqslant 1 $ sufficiently large such that $
W^{ \beta, p} \subseteq L^{\infty} $ by Sobolev embeddings.
Now \eqref{e:final-blowup} is sufficient to prove our result trough Markov's
inequality by choosing $ \ve^{\prime} $ and therefore $ \ve(\mt{t}, M , \delta)
$ sufficiently small. The next step left in our proof is
therefore to verify \eqref{e:bd-der-mu}.

\textit{Bound on the derivative of $ \mu $.}  We find
\begin{equs}
\partial_{x_{i}} \mu_{s} = \left(\sigma^{\prime}( \Psi_{0,s \wedge \tau (0, u)}[u]
)   - \frac{\sigma( \Psi_{0,s \wedge \tau (0, u)}[u] ) }{ \Psi_{0,s \wedge
\tau(0, u)}[u]  } \right) \frac{ \partial_{x_{i}}\Psi_{0,s \wedge \tau(0, u)}[u]  }{\Psi_{0,s \wedge
\tau(0, u)}[u]  } \;.
\end{equs}
Now, the first term in the product is small (of order $ \ve^{\prime} $), while
we expect the second term to be of order one. 
In particular, we can write
\begin{equs}
\frac{\partial_{x_{i}} \Psi_{s}}{\Psi_{s}} = \partial_{x_{i}} X_{s} +
\frac{\partial_{x_{i}} \overline{\Psi}}{ \overline{\Psi} } \;,
\end{equs}
where we used the definition of $ \overline{\Psi} $ from \eqref{phi:psi}.
Hence from the definition of the stopping time $ \tau^{\mathrm{tot}} $:
\begin{equs}
\left\| \frac{\partial_{x_{i}} \Psi_{s}}{\Psi_{s}}\right\|_{\infty} \leqslant
\delta s^{- \varrho} + M + \| \partial_{x_{i}} \overline{\Psi}_{s} / \overline{\Psi}_{s} \|_{\infty}\;,
\qquad \forall s < \tau^{\mathrm{tot}} (0, u) \;.
\end{equs}
Now to bound the last term we use once more the parametrix method. Indeed, we
can represent
\begin{equs}
\overline{\Psi}_{s} [u](x) = \int_{\TT^{d}} \Gamma^{\Psi}_{0,s}(x, y) \, u(y)  \ud
y \;,
\end{equs}
where $ \Gamma^{\Psi} $ is the fundamental solution appearing already in
\eqref{e:gamma-psi}. Note that $ \Psi $ is not the flow to a linear equation:
the integral representation holds only because $ \Gamma^{\Psi} $ depends itself
on the initial condition $ u $ (we omit this to lighten the notation, and
because the estimates we obtain on $ \Gamma^{\Psi} $ do not depend on $ u
$). For clarity, we recall here the equation satisfied by $ \Gamma^{\Psi} $:
\begin{equs}
(\partial_{t} - \Delta) \Gamma^{\Psi}_{s, t} & =2\nabla \Gamma^{\Psi}_{s,t} \cdot \nabla X_{s,t}[u]+
\Gamma^{\Psi}_{s, t}  |\nabla X_{s,t}[u] |^{2}-\frac{1}{2} \kappa_{\tr}
\frac{\sigma(\Psi_{s,t}[u])^2}{\Psi^2_{s,t}[u]} \Gamma^{\Psi}_{s,t} \;.
\end{equs}
Now we will prove that for some constants $ C, c > 0 $ that depend on $ M,
\mf{t} $ and $ \varrho $, we find that for all $ t <
\tau^{\mathrm{tot}}(0, u) $:
\begin{equs}[e:heat-kernal-bds]
\partial_{x_{i}} \Gamma^{\Psi}_{0, t} (x, y) & \leqslant C t^{-
\frac{1}{2}} p_{c t} (x- y)  \;,\\
\Gamma^{\Psi}_{0, t} (x,y)& \geqslant C^{-1} p_{c t}(x-y)  \;.
\end{equs}
In this way, up to choosing a larger $
C(M, \mf{t}, \varrho) > 0
$, we have proven that
\begin{equs}
 \| \partial_{x_{i}} \overline{\Psi}_{s} / \overline{\Psi}_{s} \|_{\infty}
\leqslant C (M, \mf{t}, \varrho) s^{- \frac{1}{2}} \;.
\end{equs}
As a consequence, since $ \varrho < 1/2 $, we can overall bound
\begin{equs}
\| \partial_{x_{i}} \mu (s, x) \|_{\infty} &\leqslant C(M, \mf{t}, \varrho) s^{- \frac{1}{2}
} \left\| \sigma^{\prime}( \Psi_{0,s \wedge \tau (0, u)}[u]
)   - \frac{\sigma( \Psi_{0,s \wedge \tau (0, u)}[u] ) }{ \Psi_{0,s \wedge
\tau(0, u)}[u]  } \right\|_{\infty} 
\leqslant C(M, \mf{t}, \varrho) s^{- \frac{1}{2}} \ve^{\prime}  \;,
\end{equs}
where the last inequality follows provided that we choose $ \ve_{1}
(\ve^{\prime}) $ sufficiently small. Hence we have proven \eqref{e:bd-der-mu}.

\textit{The parametrix method for $ \Gamma^{\Psi} $.}
In the final step of this proof, we verify \eqref{e:heat-kernal-bds}.
Following somewhat the arguments in the proof of
Lemma~\ref{lem:comp-phi-psi}, we will show that 
\begin{equs}
\Gamma^{\Psi}_{0, t}(x, y) = p_{t} (x- y) + \sum_{k = 1}^{\infty}
p \star \Xi^{(k)}_{y} (t, x) \;,
\end{equs}
where $ p_{t} (x) $ is the periodic heat kernel, and
\begin{equs}
\Xi^{(1)}_{y} (s,z)  & = b(s, z) \cdot \nabla_{z} p_{s} (z -y) + d(s, z)
p_{s}(z-y) \;, \\
\Xi^{(k)}_{y} (s, z) & = b (s, z) \cdot \nabla_{z} (p \star \Xi^{(k-1)
}_{y}) (s,z)+ d(s, z)p \star \Xi^{(k)}_{y} (s, z) \;, \qquad \forall k \geqslant 2 \;.
\end{equs}
where $ b = 2 \nabla X $ and $ d = |\nabla X_{s,t}[u] |^{2}-\frac{1}{2} \kappa_{\tr}
\frac{\sigma(\Psi_{s,t}[u])^2}{\Psi^2_{s,t}[u]} $. 
The coefficients $ b, d $ satisfy the a-priori bounds
\begin{equs}
\| b_{s} \|_{\infty} \leqslant 2 (\delta s^{- \varrho}+M) \;, \quad \|
d_{s} \|_{\infty} \leqslant \| \kappa \|_{\infty} + M^{2} + \delta
s^{- \varrho}(2M + \delta s^{- \varrho} ) \;, \quad \forall 0
\leqslant s < \tau^{\mathrm{tot}}(0, u) \;.
\end{equs}
Since we are not keeping track of the parameter $ M, \mf{t} > 1 $ we can
simplify these bounds and find a constant $ C(M, \mf{t}) > 1 $ (from now on we
will omit the dependence on $ M, \mf{t} $ in the constants) such that
\begin{equs}
\| b_{s} \|_{\infty} \leqslant C s^{- \varrho} \;, \quad \|
d_{s} \|_{\infty} \leqslant C s^{- 2 \varrho} \;, \quad \forall 0
\leqslant s < \tau^{\mathrm{tot}}(0, u) \;.
\end{equs}
Now we find that for any $ k \geqslant 1 $ (with $ \Xi^{(0)}_{y} =
\delta_{0}(s) \delta_{y}(z) $ in the case $ k=1 $) we obtain for some $ C > 0 $
\begin{equs}[e:bd-xi]
| \Xi^{(k)}_{y} (s, z) | \leqslant & C    
s^{- \varrho}   \left\vert \int_{0}^{s} \int_{\TT^{d}} \partial_{z_{i}}p_{s -r}
(z-v)  \Xi^{(k-1)}_{y} (r,v) \ud v \ud r \right\vert \\
& \qquad + C s^{-2 \varrho} \int_{0}^{s} \int_{\TT^{d}}
p_{s -r} (z-v) | \Xi^{(k-1)}_{y} |(r,v) \ud v \ud r \;.
\end{equs}
Applying this estimate to $ k = 1 $ (with $ \Xi^{(0)}_{y} =
\delta_{0}(s) \delta_{y}(z) $) we deduce that for some $ C> 1 $ 
\begin{equs}
| \Xi^{(1)}_{y}|(s,z) \leqslant C s^{- \frac{1}{2} - \varrho} p_{s}(z-y) \;,
\end{equs}
where we have used that $ \varrho < 1/2  $. Integrating against the heat-kernel
we obtain the following bound:
\begin{equs}
| p \star \Xi^{(1)}_{y} | (s,z) & \leqslant \left( C \int_{0}^{s}     
 r^{- \frac{1}{2}- \varrho } \ud r \right) p_{s}(z -y) \leqslant C(1/2-
\varrho) s^{ \frac{1}{2} - \varrho} p_{s}(z -y) \;,
\end{equs}
and similarly for the derivative (allowing the value of the constant $ C > 0
$ to change from line to line):
\begin{equs}
| \partial_{z_{i}} p \star \Xi^{(1)}_{y} | (s, z) & \leqslant C
\int_{0}^{s} (s-r)^{- \frac{1}{2}} r^{- \frac{1}{2} - \varrho} \ud r
p_{s}(z-y) \leqslant C s^{- \varrho} p_{s}(z-y)\;.
\end{equs}
Now we claim that we can iterate this bound for every $ k \geqslant 0 $ to find
that for some $ C > 1 $ (uniformly over $ k $)
\begin{equs}[e:aim-induction]
| \Xi^{(k)}_{y} | (s, z) & \leqslant \frac{C^{k}}{ \sqrt{k !}} \,  s^{k(1/2 -
\varrho)- 1} \, p_{s}(z-y)\;.
\end{equs}
This claim can be verified by induction: indeed the case $ k=1 $ has already
been checked. Assume therefore that the bound holds true for $ k-1 $ and let us
prove it for $ k $.
Via \eqref{e:bd-xi} we can bound for some constant $ C^{\prime} > 1 $ whose
value may change from line to line and for $ \alpha = 1/2- \varrho $:
\begin{equs}
| \Xi^{(k)}_{y}  | (s, z) & \leqslant C^{\prime} 
 \left( s^{- \varrho} \int_{0}^{s} (s-r)^{- \frac{1}{2}} r^{(k-1) \alpha -1}\ud r +
s^{- 2 \varrho} \int_{0}^{s} r^{(k-1) \alpha -1} \ud r \right) 
\frac{C^{k-1}}{\sqrt{(k-1)!}} \, p_{s}(z-y) \;.
\end{equs}
Now the integrals in the parenthesis are estimated by
\begin{equs}
s^{\frac{1}{2}- \varrho} s^{(k-1) \alpha -1}  & \int_{0}^{1} (1-r)^{-
\frac{1}{2}} r^{(k-1) \alpha -1}\ud r +
s^{1- 2 \varrho} s^{ (k-1) \alpha -1} \int_{0}^{1} r^{(k-1) \alpha -1} \ud r \\
& \leqslant C^{\prime} s^{k\alpha  -1} \left( B (1/2, (k-1) \alpha) + 
\frac{1}{(k-1) \alpha} \right) \\
& \leqslant C^{\prime} s^{k \alpha -1} k^{- \frac{1}{2} } \;,
\end{equs}
where in the last step we used the asymptotic $ B(1/2, k) \lesssim
k^{- \frac{1}{2}} $, cf. \cite[Lemma 3.9]{PerkowskiVanZuijlen23HeatKernel}.
Therefore the estimate is proven, provided that $ C  $ is chosen
sufficiently large, with respect to $ C^{\prime} $.

Now, from \eqref{e:aim-induction} we can deduce that for some $ C^{\prime} > 1 $
\begin{equs}
| p \star \Xi^{(k)}_{y} |(s,z) & \leqslant C^{\prime} \frac{C^{k}}{ \sqrt{k !}} \,  s^{k(1/2 -
\varrho)} \, p_{s}(z-y)\;, \\
| \partial_{z_{i}} p \star \Xi^{(k)}_{t} | (s, z) & \leqslant C^{\prime} \frac{C^{k}}{ \sqrt{k !}} \,  s^{k(1/2 -
\varrho)- 1/2} \, p_{s}(z-y)\;. \label{e:grad}
\end{equs}
Since $ \sum_{k} t^{k} / \sqrt{k!} \lesssim  e^{t^{2}}$, we find that
\eqref{e:aim-induction} is proven (in the sense that the series is absolutely
convergent). 
As a consequence, we deduce that \eqref{e:heat-kernal-bds} is verified
(the upper bound follows immediately from \eqref{e:grad}, while the lower bound
follows through same steps that lead to the analogous bound in
\cite[Theorem 1.1]{PerkowskiVanZuijlen23HeatKernel}) and this completes overall the
proof of the lemma.
\end{proof}
Now we deduce from all the previous lemmata that the stopping times $ \tau_{i} $ satisfy $
\tau_{i+1} = \tau_{i} + \mf{t}$ with high probability. This is the content of
the following corollary.
\begin{corollary}\label{cor:stop}
For any $ \mf{t} > 1 $ there exists an $ M (\mf{t}) > 1 $, a $ \delta
(\mf{t}, M) \in (0, 1)  $, a $ \ve_{1} (\mf{t}, M, \delta) $ and a $
\ve(\mf{t}, M , \delta, \ve_{1}) $ such that the stopping times $ \{
\tau_{i} \}_{i \in \NN} $ from Definition~\ref{def:flow:nl} satisfy
\begin{equs}
\PP_{\tau_{i}} ( \tau_{i+1} < \tau_{i}+ \mf{t}) \leqslant e^{- \mf{t}} \;.
\end{equs}
Furthermore, the parameters $ M, \delta, \ve_{1} $ 
satisfy \eqref{e:dtm}.
\end{corollary}
\begin{proof}
This is a consequence of Lemma~\ref{u:phi}, Lemma~\ref{ux} and
Lemma~\ref{uy}. Indeed, since $ \tau_{i+1} = \tau^{\mathrm{tot}}
(\tau_{i}, \underline{u}_{\tau_{i}}) $, we find
\begin{equs}
\PP_{\tau_{i}} ( \tau_{i+1} < \tau_{i}+ \mf{t}) \leqslant & \PP_{
\tau_{i}} (\tau^{X} < \tau \wedge \tau^{Y} \wedge (\tau_{i} + \mf{t})) +
\PP_{\tau_{i}} ( \tau^{Y} < \tau \wedge \tau^{X} \wedge (\tau_{i} + \mf{t})) \\
& + \PP_{\tau_{i}} ( \tau < \tau^{X} \wedge \tau^{Y} \wedge (\tau_{i} +
\mf{t})) \\
 \leqslant & e^{- \mf{t}}\;,
\end{equs}
by applying the three quoted lemmata with $ \tau^{X} =
\tau^{X}(\tau_{i}, \underline{u}_{\tau_{i}}) $ and similarly for $
\tau^{Y} $ and $ \tau $. In particular, Lemma~\ref{u:phi} is used to identify a
value $ \ve (\mf{t}, M , \delta, \ve_{1}) $ such that $ \PP_{\tau_{i}}
(\tau < \tau^{X} \wedge \tau^{Y} \wedge (\tau_{i} +
\mf{t})) =0 $. The fact that the parameters can be chosen to satisfy
\eqref{e:dtm} follows from the fact that we can choose first $ M
(\mf{t}) $ in Lemma~\ref{uy}, and then $ \delta	(\mf{t}, M) $ in
Lemma~\ref{uy}. Since we can choose $ \delta $ small at will, we can choose it
so small that the first bound in \eqref{e:dtm} is satisfied. Then for such
choice of $ \delta $, we choose $ \ve_{1} (\mf{t}, M, \delta) \in (0, 1) $ such
that Lemma~\ref{ux} holds true and we can in addition impose that the second
bound in \eqref{e:dtm} is verified. Finally, we choose $ \ve (\mf{t}, M,
\delta, \ve_{1}) $ through Lemma~\ref{u:phi}. The result is proven. 
\end{proof}
As in the linear case $\sigma(u)=u$ treated in
Proposition~\ref{prop:intermediate}, we have the following uniform negative
moment bound on $\u$. Here we recall that by the construction specified in
Definition~\ref{def:flow:nl} we have that $ u_{t} \geqslant \u_{t} \geqslant
\frac{1}{2} w_{t} $ for all $ 0\leqslant t  \leqslant \tau^{\mathrm{tot}}(0,u) \leq
\mf{t} $. 

\begin{proposition}\label{prop:intermediate:nl}
   Under the assumptions of Theorem~\ref{thm:moment-estimate}, there exist $ {\zeta},
{\eta}_{0} > 0 $ and a constant $ C_{1} (\eta_{0}) $ 
such that the following holds.
For any $ \mf{t} > 1 $ there exists a choice of $ M( \mf{t})> 1 $, $ \delta
(\mf{t}, M) \in (0, 1) $, $ \ve_{1}(\mf{t}, M , \delta) \in (0,1) $ and $
\ve(\mf{t}, M, \delta, \ve_{1}) < \ve_{1} $ such that
for a constant ${C}_{2} ( \mf{t}, M, \delta, \ve_{1}, \ve) > 0 $ 
\begin{equation*}
\begin{aligned}
\EE_{\tau_{i}} \left[ \left(\min_{x \in \TT^{d}} \u ( \tau_{i+1}, x)
\right)^{- {\eta}} \right] \leqslant {C}_{1}({\eta}_{0}) e^{-\eta  {\zeta} \mf{t}}
\left(\min_{x \in \TT^{d}} \u ( \tau_{i}, x) \right)^{- {\eta}}+
{C}_{2}( \mf{t}, {\ve}, {\eta}_{0}) \;,
\end{aligned}
\end{equation*}
for all $i \in \NN$ and $ {\eta} \in (0, {\eta}_{0}) $.
\end{proposition}

\begin{remark}\label{rem:frakt}
At this point, the only parameter that is left free to choose is the time
horizon $ \mf{t} $. To complete the proof of Theorem~\ref{thm:moment-estimate},
we will choose $ \mf{t} $ such that $ C_{1}(\eta_{0}) e^{- \zeta
\eta_{0} \mf{t}} < 1 $.
\end{remark}

\begin{proof}
    As in Proposition~\ref{prop:intermediate} we start by bounding the jump at
time $\tau_{i+1}$. The main difference is that here we work with the nonlinear
flow $\Psi$ associated to~\eqref{e:main} and use the stopping times introduced
in Definition~\ref{def:flow:nl}. In any case, we have by definition
    \[ \u(\tau_{i+1},x)=\mathcal{T} \u(\tau_{i+1-},x). \]
Therefore
\begin{align*}
    \left( \min_{x \in \TT^{d}} \u ( \tau_{i+1}, x) \right)^{-\eta} & \leqslant
\left(\min_{x \colon \u ( \tau_{i+1}, x) \leqslant {\ve}} \u ( \tau_{i+1}, x) \right)^{-\eta} +{\ve}^{-\eta} \\
    & \leqslant \left(\min_{x \in \TT^d} \u ( \tau_{i+1}-, x) \right)^{-\eta} +{\ve}^{-\eta} \;,
\end{align*}
leading to
\begin{equation} \label{e:jumpnl}
\begin{aligned}
\EE_{\tau_{i}} \left[ \left(\min_{x \in \TT^{d}} \u ( \tau_{i+1}, x)
\right)^{- \eta} \right] \leqslant  \EE_{\tau_{i}} \left[ \left(\min_{x \in \TT^{d}} \u (
\tau_{i+1}-, x) \right)^{- \eta} \right] + {\ve}^{-\eta} \;.
\end{aligned}
\end{equation}
Now it suffices to bound the right hand-side of \eqref{e:jumpnl}.
Again, as in the proof of Proposition~\ref{prop:intermediate}, we distinguish between the events
\begin{align*}
    \tau_{i+1} = \tau_i + \mf{t}\;, \qquad \text{ and } \qquad \tau_{i+1} < \tau_i + \mf{t} \;.
\end{align*}
In the first case, we use the Lyapunov property for the linear flow in
Theorem~\ref{thm:linearised-1}. Instead, the second event has small probability
by Corollary~\ref{cor:stop}, so that it will not contribute much.
In particular, we now fix $ M, \delta, \ve_{1} , \ve $ such that
Corollary~\ref{cor:stop} holds true.

\textit{The event $ \tau_{i+1} = \tau_{i} + \mf{t} $.} We observe that for any
$ \mf{t} > 1 $, since $ M, \delta, \ve_{1} $ satisfy \eqref{e:dtm}, by
Lemma~\ref{lem:comp-phi-psi}, we have
$\Phi_{\tau_i,\tau_i+\mf{t}}[\u]\leqslant 2\Psi_{\tau_i,\tau_i+\mf{t}}[\u] $.
Therefore, we obtain \begin{equation*}
\begin{aligned}
\EE_{\tau_{i}} \left[ \left( \min_{x \in \TT^{d}}  \u (\tau_{i+1}-, x) 
\right)^{- \eta } 1_{\{ \tau_{i +1} = \tau_{i} + \mf{t} \}} \right] &\leqslant 
\EE_{\tau_{i}} \left[ \left(  \min_{x \in \TT^{d}}  \Psi_{\tau_{i}, \tau_{i} +
\mf{t}}[ \u (\tau_{i}, \cdot)] (x)
\right)^{- \eta } \right] \;\\
& \leqslant {2^\eta} \EE_{\tau_{i}} \left[ \left(  \min_{x \in \TT^{d}}  \Phi_{\tau_{i}, \tau_{i} +
\mf{t}}[ \u (\tau_{i}, \cdot)] (x)
\right)^{- \eta } \right].
\end{aligned}
\end{equation*}
Here we can use Proposition~\ref{prop:lin} for the linear flow $\Phi$ in order
to obtain
\begin{equation*}
\begin{aligned}
\EE_{\tau_{i}} \left[ \left(  \min_{x \in \TT^{d}}  \u (\tau_{i+1}-, x)
\right)^{- \eta } 1_{\{ \tau_{i +1} = \tau_{i} + \mf{t} \}} \right] & \leqslant {2^\eta}
C({\eta}_{0}) e^{- \eta {\zeta}(\gamma) \mf{t}} \left( \int_{\TT^{d}}
 \u(\tau_{i}, x)  \ud x \right)^{- \eta} \\
& \leqslant {2^\eta} C({\eta}_{0}) e^{- \eta {\zeta}(\gamma) \mf{t}} \left( \min_{x \in
\TT^{d}}  \u(\tau_{i}, x)  \right)^{- \eta}\;,
\end{aligned}
\end{equation*}
with $\gamma=f'(0)$.

\textit{The event $ \tau_{i+1} < \tau_{i} + \mf{t} $.} Here, as in the proof of
Proposition~\ref{prop:intermediate}, we use Cauchy--Schwarz to bound
\begin{equation}\label{e:small:nl}
\begin{aligned}
\EE_{\tau_{i}} & \left[  \left(  \min_{x \in \TT^{d}} \u (\tau_{i+1}-, x) 
\right)^{- \eta } 1_{\{ \tau_{i +1} < \tau_{i} + \mf{t} \}}  \right] \\
& \leqslant \EE_{\tau_{i}} \left[  \left( \min_{x \in \TT^{d}}  \u (\tau_{i+1}-, x) 
\right)^{- 2 \eta } 1_{\{ \tau_{i +1} < \tau_{i} + \mf{t} \}} 
\right]^{\frac{1}{2}} \PP_{\tau_{i}}(\tau_{i +1} < \tau_{i} +
\mf{t})^{\frac{1}{2}} .
\end{aligned}
\end{equation}
Now by Corollary~\ref{cor:stop} and in view of our choice of parameters, we
have the estimate
\begin{equation}\label{e:stop:nl}
\begin{aligned}
\PP_{\tau_{i}} (\tau_{i +1} < \tau_{i} + \mf{t}) \leqslant e^{- 2
\eta_{0} \mf{t}} \;,
\end{aligned}
\end{equation}
provided that $ \eta_{0} \leqslant 1/2 $.
Next, in order to deal with the first term in~\eqref{e:small:nl}, we show
that there exists a constant $ C> 0 $ such that
\begin{equation*}
\begin{aligned}
\EE_{\tau_{i}}\left[  \left(  \min_{x \in \TT^{d}} \u (\tau_{i+1}-, x)
\right)^{- 2 \eta } 1_{\{ \tau_{i +1} < \tau_{i} + \mf{t} \}} 
\right]^{\frac{1}{2}} \leqslant C  \left( \min_{x \in \TT^{d}} \u (\tau_{i},
x) \right)^{- \eta}  \;.
\end{aligned}
\end{equation*}
For this it is sufficient to show that for any $ \eta > 0 $
\begin{equation*}
\begin{aligned}
\EE_{\tau_{i}} & \left[ \left(  \min_{x \in \TT^{d}}  \Psi_{\tau_{i}, \tau_{i+1}}
[\u] (x)  \right)^{- \eta } 1_{\{ \tau_{i+1} < \tau_{i} + \mf{t} \}}
\right] \\
& \leqslant C(\eta) \EE_{\tau_{i}} \left[ \left(  \min_{x \in \TT^{d}}  \Phi_{\tau_{i}, \tau_{i+1}}
[\u] (x)  \right)^{- \eta } 1_{\{ \tau_{i+1} < \tau_{i} + \mf{t} \}}
\right] \leqslant \overline{C}(\eta) \;, 
\end{aligned}
\end{equation*}
for all $ \u $ such that $  \min_{x
\in \TT^{d}} \u(x) \geqslant 1 $. The first bound is a consequence of
Lemma~\ref{lem:comp-phi-psi}. The last bound was shown in
Proposition~\ref{prop:intermediate} and this completes the proof.
\end{proof}
The next step is to establish an analogous statement to Proposition~\ref{prop:unif}. This immediately follows using that $\Psi_{s,t}[\u]\geqslant \frac{1}{2}\Phi_{s,t}[\u]$ for all $s\leqslant t\leqslant \tau^{\text{tot}}(s,u)\leqslant s +\mf{t}$. 

\begin{lemma}\label{prop:unif:nl}
There exist constants ${\eta}_{0},c >0$ such that for all
$\eta\in(0,{\eta}_0)$ and $ \mf{t}> 1 $, and for any $  M>1$ and $\delta,
\ve_{1}\in(0,1)$ satisfying \eqref{e:dtm}, and for all $ i \in \NN $ we have
\begin{equation*}
\begin{aligned}
\EE_{\tau_{i}} \left[ \sup_{\tau_{i} \leqslant t < \tau_{i+1}} \left( \min_{x \in \TT^{d}} \u (
t, x) \right)^{- \eta  } \right] \leqslant  e^{{c}\mf{t}}\left( \min_{x \in \TT^{d}} \u (
\tau_{i}, x) \right)^{- \eta} \;.
\end{aligned}
\end{equation*}
\end{lemma}
\begin{proof}
Without loss of generality, let us assume that $\tau_i=0$. Then, through
Lemma~\ref{lem:comp-phi-psi}, we find that 
\begin{align*}
\min_{x\in\TT^d} \Psi_{0,t}[\u](x) \geqslant \frac{1}{2}\min_{x\in\TT^d}
\Phi_{0,t}[\u](x) \;.
\end{align*}
Hence we obtain
\begin{equs}
\EE_{\tau_{i}} \left[ \sup_{\tau_{i} \leqslant t < \tau_{i+1}} \left( \min_{x \in \TT^{d}} \u (
t, x) \right)^{- \eta  } \right] \leqslant 2^{\eta} \EE_{\tau_{i}} \left[
\sup_{\tau_{i} \leqslant t < \tau_{i+1}} \left( 
\Phi_{0,t}[\u](x) \right)^{- \eta  } \right] \;.
\end{equs}
Now the upper bound on the right hand-side follows analogously to 
Proposition~\ref{prop:unif} (note that $ 2^{\eta} $ is bounded uniformly over
$ \eta \leqslant \eta_{0} $).
\end{proof}

\subsection{Discretisations of Lyapunov functionals}

The following result provide an general procedure to construct Lyapunov
functionals from certain discretised counterparts.

\begin{lemma}\label{lem:discretisations}
Let $ (X_{t})_{t \geqslant 0} $ be a continuous-time Markov process with values
in a Polish state space $ \mX $. Assume that
$ \{ \tau_{i} \}_{i \in \NN} $ is a sequence of stopping times with
\begin{equs}
\tau_{0} = 0\;, \qquad \tau_{i} \leqslant  \tau_{i+1} \;, \qquad \lim_{i \to
\infty} \tau_{i} = \infty \;, \qquad \PP-\text{almost surely,}
\end{equs}
and define $ Y_{i} = X_{\tau_{i}} $ and set $ A_{i ,t} = \{ t \in [
\tau_{i}, \tau_{i+1}) \} $. Further, let $ F \colon \mX \to [0, \infty]
$ be a functional, and assume that the following are satisfied for some
constants $ \widetilde{c} \in (0, 1)$ and $ \widetilde{C}_{2},
\widetilde{C}_{3} > 0$:
\begin{equs}
\EE_{\tau_{i}}[ F(Y_{i +1})] & \leqslant \widetilde{c} F(Y_{i}) + \widetilde{C}_{2} \;,\\
\EE_{\tau_{i}} \left[ \sup_{ \tau_{i}  \leqslant s \leqslant \tau_{i+1} }
F(X_{s}) \right] &\leqslant \widetilde{C}_{3} F(Y_{i}) \;.
\end{equs}
Then 
\begin{equs}
\EE \left[\sqrt{F}  (X_{t}) \right] \leqslant c \sqrt{F}
 (X_{0}) + C \;,
\end{equs} 
where
\begin{equs}
c = \widetilde{C}_{3} \sum_{i
\in \NN} \PP (A_{i , t})^{\frac{1}{2}} \tilde{c}^{i} \;, \qquad C = \widetilde{C}_{2}
\widetilde{C}_{3} \frac{1}{1 - \tilde{c}} \left( \sum_{i \in \NN} \PP (A_{i ,
t})^{\frac{1}{2}} \right) \;.
\end{equs}
\end{lemma}
Note that in principle we allow for $ c , C = \infty $, and in particular we do
not claim that $ c \in (0, 1) $.

\begin{proof}
We write $ A_{i,t} = \{ t \in [\tau_{i}, \tau_{i+1})\}  $, so that by
Cauchy--Schwarz
\begin{equs}
\EE \left[ \sqrt{F} (X_{t}) \right] = \sum_{i \in \NN} \EE
 \left[ 1_{A_{i, t}} \sqrt{F} (X_{t})  \right] 
\leqslant \sum_{i \in \NN} \PP (A_{i , t})^{\frac{1}{2}} \EE
 \left[1_{A_{i, t}} F (X_{t}) \right]^{\frac{1}{2}} \;.
\end{equs}
Now we can estimate
\begin{equs}
 \EE \left[1_{A_{i, t}} F (X_{t}) \right] \leqslant \widetilde{C}_{3} \EE
\left[ F (Y_{i}) \right] \leqslant \widetilde{C}_{2} \widetilde{C}_{3} \sum_{j
= 0}^{i -1 } \tilde{c}^{j} + \widetilde{C}_{3} \tilde{c}^{i} F(X_{0}) \;.
\end{equs}
Plugging this into the previous estimate we have found that
\begin{equs}
\EE \left[ \sqrt{F} (X_{t}) \right] \leqslant F^{\frac{1}{2}} (X_{0})  \left(
\widetilde{C}_{3} \sum_{i
\in \NN} \PP (A_{i , t})^{\frac{1}{2}} \tilde{c}^{i} \right) +
\widetilde{C}_{3} \widetilde{C}_{2}
\frac{1}{1 - \tilde{c}} \left( \sum_{i \in \NN} \PP (A_{i ,
t})^{\frac{1}{2}} \right)  \;,
\end{equs}
which implies the desired result.
\end{proof}

\section{Projective dynamics and corrector estimates}\label{sec:invariant}
In this section we study the projective process $(\pi_t)_{t \geqslant 0} $,
where $ \pi_{t} = \Phi_{0, t} [w] / \| \Phi_{0, t} [w] \|_{L^{1}} $, where $ \Phi
$ is the flow in \eqref{e:Phi} and $ \pi_{0} = w / \| w \|_{L^{1}} $ an arbitrary
initial condition. This process takes values in the projective space
\begin{equs}
\mathbf{P} = \left\{ \varphi \in C ( \TT^{d}; (0, \infty))\;, \ \text{ such
that } \ \int_{\TT^{d}} \varphi(x) \ud x = 1 \right\},
\end{equs}
which is a complete metric space when endowed with Hilbert's projective metric
\begin{equs}[e:hilbert]
d_{\mathbf{P}}(\varphi, \psi) = \log{ \max_{x \in \TT^d} \left( \frac{\varphi
(x) }{\psi (x)} \right)} - \log{ \min_{x \in \TT^d} \left( \frac{\varphi
(x)}{\psi(x)} \right)} \;.
\end{equs}
The fact that $ (\mathbf{P}, d_{\mathbf{P}}) $ is complete follows for example
from the inequalities
\begin{equs}
\| \log{\varphi} - \log{\psi}\|_{\infty} \leqslant d_{\mathbf{P}}(\varphi,
\psi) \leqslant 2 \| \log{\varphi} - \log{\psi} \|_{\infty}\;.
\end{equs}
We start by observing that $ (\pi_{t})_{t \geqslant 0} $ is a Markov process.
This is a consequence of the linearity of Equation~\ref{e:Phi}. A hands-on way
to see the Markov property is by observing that $ \pi $ is itself the solution
to an SPDE (which due to a non-local term is hard to solve in general). To
obtain the SPDE representation for $ \pi_{t} $ we can proceed through It\^o's
formula, to obtain:
\begin{equs}
\partial_{t} \pi_{t} (x)& = \Delta \pi_{t} + Q( \dot{W}, \pi_{t}) \pi_{t} +
\frac{u_{t}}{r_{t}^{3}}  \ud \langle r \rangle_{t} - \frac{1}{
r_{t}^{2}} \ud \langle u(x) , r \rangle_{t} \\
& =  \Delta \pi_{t}(x) + Q( \dot{W}, \pi_{t})(x) \pi_{t}(x) +
\pi_{t}(x)  \int_{\TT^{d} \times \TT^{d}} \pi_{t}(y)
\pi_{t}(z) \kappa(y, z) \ud y \ud z\\
& \qquad \qquad \qquad \qquad \qquad \qquad \qquad - \pi_{t}(x) \int_{\TT^{d}}
\pi_{t}(y) \kappa(x, y) \ud y \;.
\end{equs}
where $Q(\dot{W},\pi_t)(x):=\dot{W}(t,x) -\int_{\TT^d} \pi_t(y) \dot{W}(t,y)\ud y$.  Denoting with
\begin{equs}
\kappa \ast \pi_{t} (x) = \int_{\TT^{d}} \pi_{t} (y) \kappa(x, y) \ud x \;,
\end{equs}
we can rewrite the equation above as
\begin{equs}[eqn:angular]
\partial_{t} \pi_{t} (x) = \Delta \pi_{t}(x) + Q( \dot{W}, \pi_{t})(x) \pi_{t}(x) -
Q( \kappa \ast \pi_{t}, \pi_{t}) (x) \pi_{t}(x) \;,
\end{equs}
so that the Markov property promptly follows.

Now, our aim in this section is to build on the ergodic properties of $
(\pi_{t})_{t \geqslant 0} $ to construct the
corrector $G$ appearing in \eqref{e:G}. To this aim, let us write
\begin{equs}
t\mapsto \mP_t\;, \qquad t \mapsto \mP_t^* \;,
\end{equs}
for respectively the Markov semigroup associated to the process $\pi_t$ and its
dual. Namely, $\mP_t \colon M_b(\mathbf{P}) \mapsto M_b(\mathbf{P})$ acts on
measurable and bounded functionals on $ \mathbf{P} $ and $\mP_t^* \colon
\mM(\mathbf{P}) \to \mM (\mathbf{P})$ acts on positive measures over $
\mathbf{P} $.
Next, let $ \mL $ denote the generator of $ \pi_{t} $, when viewed as a Markov
process on $ \mathbf{P} $. Before we proceed to the technical aspects of our
analysis of $ (\pi_{t})_{t \geqslant 0} $ as a Markov process, let us provide a
road-map of the results that we would like to obtain.
Our first aim is to solve the following Poisson equation,
and study properties of its solution:
\[ \mathcal{L}G (\pi) = \gamma - \lambda -\frac{1
}{2} \int_{\TT^d \times
   \TT^d} \pi (x) \pi (y) \kappa (x, y) \ud x \ud y =
F (\pi) \;, \qquad \mu_{\infty}(G) = 0 \;, \]
where $ \mu_{\infty} $ is the unique invariant law of $ \pi_{t} $ on $
\mathbf{P} $, whose existence we will be proven in
Subsection~\ref{subsec:projective}. This equation is solvable
since $ \gamma $ and $ \lambda $ are connected by \eqref{e:fk}, which means
that $\mu_{\infty} (F) = 0$ by construction. Therefore, provided that we can prove a spectral
gap for $ \mP_{t} $, which is the content of Subsection~\ref{subsec:was}, we will be able to define
\[ G (\pi) = \int_0^{\infty} \mathcal{P}_t F (\pi) \ud t \;, \]
which is equivalent to
\begin{equation}\label{e:def-G}
  \langle G, \mu \rangle = \int_0^{\infty} \langle F, \mathcal{P}_t^{\ast} \mu
  \rangle \ud t = \int_0^{\infty} \langle F, (\mathcal{P}_t^{\ast} \mu -
  \mu_{\infty}) \rangle \ud t\;.
\end{equation}
Then our final goal will be to provide some estimates on $ G $. In particular,
that it is uniformly bounded, meaning that $ \sup_{\pi \in
\mathbf{P}} G(\pi) < \infty  $.

To obtain these results we first study the dynamics of $ (\pi_{t})_{t \geqslant
0} $ as a random dynamical system. A version of the Krein--Rutman theorem will
allow us to deduce that $ (\pi_{t})_{t \geqslant 0} $ has a unique invariant
measure, and that sample paths tend to synchronise along this invariant
solution. In a second section we the analyse the spectral gap of $
\mP_{t} $ with respect to a cut-off Wasserstein distance in the spirit of
\cite{HairerMattinglScheutzow11}, and we show that this is
sufficient to deduce all the required bounds on $ G $: this will be the content
of Subsection~\ref{subsec:bounds-corrector}.

\subsection{Synchronisation for the projective process}\label{subsec:projective}
The aim of this subsection is to review some classical results concerning the
metric space $ (\mathbf{P}, d_{\mathbf{P}}) $ and its relation to Lyapunov
exponents via the Krein--Rutman theorem, which we state here in the form of
Birkhoff's contraction principle for strictly positive maps.
Here and in the following, from a linear operator $ A $ on $ C(\TT^{d}; [0,
\infty)) $ such that $ A \varphi = 0 \iff \varphi =0 $, we construct a
(nonlinear) operator $ A^{\pr} $ on the projective space, defined by
\begin{equs}[e:proj-op]
A^{\pr} \varphi = \frac{\varphi}{\| \varphi \|_{L^{1}}} \;.
\end{equs}
For a proof of the next result see
\cite[Section 6]{Bushell1973}.
\begin{theorem}[Birkhoff's contraction principle]\label{thm:contraction}
Let $ A $ be a bounded linear operator on $ C(\TT^{d} ;[0, \infty)) $, i.e.\ $
A \in \mL(C( \TT^{d}; [0, \infty) )) $, such that $ A $ is defined by
\begin{equs}
A ( \varphi)(x) = \int_{\TT^{d}} K(x, y) \varphi(y) \ud y \;, \qquad \forall x \in 
\TT^{d}\;,
\end{equs}
with $ K \in C (\TT^{d} \times \TT^{d}) $ and $ K (x, y) > 0 \, \forall x,y \in
\TT^{d} $. Then there exists a constant $ \tau(A) \in [0, 1) $ such that
\begin{equs}
d_{\mathbf{P}}( A^{\pr} \varphi , A^{\pr} \psi ) \leqslant \tau(A) \,
d_{\mathbf{P}}(\varphi, \psi)\;.
\end{equs}
\end{theorem}
As a corollary of this result, we obtain that the projective flow of the linear equation
\eqref{e:Phi} admits a unique invariant solution in $ \mathbf{P}  $, to which
all positive solutions synchronise with exponential speed.
\begin{corollary}\label{cor:syncrho}
Under Assumption~\ref{assu:nonlinearities}, there exists a random, adapted
initial condition $ \pi_{0}^{\infty} \in \mathbf{P} $ such that the following
hold:
\begin{enumerate}
\item \textbf{(Stationarity)} The projective process is stationary started in $ \pi_{0}^{\infty} $,
namely $ \Phi^{\pr}_{t} \pi^{\infty}_{0}  \overset{d}{=}
\pi^{\infty}_{0} $.
\item \textbf{(Synchronisation)} There exists a deterministic $ \alpha > 0 $
such that for any adapted initial condition $ \pi_{0} \in \mathbf{P}  $ the
flow $ \Phi^{\pr} \pi_{0} $ syncrhonises to the invariant solution:
\begin{equs}
\limsup_{t \to \infty} \frac{1}{t} \log{\left(d_{\mathbf{P}} (\Phi^{\pr}_{t}
\pi^{\infty}_{0}, \Phi^{\pr}_{t} \pi_{0}) \right)} \leqslant - \alpha \;.
\end{equs}
\end{enumerate}
\end{corollary}
The proof of this result is an application of Theorem~\ref{thm:contraction},
and can be found for example in \cite[Theorem 4.3]{Rosati22Synchro}.
Finally, for later use, it will be convenient to also recall the following bound of the
$ L^{1} $ distance in terms of Hilbert's metric
(see for example \cite[Theorem 4.1]{Bushell1973}).
\begin{lemma}\label{lem:bush}
Let $\pi,\nu\in \mathbf{P}$. Then we can bound
\begin{equs}
\| \pi - \nu \|_{L^1} \leqslant \exp (d_{\mathbf{P}} (\pi, \nu)) -1 \;.
\end{equs}
\end{lemma}

\begin{proof}
We have that
\begin{equs}
\int_{\TT^d} |\pi (x) - \nu (x)| \ud x  & \leqslant \max \left\{ \left( \max_{x \in \TT^d} \frac{\pi (x)}{\nu(x)}  - 1 \right)  , - \left( \min_{x \in \TT^d} \frac{\pi(x)}{\nu(x)}-1\right) \right\} \int_{\TT^d} \nu(x) \ud x \\
& \leqslant \max_{x \in \TT^d} \frac{\pi(x)}{\nu(x)}- \min_{x \in \TT^d} \frac{\pi(x)}{\nu(x)}  \leqslant \exp (d_{\mathbf{P}}(\pi, \nu)) -1 \;.
\end{equs}
\end{proof}

\subsection{Contraction semigroup in Wasserstein distance}\label{subsec:was}

In this section we extend the pathwise convergence established in
Corollary~\ref{cor:syncrho} to a spectral gap in a suitable Wasserstein
distance.
As before, let $\mathcal{P}_t$ be the semigroup associated to the Markov
process $\pi_t$ (that is a linear operator on $
\mathbf{M}_{b}(\mathbf{P}) $) and $ \mP_{t}^{*} $ its dual, acting on positive
finite measures $ \mM (\mathbf{P}) $. Now, for a generic metric space $ (\mX, d) $
we write $ \mathrm{Lip}(d) $ for the space of globally Lipschitz functions
\begin{equs}[e:def-lip]
\mathrm{Lip} (d) = \left\{ \varphi \colon \mX \to \RR  \; \colon \; \| \varphi
\|_{\mathrm{Lip} (d)} = \sup_{x, y \in \mX} \frac{| \varphi (x) - \varphi (y)
|}{d(x, y)}  < \infty \right\}\;.
\end{equs}
Then we can extend Hilbert's distance to
the Kantorovich--Rubinstein (or Wasserstein with $ p=1 $) distance on the space
of measures $ \mM(\mathbf{P}) $ by setting
\[ d_{\mathbf{P}} (\mu_1, \mu_2) = \sup_{\varphi \in \mathrm{Lip}^1 (d_{\mathbf{P}})} | \mu_1
   (\varphi) - \mu_2 (\varphi) | = \inf_{\mu \in \mathcal{C} (\mu_1, \mu_2)}
   \int_{\mathbf{P} \times \mathbf{P}} d_{\mathbf{P}} (\pi_1, \pi_2) \mu
   (\ud \pi_1, \ud \pi_2) \;, \]
where the last identity follows by duality and $
\mathrm{Lip}^{1}(d_{\mathbf{P}}) = \{ \varphi
\in \mathrm{Lip}(d_{\mathbf{P}})  \; \colon \; \| \varphi
\|_{\mathrm{Lip}(d_{\mathbf{P}})} \leqslant 1 \} $.
In this setting we can readily obtain a spectral gap for $ \mP_{t}^{*} $.
\begin{lemma}\label{lem:sg-was}
There exists a $ \zeta >0 $ such that
\begin{equs}
 d_{\mathbf{P}} (\mathcal{P}_t^{\ast} \mu_1, \mathcal{P}_t^{\ast} \mu_2)
\leqslant  e^{- \zeta
   t} d_{\mathbf{P}} (\mu_1, \mu_2)\;.
\end{equs}

\end{lemma}

\begin{proof}
If we denote with $ \Phi^{\pr}_{t} \pi_{0} = \pi_{t} $ the projective
solution map to \eqref{e:Phi}, as defined in \eqref{e:proj-op}, then we can estimate:
\[ d_{\mathbf{P}} (\mathcal{P}_t^{\ast} \mu_1, \mathcal{P}_t^{\ast} \mu_2) \leqslant
   \int_{\mathbf{P} \times \mathbf{P}} \mathbb{E} \left[ d_{\mathbf{P}}
\left( \Phi_{t}^{\pr} \pi_1, \Phi_{t}^{\pr} \pi_2 \right) \vert \mF_{0} \right]
   \mu (\ud \pi_1 , \ud \pi_2) \;, \]
for any $\mu \in \mathcal{C} (\mu_1, \mu_2)$. Now by Birkhoff's contraction
principle, Theorem~\ref{thm:contraction}, we have that
\[ d_{\mathbf{P}} (\Phi^{\pr}_t \pi_1, \Phi^{\pr}_t \pi_2) \leqslant \tau
(\Phi^{\pr}_t) \cdot d_{\mathbf{P}} (\pi_1, \pi_2), \]
with a contraction constant $ \tau (\Phi^{\pr}_{t}) $ which is independent of
the initial condition because the noise is white in time, and such that $\tau
(\Phi^{\pr}_t) < 1$ almost surely. Since $\tau (A B)
\leqslant \tau (A) \tau (B)$ and since the increments of the operator $
\Phi^{\pr}$ are
independent we find a $\zeta > 0$ such that
\[ \mathbb{E} [\tau ( \Phi^{\pr}_t)] \leqslant e^{- \zeta t} \;, \qquad \forall
t > 0 \;. \]
Hence we can promptly deduce the required bound. 
\end{proof}
In particular, we obtain that by Lemma~\ref{lem:sg-was} the invariant measure $\mu_{\infty}$
which exists by Corollary~\ref{cor:syncrho}, satisfies:
\[ d_{\mathbf{P}} (\mathcal{P}_t^{\ast} \mu, \mu_{\infty}) \leqslant e^{- \zeta
t} d_{\mathbf{P}} (\mu, \mu_{\infty}) \;. \]
While this bound is sufficient to construct the corrector $ G $ through the
identity \eqref{e:def-G}, it is not yet enough to obtain a uniform bound on $
G $. The reason for this is that $ F $ (the right hand-side in \eqref{e:def-G})
is not globally Lipschitz with respect to the metric $ d_{\mathbf{P}} $. We
therefore prove a spectral gap also in a slightly different cut-off metric,
or better, a distance-like function in the terminology of
\cite{HairerMattinglScheutzow11}. Therefore, for any $ \M > 0 $ (we will fix $
\M $ in the lemma below), let us define the distance
\begin{equs}[e:dM]
d_{\M} \colon
\mathbf{P} \times \mathbf{P} \to [0, \M] \;, \qquad d_{\M}(\cdot, \cdot) \eqdef
d_{\mathbf{P}}(\cdot, \cdot) \wedge \M  \;.
\end{equs}
Then we can go one step further and obtain that $\mP_t$ is a contraction also
under $ d_{\M} $. As the proof shows, a slightly more refined argument can show
that the contraction property holds for any $ \M > 0 $, but this lies beyond our
needs.
\begin{lemma}\label{lem:dM}
There exists an $\M>0$ such that $d_\M$ as defined in \eqref{e:dM} satisfies for some $\zeta' >0$
\begin{equs}
d_{\M} (\mathcal{P}_t^{\ast} \mu_1, \mathcal{P}_t^{\ast} \mu_2) \leqslant e^{- \zeta'
   t} d_{\M} (\mu_1, \mu_2) \;.
\end{equs}
\end{lemma}

\begin{proof}
The proof follows analogously to that of Lemma~\ref{lem:sg-was}. In particular,
it suffices to show that for some $ \alpha \in (0,1) $ we have $
d_{\M}(\mP_{1}^{\ast} \delta_{\pi_{1}}, \mP_{1}^{\ast} \delta_{\pi_{2}})
\leqslant \alpha d_{\M}(\pi_{1}, \pi_{2}) $ for all $ \pi_{1}, \pi_{2} \in
\mathbf{P} $ (where on the left we have the
Wasserstein distance, and on the right the cut-off distance on $
\mathbf{P} $).

To prove that the above bound holds, let $ K(x, y) $ be the integral kernel
associated to $ \Phi_{ 1} $. That is $ K $ is the solution to \eqref{e:Phi} with
initial condition $ w(x) = \delta_{y}(x) $, evaluated at time $ t=1 $. Then fix $ \M > 0 $ such
that the event $$ A_{\M} = \Big\{  \min_{x,y \in \TT^{d}} K(x, y) > \exp{-(\M/16)}
, \quad \max_{x,y \in \TT^{d}} K(x,y) <
\exp{(\M/16)} \Big\} $$
happens with some positive probability
\begin{equs}
\ve = \PP ( A_{\M} ) > 0\;.
\end{equs}
Note that arguing through the support theorem leads to conclude that $
P(A_{\M}) > 0 $ for all $ \M > 0 $. But to avoid unnecessary long arguments, we
skip this and simply observe that the above probability must be positive for
some $ \M > 0 $ since $ \PP ( \inf_{x, y} K (x, y) > 0\;, \; \| K \|_{\infty}
< \infty  ) = 1 $.

Then we can bound, for any $ \pi_{1}, \pi_{2} \in \mathbf{P} $ and $\zeta > 0$ as
in Lemma~\ref{lem:sg-was}:
\begin{equs}
d_{\M}(\mP_{1}^{\ast} \delta_{\pi_{1}}, \mP_{1}^{\ast} \delta_{\pi_{2}})
& \leqslant \EE d_{\M}(\Phi^{\pr}_{1} \pi_{1}, \Phi^{\pr}_{1} \pi_{2})\\
&  \leqslant \begin{cases}
e^{- \zeta} d_{\mathbf{P}}(\pi_{1}, \pi_{2}) & \text{ if } d_{\mathbf{P}}(\pi_{1},
\pi_{2}) \leqslant \M \;, \\
\EE \Big[ d_{\M}(\Phi^{\pr}_{1} \pi_{1}, \Phi^{\pr}_{1} \pi_{2}) \Big] & \text{
if } d_{\mathbf{P}}(\pi_{1}, \pi_{2}) \geqslant \M\;.
\end{cases}
\end{equs}
Now on the event $ A_{\M} $ we have, uniformly over $ \pi_{1} $:
\begin{equs}
\Phi^{\pr}_{1} \pi_{1} (x) = \bigslant{\int_{\TT^{d}} K(x,y) \pi_{1}(y) \ud
y}{\int_{\TT^{d} \times \TT^{d}} K(x, y) \pi_{1}(y) \ud x \ud y} \in
\Big[ \exp {- \M/8}, \exp{\M/8}\Big]\;,
\end{equs}
so that we can bound
\begin{equs}
d_{\mathbf{P}}(\Phi^{\pr}_{t} \pi_{1}, 1) 1_{A_{\M}} \leqslant \frac{\M}{4}\;.
\end{equs}
In particular
\begin{equs}
\EE \Big[ d_{\M}(\Phi^{\pr}_{1} \pi_{1}, \Phi^{\pr}_{1} \pi_{2}) \Big] & \leqslant \EE \Big[
(d_{\mathbf{P}}(\Phi^{\pr}_{1} \pi_{1}, 1) + d_{\mathbf{P}}(\Phi^{\pr}_{1} \pi_{2}, 1))
1_{A_{\M}} + \M 1_{A_{\M}^{c}}\Big] \\
& \leqslant \ve \frac{\M}{2} + (1 - \ve) \M = \overline{\alpha} \M\;,
\end{equs}
with $ \overline{\alpha} = \frac{\ve}{2} + (1 - \ve) \in (0, 1) $.
We have found that
\begin{equs}
d_{\M}(\mP_{1}^{\ast} \delta_{\pi_{1}}, \mP_{1}^{\ast} \delta_{\pi_{2}}) \leqslant \alpha
d_{\M}(\pi_{1}, \pi_{2})\;, \qquad \alpha = \max \{ \overline{\alpha},
\exp(- \zeta) \} \in (0, 1)\;,
\end{equs}
which is sufficient to complete the proof of the lemma.
\end{proof}

\subsection{Bounds on the corrector} \label{subsec:bounds-corrector}
%The locally Lipschitz continuity of $F$ can be obtained from the following result.
Now we are ready to rigorously introduce the corrector $ G $, formally defined
through the identity \eqref{e:def-G} and analyse its properties. The starting
point of our analysis is the following regularity estimate on the functional $
F $ defined in \eqref{e:F}. 

\begin{lemma}\label{lem:FLip}
Consider the functional $F \colon \mathbf{P} \to \RR$ as defined in
\eqref{e:F}. Then the following estimate hold:
\begin{equs}
|F(\pi) - F(\nu)| &\leqslant
\|\kappa\|_{\infty} \left\{ \exp(d_{\mathbf{P}}(\pi, \nu)) - 1\right\} \;,
\qquad & & \forall \pi, \nu \in \mathbf{P}\;, \\
| F (\pi) | & \leqslant | \lambda - \gamma | + \frac{1}{2}  \| \kappa
\|_{\infty}\;, \qquad & & \forall \pi \in \mathbf{P} \;.
\end{equs}
As a consequence, we obtain that for any $ \M > 0 $ there exists a constant $
C (\M) > 0 $ such that with $ d_{\M} $ the distance defined in \eqref{e:dM}, and
the Lipschitz norm as in \eqref{e:def-lip}:
\begin{equs}
\| F \|_{\mathrm{Lip} (d_{\M})} \leqslant C(\M) < \infty \;.
\end{equs}
\end{lemma}

\begin{proof}
The two bounds on $ F $ follow from Lemma~\ref{lem:bush}, together with the
definition of $F$. As for the Lipschitz regularity of $ F $ with respect to $
d_{\M} $, this follows, since from the first bound there exists a constant $
C (\M) > 0 $ such that
\begin{equs}
\sup_{\pi, \nu \in \mathbf{P}  \; \colon \; d(\pi, \nu) \leqslant \M}
\frac{| F(\pi) - F (\nu) |}{d_{\mathbf{P}}(\pi, \nu)} \leqslant  C_{1} (\M) <
\infty \;.
\end{equs}
Instead, if $ d_{\mathbf{P}}(\pi, \nu) \geqslant \M $, we have that $ |
F(\pi) - F (\nu)| \leqslant \| \kappa \|_{\infty} $. We obtain the result with
$ C (\M) = \max \{ C_{1} (\M), \| \kappa \|_{\infty} \} $.
\end{proof}
This result simply shows that $F$ is bounded and Lipschitz with
respect to the metric $d_{\mathbf{P}}$. From
Lemma~\ref{lem:dM}, we know that for $\M> 0$ sufficiently large,
the semigroup $\mP_t$ is a contraction in the Wasserstein distance associated
to $d_\M$.
As a consequence, the integral in \eqref{e:def-G} is well defined since for
such $\M$ and $\zeta'>0$ as in Lemma~\ref{lem:dM}, we find
\begin{equs}[e:g-defined]
\int_0^{\infty} | \langle F, (\mathcal{P}_t^{\ast} \mu - \mu_{\infty})
   \rangle | \ud t \leqslant \| F \|_{\mathrm{Lip}(d_{\M})}
\int_0^{\infty} d_{\M} (\mathcal{P}_t^{\ast} \mu, \mu_{\infty}) \ud t \leqslant  \frac{\| F 
   \|_{\mathrm{Lip}(d_{\M})}}{\zeta'} d_{\M} (\mu, \mu_{\infty}) \;. \quad
\end{equs}
As a consequence $ G $ as in \eqref{e:def-G} not only is well-defined, but it
is also Lipschitz in the $ d_{\M} $-metric.

\begin{lemma}\label{lem:local-lipschitz}
Consider the linear functional $G \in
  \mathcal{M}^{\ast}(\mathbf{P})$ given by \eqref{e:def-G} (the integral being
defined in view of \eqref{e:g-defined}). Then if we define $ G(\pi) = \langle
G, \delta_{\pi} \rangle $ for all $ \pi \in \mathbf{P} $, we find $G \in
\mathrm{Lip}_{d_\M} (\mathbf{P}; \mathbb{R})$, for $\M, \zeta' >0$ as in
Lemma~\ref{lem:dM} with
  $$\|G\|_{\mathrm{Lip}_{d_\M}(\mathbf{P};\RR)} \leqslant (\zeta^{\prime})^{-1} \|F\|_{\mathrm{Lip}_{d_\M}(\mathbf{P}; \RR)} < \infty\;.$$
\end{lemma}

\begin{proof}
From the definition of $ G $ we obtain that
\begin{equs}
| G(\pi_{1}) - G(\pi_{2}) | & \leqslant \int_{0}^{\infty} | \langle F,
\mP_{t}^{\ast} \delta_{\pi_{1}} - \mP_{t}^{*} \delta_{\pi_{2}} \rangle | \ud t \\
& \leqslant \| F \|_{\mathrm{Lip}(d_{\M})} \int_{0}^{\infty}
d_{\M}(\mP_{t}^{*} \delta_{\pi_{1}}, \mP_{t}^{*} \delta_{\pi_{2}} ) \ud t
& \leqslant \frac{\| F \|_{\mathrm{Lip}(d_{\M})}d_\M(\pi_{1},
\pi_{2})}{\zeta'},
\end{equs}
so that the result follows by Lemma~\ref{lem:FLip}.
\end{proof}
The previous result allows us to deduce that $ G $ is bounded: this is
the content of the next proposition.

\begin{proposition}\label{prop:bdd}
Consider $ G $ as in \eqref{e:def-G}. Then
\begin{equs}
\sup_{\pi \in \mathbf{P}} | G(\pi) | < \infty\;.
\end{equs}
\end{proposition}

\begin{proof}
Observe that from the definition of the distance $d_\M$, we have
\begin{equs}
 \sup_{\pi, \nu \in \mathbf{P}} | 
\varphi(\pi) -  \varphi(\nu)| \leqslant \M \| \varphi
\|_{\mathrm{Lip}(d_{\M}) } \;, \qquad \forall \varphi \in
\mathrm{Lip}(d_{\M})  \;.
\end{equs}
Then the proof follows from Lemma~\ref{lem:local-lipschitz}, if we can show
that there exists a point $\nu_0 \in \mathbf{P}$ such that $G(\nu_0)=0$. This
is the case because $\EE_{\mu_{\infty}}[G] = 0$ (with $ \mu_{\infty}$ the invariant law of $t \mapsto \pi_t$) and since $G$ is continuous and the domain $\mathbf{P}$ is connected by paths.

\end{proof}

\section{Lower bounds for stochastic flows} \label{sec:lbds}

In this section we recall basic results concerning the well-posedness of the
linear equation \eqref{e:Phi} and obtain some lower bounds on the
fundamental solution of the same equation.

Recall that our starting point is a linear equation of the following form
(without loss of generality we consider the flow from time zero $
\Phi_{t} =  \Phi_{0, t} $, since the full flow can be studied with the same
tools):
\begin{equation*}
\begin{aligned}
\partial_{t} \Phi = \Delta \Phi + \gamma \Phi +
 \Phi \ud W   \;, \qquad \Phi (0, \cdot) = w(\cdot) \in C
(\TT^{d}; [0, \infty))
\end{aligned}
\end{equation*}
with $ W $ satisfying Assumption~\ref{assu:nonlinearities} and
$\gamma \in \RR $ (in the case of \eqref{e:Phi} we have used $ \gamma = f'(0)$,
but this is irrelevant to our current discussion). 
One way to construct
solutions to this equation is through the transformation
\begin{equation} \label{e:v}
\begin{aligned}
\overline{\Phi}_{t}   = e^{- Y_{t}  } \Phi_{t} \;,
\end{aligned}
\end{equation}
where following \eqref{e:XY}, $ Y $ is the solution to the linear equation with additive noise
\begin{equation} \label{e:X}
\begin{aligned}
\ud Y = \Delta Y \ud t + \gamma  \ud t +  \ud W_t  \;,
\qquad Y(0, \cdot) = 0 \;.
\end{aligned}
\end{equation}
We observe that the process $ Y $  has an explicit
representation in terms of the heat semigroup $ (P_{t})_{t \geqslant 0} $:
\begin{equs}
Y_{t} = \gamma t +  \int_{0}^{t} P_{t-s} \ud W_{s }
\;,
\end{equs}
and furthermore, since $ Y $ is Gaussian, we can control moments of $ Y $ in
the following way.

\begin{lemma}\label{lem:X}
Let $ Y \colon [0, \infty) \times \TT^{d}  \to \RR  $ be the solution to
\eqref{e:X}. Then for every $ T > 0 $ there exist $ \alpha_{T},
C_{T} > 0 $ such that
\begin{equs}
\EE \left[ \exp \left( \alpha_{T}  \sup_{0 \leqslant s \leqslant T} \| {Y}_{s}
\|_{\mC^{1}}^{2} \right) \right] \leqslant C_{T} < \infty \;.
\end{equs}
\end{lemma}
\begin{proof}
Let us write $ p_{t}(x) $ for the periodic heat kernel, meaning that $p_t(x)=
\sum_{z \in \ZZ^{d}} (2\pi t)^{-\frac{d}{2}} e^{-|x-zj|^2/2t}$ and $
P_{t} \varphi = \smallint_{\TT^d} p_{t}(x-y) \varphi(y) \ud y $. Then we find that for any $ i, j \in \{ 1, \dots, d \} $:
\begin{equs}
\EE[  \partial_{x_{i} }{Y} (t, x) \partial_{y_{j}}  {Y} (t, y) ] & = \int_{0}^{t}
\int_{\TT^{d}} p_{t-s} (x-z_{1}) p_{t -s}(y- z_{2}) \partial_{z_{1, i}}
\partial_{z_{1, j}}\kappa (z_{1}, z_{2}) \ud
z_{1} \ud z_{2} \ud s \\
& \lesssim \int_{0}^{t} (t-s)^{- \frac{1}{2}} \| \kappa \|_{\mC^{1}} \ud s
\lesssim \sqrt{t} \| \kappa \|_{\mC^{1}} \;,
\end{equs}
where we have used once integration by parts.
Since $ \kappa \in C^{1}(\TT^d\times\TT^d)$ 
we immediately deduce that
\begin{equs}
\EE \left[ \| Y \|^{2}_{ L^{2}([0, T]; \mC^{1}(\TT^{d}))} \right] < \infty
\;.
\end{equs}
Through a standard argument using the factorisation method (see for example
\cite[Theorem 1.1]{DPZConvo}), one can similarly
deduce that the bound holds uniformly in time:
\begin{equs}
\EE \left[ \| Y \|^{p}_{ L^{\infty}([0, T]; \mC^{1}(\TT^{d}))} \right] < \infty
\;,
\end{equs}
for any $ p \geqslant 1 $.
Finally, the exponential moment bound holds by Fernique's
theorem,~\cite[Theorem 4.3]{ISEM}. 
\end{proof}

\begin{remark}\label{rem:regkappa}
To be precise, the previous theorem holds under the weaker assumption $ \kappa
\in \mC^{\ve} $ for an arbitrary $ \ve \in (0, 1) $. Since we are not interested
in this generalisation we work under the stronger assumption $ \kappa \in
\mC^{1} $.
\end{remark}
Next, we use \eqref{e:v} to rewrite \eqref{e:Phi} as an equation with random
coefficients (as opposed to having a noise that is white in time).
Namely, we compute that $ \overline{\Phi} $ satisfies the following identities:
\begin{equs}
\ud \overline{\Phi} & = -  \Phi e^{- {Y}}( (\Delta Y + \gamma )  \ud t  +  \ud
W ) + e^{-{Y}}
( (\Delta + \gamma) \Phi \ud t + \Phi \ud W  ) - \frac{1}{2}
\Phi e^{-{Y}}  \kappa_{\mathrm{tr}}  \ud t \end{equs}
where the trace $ \kappa_{\tr} (x) = \kappa (x, x) $ appears from the quadratic variation
terms in It\^o's formula, which amount to
\begin{equation*}
\begin{aligned}
\frac{1}{2} \left\{ e^{-{Y}} \Phi \ud \langle Y \rangle -2e^{-{Y}} \ud \langle {Y},
\Phi \rangle \right\} = - \frac{1}{2} \kappa_{\tr} \overline{\Phi} \ud t\;.
\end{aligned}
\end{equation*}
This further simplifies to the following PDE (we drop the differential
formulation, because the noise terms have canceled):
\begin{equs}
\partial_{t} \overline{\Phi}& = \Delta \overline{\Phi} + 2 e^{- {Y}} \nabla {Y} \cdot \nabla
\Phi - e^{- {Y}} | \nabla {Y}
|^{2} \Phi - \frac{1}{2} e^{-{Y}} \kappa_{\mathrm{tr}} \Phi \\
& = \Delta \overline{\Phi} + 2 \nabla {Y} \cdot \nabla \overline{\Phi} + 2
e^{-{Y}} | \nabla {Y} |^{2} \Phi - e^{- {Y}} | \nabla {Y}
|^{2} \Phi - \frac{1}{2}  e^{-{Y}} \kappa_{\mathrm{tr}} \Phi  \\
& = \Delta \overline{\Phi} + 2 \nabla {Y} \cdot \nabla \overline{\Phi}  +
e^{-{Y}} | \nabla {Y} |^{2} \Phi  - \frac{1}{2} 
e^{-{Y}} \kappa_{\mathrm{tr}} \overline{\Phi} \;,
\end{equs}
so that overall we have obtained the PDE with random coefficients:
\begin{equs}[e:v2]
\partial_{t} \overline{\Phi} = \Delta \overline{\Phi} + 2 \nabla {Y} \cdot \nabla \overline{\Phi} + | \nabla {Y}
|^{2} \overline{\Phi}  - \frac{1}{2} \kappa_{\mathrm{tr}}  \overline{\Phi} \;. 
\end{equs}
Now, in view of Lemma~\ref{lem:X}, we see that Equation~\eqref{e:v2} and
therefore also \eqref{e:v} admits, almost
surely, a path-wise solution for all times and initial conditions. Furthermore,
we can analogously construct the flow $ \Phi $ in 
\eqref{e:Phi}, and represent it through an integral kernel
\begin{equs}[e:kernel]
\Phi_{s, t}[w] (x) = \int_{\TT^{d}} K_{s, t}(x, y) w (y) \ud y \;.
\end{equs}
Now, we use the representation through
\eqref{e:v} to obtain lower bounds on the kernel $ K_{s, t} $
associated to the flow $ \Phi_{s, t} $ in \eqref{e:kernel}.

\begin{lemma}\label{lem:c-bd}
For every $ T > 0 $ there exists an $ \eta(T) > 0 $
such that the following holds for all $ \eta \in [0, \eta(T)] $ and for $ \Phi $
the flow associated to \eqref{e:Phi} and \( K \) the kernel in \eqref{e:kernel} associated to the
flow. There exists a
deterministic constant $ C (T) > 0 $ such that for any stopping time $ \tau
 $ and with
\begin{equation*}
\begin{aligned}
c_{K} (s, t) = \min_{x \in \TT^{d}} \int_{\TT^{d}} K_{s, t} (x, y) \ud y \;,
\end{aligned}
\end{equation*}
it holds that:
\begin{equation*}
\begin{aligned}
\EE_{\tau} \left[ \sup_{0 \leqslant t \leqslant T} c_{K} (\tau, \tau +
t)^{- \eta}   \right] \leqslant  C( T ) \;.
\end{aligned}
\end{equation*}
Furthermore, it holds that
\begin{equs}
 \EE_{\tau} \left[ \left( \min_{x, y} K_{\tau, \tau + T} (x, y) \right)^{-
\eta} +   \left( \max_{x, y} K_{\tau, \tau+ T} (x, y) \right)^{\eta} \right]
\leqslant C(T) \;.
\end{equs}
\end{lemma}

\begin{proof}
Without loss of generality, by the strong Markov property of $ u $, let us assume $ \tau=0 $.
We observe that by \eqref{e:v2}, the process $ \overline{\Phi} $ defined in \eqref{e:v} is a
super-solution to 
\begin{equation*}
\begin{aligned}
\partial_{t} \overline{\Phi} \geqslant \Delta \overline{\Phi} + 2 \nabla {Y}
\cdot \nabla \overline{\Phi} - \frac{\kappa_{\mathrm{tr}}}{2} \overline{\Phi} \;.
\end{aligned}
\end{equation*}
Therefore, the kernel $ K_{0, t} (x, y) $ is
lower bounded by
\begin{equs}[e:lbkernel]
K_{0, t} (x, y) \geqslant  e^{- \| {Y}_{t} \|_{\infty} - \frac{\| \kappa_{\tr}
\|_{\infty}}{2} t } \, \Gamma_{0, t} (x, y) \;,
\end{equs}
where $ {Y}_{t}  $ is the solution to \eqref{e:X} with initial condition $
{Y}_{0} = 0 $ and $ \Gamma $ is the fundamental solution associated to
\begin{equs}[e:drift]
\partial_{t} \Gamma  =  \Delta \Gamma  + 2 \nabla Y \cdot \nabla \Gamma \;,
\qquad \Gamma_{0,0}(\cdot, y) = \delta_{y}(\cdot) \;.
\end{equs}
Now we can lower bound the fundamental solution $ \Gamma $ through the quantitative
heat kernel estimates in \cite[Theorem 1.1]{PerkowskiVanZuijlen23HeatKernel},
which we conveniently recall below, in Lemma~\ref{lem:pvz}. In particular,
recalling that $ \overline{\Phi}_{t} =e^{-{Y_{0,t}}} \Phi_{t}$, we
deduce 
\begin{equs}\label{k}
\int_{\TT^{d}} K_{0, t}(x, y) \ud y & \geqslant e^{- \| {Y}_{t} \|_{\infty}
-\frac{ \| \kappa_{\tr}
\|_{\infty}}{2} t } \int_{\TT^{d}} \Gamma_{0, t} (x, y) \ud y \\
& \geqslant c  \exp \left(- \| {Y}_{t} \|_{\infty} - \frac{ \| \kappa_{\tr}
\|_{\infty}}{2} t - C \, t\| \nabla {Y} \|_{L^{\infty}([0, t] \times
\TT^{d})}^{2} \right) \;,
\end{equs}
with the constants $ c , C > 0  $ as in Lemma~\ref{lem:pvz}. In view of this
inequality, the claimed result is a consequence of Lemma~\ref{lem:X}.
Indeed, we can estimate the first norm $ \| Y_{t} \|_{\infty} $ by
Young's inequality to obtain
\begin{equs}
\EE \left[ \sup_{0 \leqslant t \leqslant T } c(0, t)^{- \eta} \right]
& \lesssim e^{ \frac{ \| \kappa_{\mathrm{tr}} \|_{\infty}}{2}T } \EE
\left[ \exp \left(   \eta ( CT +1) \| Y \|_{L^{\infty}([0, T];
\mC^{1}(\TT^{d}))}^{2} \right) \right] \lesssim_{T} 1
\end{equs}
where the last bound follows from Lemma~\ref{lem:X}, provided that $ \eta
\leqslant \alpha_{T}(C T +1)^{-1}$.

The proof of the second claim follows analogously by using both the upper and
the lower bound in Lemma~\ref{lem:pvz}, together with the fact that the
periodic heat kernel $ p(t, x) $ satisfies for some constant $ c(t) >0 $ that $
p(t, x) \geqslant c(t) $ for all $ x \in \TT^{d} $. 

\end{proof}
To conclude this section, we recall the quantitative heat kernel estimate
presented in \cite[Theorem 1.1]{PerkowskiVanZuijlen23HeatKernel}. We observe
that the following statement corresponds to the case $ \alpha = 0 $ in the
cited result. Although the authors assume in the statement $ \alpha > 0 $, the
same result and proof holds with $ \alpha = 0 $, by replacing the $
\mC^{\alpha} $ norm with the $ L^{\infty} $ norm.
\begin{lemma}\label{lem:pvz}
Let $ \Gamma $ be the fundamental solution to \eqref{e:drift}. Then there
exist (deterministic) constants $ c, C > 0 $ such that for every $ s \geqslant
0 $ and all $ i \in \{ 1, \dots, d \} $ and $ \mu \in \{ 0, 1 \} $
\begin{equation*}
\begin{aligned}
\Gamma_{s, t} (x, y) & \geqslant c \, p (c (t-s), x-y) \exp \left( - C
\, (t-s) \| \nabla {Y}_{s , \cdot} \|_{L^{\infty}([s, t] \times \TT^{d})}^{2} \right) \;, \\
|\partial_{x_{i}}^{\mu}  \Gamma_{s, t} (x, y) |& \leqslant  C \, (t^{-
\frac{\mu}{2}} \vee 1) p (c (t-s), x-y) \exp \left(  C
\, (t-s) \| \nabla {Y}_{s , \cdot} \|_{L^{\infty}([s, t] \times \TT^{d})}^{2} \right) \;,
\end{aligned}
\end{equation*}
where $ p $ is the fundamental solution to the periodic heat equation.
\end{lemma}
Next, we can obtain the following upper bound on the solution map $ \Phi $.
This result will be the elementary consequence of a maximum principle.
\begin{lemma}\label{lem:ub}
For any $ 0 \leqslant s \leqslant T $ and for every $ w \in C(\TT^{d}; [0,
\infty)) $ we can bound
\begin{equs}
\sup_{ s \leqslant t \leqslant T} \| \Phi_{s, t} [w] \|_{\infty}
\leqslant  \exp \left( \| {Y}_{s, \cdot} \|_{L^{\infty}([s, T] \times \TT^{d})}
+ (T -s) \|
\nabla Y_{s, \cdot}  \|^{2}_{L^{\infty}([s, T] \times \TT^{d})}  \right) \|
w \|_{\infty} \;,
\end{equs}
where $ t \mapsto {Y}_{s, t} $ is the solution to \eqref{e:X} with initial condition $
{Y}_{s} = 0 $. 
\end{lemma}
\begin{proof}
From the definition of $ \Phi $, we find that $ \| \Phi_{s,t}[w] \|_{\infty}
\leqslant e^{\| Y_{s, t} \|_{\infty} }\| \overline{\Phi}_{s, t}[w] \|_{\infty}
$. As for $ \overline{\Phi} $, we can use a maximum principle applied to
\eqref{e:v2} to obtain for any $ w \in L^{\infty} $:
\begin{equs}
\| \overline{\Phi}_{s, t}[w] \|_{\infty} \leqslant \exp \left( (t-s) \sup_{s \leqslant r
\leqslant t } \| \nabla Y_{s, r} \|_{\infty}^{2} \right) \| w \|_{\infty} \;.
\end{equs}
This concludes the proof.
\end{proof}
As a consequence of this result, we can obtain upper bounds on certain exit
times.

\begin{lemma}\label{lem:lb-st}
Fix any two parameters $ 0 \leqslant \alpha < \beta < \infty $. For any stopping time $ \sigma $, let $ w $ be and $ \mF_{\sigma} $-measurable
initial condition in $ C(\TT^{d}; [0, \infty)) $ such that $ \| w
\|_{\infty} \leqslant \alpha $. Then, if we define
\begin{equs}
\tau^{\beta}(\sigma, w) = \inf \{ t \geqslant \sigma  \; \colon \; \| \Phi_{\sigma, t}
[w] \|_{\infty} \geqslant \beta \}  \;,
\end{equs}
it holds that for any $ \zeta > 0 $ there exists a deterministic constant $
C (\alpha, \beta, \zeta) > 0 $ (but independent of $ w $) such that
\begin{equs}
\EE_{\sigma} \left[ \left(\tau^{\beta}(\sigma, w) - \sigma\right)^{- \zeta} \right] \leqslant C(\alpha, \beta,
\zeta ) \;.
\end{equs}
\end{lemma}
\begin{proof}
Let us write for short $ \tau^{\beta} = \tau^{\beta}(\sigma, w) $ and assume
that $ \sigma = 0 $ by the strong Markov property of $ \Phi_{s,t} w $. It suffices
to control, for $ t \in (0, 1) $ the probability $ \PP ( \tau^{\beta}
\leqslant  t ) $.
Here we can use Lemma~\ref{lem:ub} to find
\begin{equs}
\PP (\tau^{\beta} < t ) \leqslant \PP \left(  \sup_{0 \leqslant s
\leqslant t} \| Y_{0, s} \|_{\infty} + t
\sup_{0 \leqslant s \leqslant t} \| \nabla Y_{0, s} \|^{2}_{\infty}
\geqslant \log{(\beta / \alpha)}  \right) \;.
\end{equs}
Now, let us introduce the parabolic H\"older--Lipschitz space \(
\mC^{1}_{\mathrm{par}, t} \) of functions on $ [0, t] \times \TT^{d} $ via the norm
\begin{equs}
\| \varphi \|_{\mC^{1}_{\mathrm{par}, t}} = \sup_{0 \leqslant s < r \leqslant t
} \frac{\| \varphi_{r} - \varphi_{s} \|_{\infty}}{\sqrt{| r-s |}} +
\sup_{0 \leqslant s \leqslant t} \| \nabla \varphi_{s} \|_{\infty}\;.
\end{equs}
Then, the above probability can be upper bounded by
\begin{equs}[e:prob-bd]
\PP \left( \sqrt{t} \| Y_{0, \cdot} \|_{\mC^{1}_{\mathrm{par}, t}} +
\left( \sqrt{t} \| Y_{0, \cdot} \|_{\mC^{1}_{\mathrm{par}, t}} \right)^{2} \geqslant
\log{(\alpha / \beta)} \right) \lesssim_{\alpha, \beta, \zeta} t^{\zeta}\;,
\end{equs}
where the last inequality follows from Markov's inequality, for any $ \zeta > 0
$, since $ \EE \left[ \| Y_{0, \cdot} \|_{\mC^{1}_{\mathrm{par}, 1}}^{2 \zeta}
\right] < \infty $: this fact follows analogously to Lemma~\ref{lem:X}, by
considering in addition the time-regularity of the solution. The moment
estimate is then an immediate consequence of \eqref{e:prob-bd}.
\end{proof}

\bibliographystyle{plain}

\appendix

\end{document}